\documentclass[1 [leqno,11pt]{amsart}
\usepackage{amssymb, amsmath,amsmath,latexsym,amssymb,amsfonts,amsbsy, amsthm,color,mathrsfs,graphicx}

\numberwithin{equation}{section}

\allowdisplaybreaks[1]

\newtheorem{theorem}{Theorem}[section]
\newtheorem{lemma}[theorem]{Lemma}
\newtheorem{remark}[theorem]{Remark}

\newtheorem{proposition}[theorem]{Proposition}

\numberwithin{equation}{section}

\newcommand{\te}{\theta}
\newcommand{\al}{\alpha}
\newcommand{\tp}{\tilde{\psi}}
\newcommand{\om}{\omega}
\newcommand{\ml}{\mathcal{L}}
\newcommand{\smu}{\sqrt{\mu}}
\newcommand{\eq}{\sqrt{\frac{\e}{Q}}}
\newcommand{\qe}{\sqrt{\frac{Q}{\e}}}
\newcommand{\hh}{\hslash}
\newcommand{\bu}{\bar{u}}
\newcommand{\bv}{\bar{v}}
\newcommand{\BA}{\bar{A}_s^+}
\newcommand{\Up}{\Upsilon}
\newcommand{\UB}{\bar{U}}
\newcommand{\VB}{\bar{V}}
\newcommand{\MO}{\mathcal{O}}
\newcommand{\Bz}{\Big\z}
\newcommand{\By}{\Big\y}
\newcommand{\bz}{\big\z}
\newcommand{\by}{\big\y}
\newcommand{\TT}{\tilde{\mathcal{T}}}

\newcommand{\dd}{\mathrm{d}}

\newcommand{\e}{\varepsilon}
\newcommand{\T}{\mathcal{T}}
\newcommand{\X}{\mathbb{X}}
\newcommand{\Y}{\mathbb{Y}}

\newcommand{\f}{\mathcal{F}}

\newcommand{\z}{\langle}
\newcommand{\y}{\rangle}
\newcommand{\p}{\partial}

\begin{document}

\title[Prandtl Boundary Layers in An Infinitely Long Convergent Channel]
{Prandtl Boundary Layers in An Infinitely Long Convergent Channel}

\author{Chen GAO}
\address{The Institute of Mathematical Sciences and Department of Mathematics, The Chinese University of Hong Kong, Shatin, NT, Hong Kong}
\email{gaochen@amss.ac.cn}

\author{Zhouping Xin}
\address{The Institute of Mathematical Sciences and Department of Mathematics, The Chinese University of Hong Kong, Shatin, NT, Hong Kong}
\email{zpxin@ims.cuhk.edu.hk}

\begin{abstract}
This paper concerns the large Reynold number limits and asymptotic behaviors of solutions to the 2D steady Navier-Stokes equations in an infinitely long convergent channel. It is shown that for a general convergent infinitely long nozzle whose boundary curves satisfy curvature-decreasing and any given finite negative mass flux, the Prandtl's viscous boundary layer theory holds in the sense that there exists a Navier-Stokes flow with no-slip boundary condition for small viscosity, which is approximated uniformly by the superposition of an Euler flow and a Prandtl flow. Moreover, the singular asymptotic behaviors of the solution to the Navier-Stokes equations near the vertex of the nozzle and at infinity are determined by the given mass flux, which is also important for the constructions of the Prandtl approximation solution due to the possible singularities at the vertex and non-compactness of the nozzle. One of the key ingredients in our analysis is that the curvature-decreasing condition on boundary curves of the convergent nozzle ensures that the limiting inviscid flow is pressure favourable and plays crucial roles in both the Prandtl expansion and the stability analysis.
\end{abstract}
\maketitle

\section{Introduction}\label{sec:intro}
Consider the steady Navier-Stokes equations
\begin{equation}\label{NSE}
\left\{
\begin{aligned}
&u u_r +\frac{vu_{\te}}{r}-\frac{v^2}{r}+p_r-\e\Big(u_{rr}+\frac{u_r}{r}+\frac{u_{\te\te}}{r^2}-\frac{2v_{\te}}{r^2}-\frac{u}{r^2}\Big)=0,\\
&u v_r +\frac{vv_{\te}}{r}+\frac{uv}{r}+\frac{p_{\te}}{r}-\e\Big(v_{rr}+\frac{v_r}{r}+\frac{v_{\te\te}}{r^2}+\frac{2u_{\te}}{r^2}-\frac{v}{r^2}\Big)=0,\\
&u_r+\frac{u}{r}+v_{\te}=0
\end{aligned}
\right.
\end{equation}
in a two dimensional symmetric domain $\Omega=\big\{(r,\te)\big| r>0,\big|\te|<g(r)\big\}$. Here $u(r,\te)$ and $v(r,\te)$ are the radial and angular component of the velocity in polar coordinates respectively. We assume that $g$ is a smooth function satisfying $g(0)=\al$, $g>0$, and $\underset{{r\rightarrow\infty}}{\lim}g(r)=\beta$ with $0<\al\leqslant\beta\leqslant\frac{\pi}{2}$. 

The velocity satisfies the no-slip boundary condition 
\begin{align}\label{BC}
[u,v]\big|_{\Gamma^{\pm}}=0,
\end{align}
where $\Gamma^{\pm}=\big\{(r,\te)\big| r>0,\te=\pm g(r)\big\}$. The velocity is allowed to be singular at the vertex. The mass flux of the flow is prescribed as a negative constant, i.e.
\begin{align}\label{FX}
\int_{\mathcal{C}_{\rho}}u\dd s=-Q,\quad Q>0,
\end{align}
where $\mathcal{C}_{\rho}=\big\{(r,\te)\big|r=\rho,|\te|<g(r)\big\}$. Due to the divergence-free condition, $\int_{\mathcal{C}_{\rho}}u\dd s$ is independent of $\rho$, so
\begin{align}\label{FXF}
\underset{{\rho\rightarrow \infty}}{\lim}\int_{\mathcal{C}_{\rho}}u\dd s=\underset{{\rho\rightarrow0}}{\lim}\int_{\mathcal{C}_{\rho}}u\dd s=-Q.
\end{align}
It will be shown later that the asymptotic behaviors of the solution at the vertex of the nozzle and infinity can be determined by the given mass flux. It is an extremely important problem both physically and theoretically to study the asymptotic behavior as the viscosity approaching zero of the solution to (\ref{NSE}) with the non-slip boundary condition (\ref{BC}) and the flux condition (\ref{FX}) on $\Omega$.

\begin{figure}[h]
\centering
\includegraphics[scale=0.6]{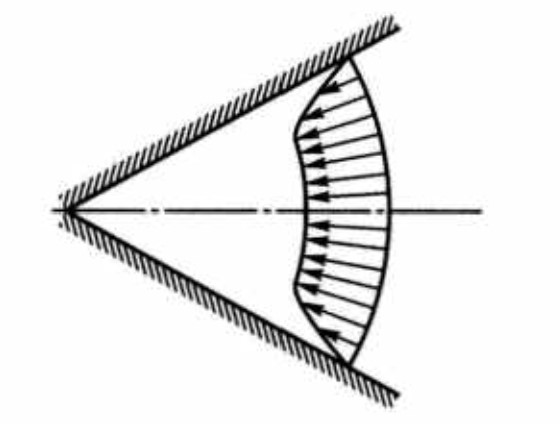}
\caption{Velocity profile in $\Omega_\beta$.}
\label{pic}
\end{figure}

In the special case that $\Gamma^{\pm}$ are two rays, $\Omega_\beta:=\big\{(r,\te)\big||\te|<\beta\big\}$, there have been many physics literature, see \cite{LL} and \cite{SG}. The solution to (\ref{NSE})-(\ref{FX}) is given as
\begin{align}\label{similar}
\left\{
\begin{aligned}
u=&-\frac{Q}{2\beta r}-\frac{Q}{2\beta r}\Big[f'\big(\sqrt{\frac{Q}{2\beta\e}}(\te+\beta)\big)-1\Big]\\
&-\frac{Q}{2\beta r}\Big[f'\big(\sqrt{\frac{Q}{2\beta\e}}(\beta-\te)\big)-1\Big]+\mathcal{O}\Big(\frac{\sqrt{\e Q}}{r}\Big),\\
v=&0,
\end{aligned}
\right.
\end{align}
as $\e\rightarrow 0$. Here $[-\frac{Q}{2\beta r},0]$ is an irrotational Euler flow called a sink. $[-\frac{Q}{2\beta r}f'\big(\sqrt{\frac{Q}{2\beta\e}}(\te+\beta)\big),0]$ is some kind of self-similar solution to the Prandtl equations, and $f$ satisfies
\begin{equation}\label{FSE}
\left\{
\begin{aligned}
&f'''-(f')^2+1=0,\\
&f'(0)=f''(+\infty)=0, \\
&f'(+\infty)=1.
\end{aligned}
\right.
\end{equation}
(\ref{similar}) shows that the solution behaves as the Euler flow in the region away from the boundaries, while as the Prandtl flow close to the boundaries. Thus, the Prandtl's viscous boundary layer theory holds true in this special case. One of the main results of this paper is to show that the similar theory holds true for general convergent nozzles.

Due to the geometry of an infinitely long convergent nozzle, motivated by (\ref{similar}), we need to understand more precisely the asymptotic behaviors of solution to (\ref{NSE})-(\ref{FX}) as $r\rightarrow+\infty$ and $r\rightarrow0$ for general $\Omega$, which can be done formally as follows. Since the arc length of $\mathcal{C}_\rho$ equals to $2\beta \rho$ as $\rho$ goes to $+\infty$ ($2\al\rho$ as $\rho$ goes to 0), the flux condition (\ref{FXF}) indicates the following asymptotic behaviors of the solution:
$$[ru,rv]=\mathcal{O}\big(1\big)\quad\text{ as $r\rightarrow+\infty$ and $r\rightarrow0$.}$$
More precisely, if there exists $[A^{\infty}(\te), B^{\infty}(\te)]\in C^\infty[-\beta,\beta]$ such that
\begin{align}\label{asymb}
\begin{aligned}
&r^k\p_r^k\p_\te^l[ru(r,\te)-A^\infty(\te)]=\mathit{o}(1),\\
&r^k\p_r^k\p_\te^l[rv(r,\te)-B^\infty(\te)]=\mathit{o}(1), \quad \text{ as $r\rightarrow+\infty$,}
\end{aligned}
\end{align}
where $k$ and $l$ are non-negative integers, then (\ref{BC}) implies:
$$A^\infty(-\beta)=A^\infty(\beta)=B^\infty(-\beta)=B^\infty(\beta)=0.$$
It follows from the divergence-free condition and (\ref{asymb}) that
\[
\begin{split}
\p_\te B^\infty&=\p_\te[B^\infty-rv]+r\p_\te v=\p_\te[B^\infty-rv]-r\p_r[ru]\\
               &=\p_\te[B^\infty-rv]-r\p_r[ru-A^\infty]=\mathit{o}(1),\quad\text{as $r\rightarrow +\infty$}.
\end{split}
\]
This, together with $B^\infty(\beta)=0$, implies $B^\infty\equiv 0$, namely, $\underset{r\rightarrow \infty}{\lim}rv=0$. Let $\Phi$ be a streamfunction of $[u,v]$, i.e. $\p_\te\Phi=-ru$, $r\p_r\Phi=rv$. Then 
\[
\underset{r\rightarrow \infty}{\lim}[\p_\te\Phi,r\p_r\Phi]=[-A^\infty,0].
\]
Since $\Phi$ satisfies the Navier-Stokes equations in the following form:
\begin{align}\label{NSPHI}
r^2\Phi_\te\Delta(r\Phi_r)-r^2(r\Phi_r)\Delta\Phi_\te-2r^2\Phi_\te\Delta\Phi+\e r^4\Delta^2\Phi=0,
\end{align}
evaluating the above equation as $r\rightarrow\infty$ shows that $A^\infty$ solves
\begin{equation}\label{INBE}
\left\{
\begin{aligned}
&-2A^\infty\frac{\dd A^\infty}{\dd \te}-\e\Big(\frac{\dd^3 A^\infty}{\dd\te^3}+4\frac{\dd A^\infty}{\dd \te}\Big)=0,\\
&A^\infty(-\beta)=A^\infty(\beta)=0, \\
&\int_{-\beta}^\beta A^\infty(\te)\dd\te=-Q.
\end{aligned}
\right.
\end{equation}
The last integration condition in (\ref{INBE}) follows from the flux condition (\ref{FXF}). It can be shown that if $\e/Q$ is small enough, there is only one solution to (\ref{INBE}) which is non-positive. So the mass flux $-Q$ determines the asymptotic behavior at $+\infty$ completely, i.e.
\[
\underset{r\rightarrow \infty}{\lim}[ru,rv]=[A^\infty,0],
\]
where $A^\infty$ solves system (\ref{INBE}). Moreover, $\big[\frac{A^\infty(\te)}{r},0\big]$ is the solution to (\ref{NSE}) with the boundary condition (\ref{BC}) and the flux condition (\ref{FX}) when $\Omega=\Omega_\beta$, and $\big[\frac{A^\infty(\te)}{r},0\big]$ satisfies (\ref{similar}) for small $\e$. Similarly, $\underset{r\rightarrow 0}{\lim}[ru,rv]=[A^0(\te),0]$, $A^0$ solves the following problem:
\begin{equation}\label{ZEBE}
\left\{
\begin{aligned}
&-2A^0\frac{\dd A^0}{\dd \te}-\e\Big(\frac{\dd^3 A^0}{\dd\te^3}+4\frac{\dd A^0}{\dd \te}\Big)=0,\\
&A^0(-\al)=A^0(\al)=0, \\
&\int_{-\al}^\al A^0(\te)\dd\te=-Q.
\end{aligned}
\right.
\end{equation}
$\big[\frac{A^0(\te)}{r},0\big]$ satisfies (\ref{similar}) if all $\beta$ in (\ref{similar}) is replaced with $\al$.

In this paper, we will show rigorously that if the curvature of $\Gamma^{\pm}$ monotonically decreases concerning $r$, the flux of flow is negative, then for small enough $\e$, the Prandtl boundary layer expansion is valid for the Navier-Stokes flows. Moreover, the singular asymptotic behaviors of the viscous flow at the infinity and the origin are determined by the systems (\ref{INBE}) and (\ref{ZEBE}).

Rigorous justification of the Prandtl's strong boundary layer expansion for the non-slip boundary condition has always been an important challenging problem. There are some important results for steady flows in recent years. Guo and Nguyen first studied the case of a moving plate in \cite{GN}. They considered a Navier-Stokes flow in $(0,L)\times(0,\infty)$ with its horizontal component on the boundary $y=0$ equal to a constant $u_b>0$, and there is a mismatch between $u_b$ and the horizontal component of a shear Euler flow on $y=0$. They showed the validity of the Prandtl expansion when posing some suitable boundary conditions of the remainder. Then Iyer extended this result to a Navier-Stokes flow over a rotating disk in the case the Euler flow is a rotating shear flow in \cite{Iy17}. The case that the Euler flow is a perturbation of a shear flow is considered by Iyer in \cite{Iy19}. In \cite{GI}, Guo and Iyer justified the Prandtl boundary layer expansion in $(0,L)\times(0,\infty)$ for the non-slip boundary condition, i.e. the velocity of the viscous flow equals to $(0,0)$ on $y=0$, with the inviscid Euler flow being a shear flow. And meanwhile, Gerard-Varet and Maekawa in \cite{GM} showed a stability result of the forced steady Navier-Stokes equations for some shear Prandtl type profile in $x$-periodic domain. Later, Gao and Zhang in \cite{GZ21} generalized the result in \cite{GI} to the case of non-shear Euler flows. However, all the mentioned results, \cite{GZ21}-\cite{Iy19}, are considered in a narrow domain, i.e. $L\ll 1$ or the period in $x$ direction is small. For possible large $x$, Iyer studied the case over a moving plate in $(0,\infty)\times(0,\infty)$ in \cite{Iy20}, the horizontal component of the viscous flow on $y=0$ equals to $1-\delta$, $0<\delta\ll1$, and the inviscid flow is the Shear flow $[1,0]$. In the case of no-slip boundary condition, Gao and Zhang in \cite{GZ20} got a result in the domain $(0,L)\times(0,\infty)$ for any positive constant $L$, where the Euler flow is a shear flow and the Prandtl profile is concave in $y$ direction. After that, Iyer and Masmoudi studied the case $L=\infty$ for the shear Euler flow $[1,0]$ and the self-similar Prantdl flow called Blasius flow in \cite{IM20}. And Chen, Wu and Zhang in \cite{CWZ} obtained a stability result for the Prandtl type profile in $x$-periodic domain with a large period under some spectral condition. Fei, Gao, Lin and Tao in \cite{FGLT1} and \cite{FGLT2} considered the problems in disk and annulus. The boundary conditions they are concerned with are the velocity slightly differs from the rigid-rotation. The selection principle of the Euler flow in the zero-viscosity limit seems interesting. Thanks to the Prandtl-Batchelor theory, the vorticity of the Euler flow must be a constant. And this constant is determined by the wood formula.     

For an infinitely long convergent nozzle, the vanishing viscosity problem seems more subtle due to the possible singular asymptotic behaviors of viscous flows as indicated in the special case (\ref{similar}). Besides the usual non-slip boundary condition (\ref{BC}) on the boundaries $\Gamma^\pm$ for the viscous flow, one needs to specify the asymptotic behaviors of viscous solution upstream and downstream at the vertex of the nozzle, which turn to be determined completely by prescribing the max flux as in (\ref{FX}). Indeed, due to the divergence-free conditions and no-penetration condition, (\ref{FX}) implies (\ref{FXF}). Then a formal asymptotic analysis as $r\rightarrow\infty$ and $0$ shows the solution must satisfy the asymptotic behaviors $\underset{r\rightarrow \infty}{\lim}[ru,rv]=[A^\infty,0]$ and $\underset{r\rightarrow 0}{\lim}[ru,rv]=[A^0,0]$. Here $A^\infty$ and $A^0$ are solutions of systems (\ref{INBE}) and (\ref{ZEBE}) with the flux $-Q$, respectively. Thus the mass flux determines the asymptotic behaviors at $r\rightarrow\infty$ and $0$ for the viscous flow.

In addition to the usual difficulties in the analysis for such kind inviscid limits problems, such as singular limits with large sharp changes of gradients near the boundaries due to the possible creation of vorticity, there are at least two substantial new difficulties arising for the case of flows in the infinitely long convergent nozzle in both the constructions of the Prandtl's approximation and the stability analysis. First, since the solution may be singular at the origin and vanishes at infinity (see (\ref{similar})), one should construct an approximate solution with suitable asymptotic behaviors and estimate the remainder in an appropriate space. Second, for the general domain $\Omega$, the inviscid flow is non-shear, which makes the stability analysis of the linearized Navier-Stokes equations much more difficult.
 
To be precise, we write the boundary curves in the Cartesian coordinate system with arc length parameter as
\begin{align}\nonumber
\Gamma^\pm=\Big\{\big(x(s),\pm y(s)\big)\in\mathbb{R}^2\Big|x(s)=r(s)cos\big(g\big(r(s)\big)\big),\hspace{1mm}y(s)=r(s)sin\big(g\big(r(s)\big)\big),\hspace{1mm} s\geqslant0\Big\},
\end{align}
with $\overset{.}{x}^2(s)+\overset{.}{y}^2(s)=1$. $\big(x(s),y(s)\big)$ satisfies
\begin{align}\label{CUR}
\begin{aligned}
&\big(x(0),y(0)\big)=(0,0),\quad (\overset{.}{x}(0),\overset{.}{y}(0))=(\cos\al,\sin\al),\\
&\big(\overset{.}{x}-\cos\beta,\overset{.}{y}-\sin\beta\big)\text{ and its derivatives decay fast to 0 as $s\rightarrow\infty$},
\end{aligned}
\end{align}
where $0<\al\leqslant\beta\leqslant\frac{\pi}{2}$. Set $\kappa(s):=|\frac{\dd}{\dd s}\arctan\big(\frac{\overset{.}{y}(s)}{\overset{.}{x}(s)}\big)|=|\overset{..}{y}(s)\overset{.}{x}(s)-\overset{..}{x}(s)\overset{.}{y}(s)|$ to be the curvature of the $\Gamma^{+}$. We assume that the curvature in the arc length satisfies:
\begin{align}\label{CURC}
\frac{\dd^2}{\dd s^2}\arctan\Big(\frac{\overset{.}{y}(s)}{\overset{.}{x}(s)}\Big)\leqslant0,\text{ for any $s\geqslant0$.}
\end{align}
It follows from (\ref{CUR}) and (\ref{CURC}) that $\kappa=\frac{\dd}{\dd s}\arctan\big(\frac{\overset{.}{y}}{\overset{.}{x}}\big)\geqslant0$, hence the curve $\Gamma^+$ is convex. Therefore (\ref{CURC}) implies that the curvature decreases as $s$ increases, namely, as the arc length increases, the curve becomes flatter. See Figure \ref{cpic} below.
\begin{figure}[h]
\centering
\includegraphics[width=4cm, height=3.9cm]{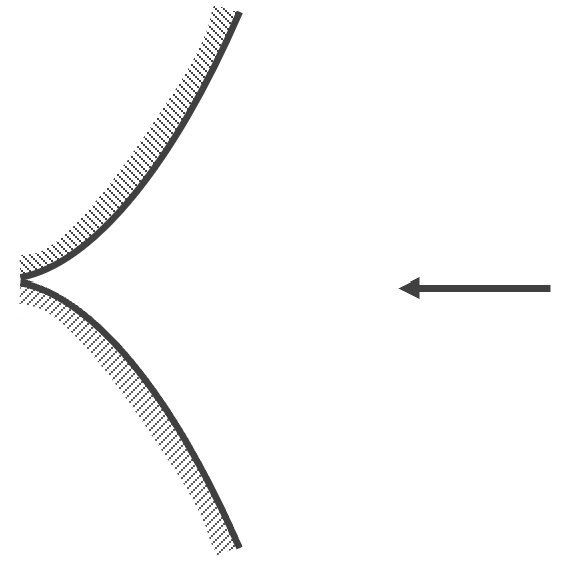}
\caption{Boundary curves with (\ref{CURC}).}
\label{cpic}
\end{figure}

To overcome the difficulties mentioned above, we introduce a new variable $(\xi,\psi)$. First, let $\psi$ be the harmonic function in $\Omega$ satisfying:
\begin{equation}\label{psi}
\left\{
\begin{aligned}
&-\Delta\psi=0\quad\text{ in }\Omega,\\
&\psi|_{\Gamma^+}=1, \hspace{2mm} \psi|_{\Gamma^-}=-1,\\
&\psi|_{r\rightarrow\infty}=\frac{\te}{\beta},\hspace{2mm} \psi|_{r\rightarrow0}=\frac{\te}{\al}.
\end{aligned}
\right.
\end{equation}
By the maximum principle, $\psi_\te$ is bounded and strictly positive in $\bar{\Omega}$. Due to the boundary conditions at the origin and the infinity,
$$\underset{r\rightarrow\infty}{\lim}\psi_\te=\frac{1}{\beta},\hspace{2mm}\underset{r\rightarrow0}{\lim}\psi_\te=\frac{1}{\al}.$$
It will be shown in the next section that $r^k\p_r^k\p_\te^l[\psi-\frac{\te}{\beta}]=\mathcal{O}(\frac{1}{r^\lambda}), \hspace{1mm}\text{for some }\lambda>0$ as $r\rightarrow \infty$. Set
$$
\xi:=\frac{1}{\beta}\ln{r}-\int_{r}^\infty\frac{1}{\rho}\big(\psi_\te(\rho,\te)-\frac{1}{\beta}\big)\dd\rho.
$$
It is easy to see that
\[\xi\sim
\left\{
\begin{split}
&\frac{1}{\beta}\ln{r}, \quad r\rightarrow \infty,\\
&\frac{1}{\al}\ln{r}, \quad r\rightarrow0.
\end{split}
\right.
\]
The symbol `$\sim$' means $\underset{r\rightarrow \infty}{\lim}\frac{\xi}{\frac{1}{\beta}\ln{r}}=1$. Since $\psi$ is harmonic,  the new variable $(\xi,\psi)$ satisfies
$$\p_r\xi=\frac{\p_\te\psi}{r}, \quad \frac{\p_\te\xi}{r}=-\p_r\psi.
$$
So $(\xi,\psi): \hspace{1mm} \Omega\rightarrow (-\infty,\infty)\times(-1,1)$ is a conformal isomorphism. Moreover, $\p_\xi$ and $\p_\psi$ are non-dimensional derivatives, which help us to deal with the asymptotic behaviors of the solution as $r\rightarrow\infty$ and $r\rightarrow0$.

Next we show that a stream-function of the limiting Euler flow can be chosen as $\frac{Q}{2}\psi$. Indeed, denote by $\Phi_e$ the stream-function of the Euler flow with the flux $-Q$. Similar to the viscous flow, there exists a smooth function $A^\infty_e(\te)\in C^\infty[-\beta,\beta]$ such that  
\begin{align}\label{asyE}
\begin{aligned}
&r^k\p_r^k\p_\te^l[-\p_\te\Phi_e(r,\te)-A_e^\infty(\te)]=\mathit{o}(1),\\
&r^k\p_r^k\p_\te^l[r\p_r\Phi_e(r,\te)-0]=\mathit{o}(1), \quad \text{ as $r\rightarrow+\infty$.}
\end{aligned}
\end{align}
If the inviscid flow has no stagnation point in $\Omega$, then there exists a smooth function $F_e$ such that
$$-\p_r^2\Phi_e-\frac{\p_r\Phi_e}{r}-\frac{\p^2_\te\Phi_e}{r^2}=F_e(\Phi_e).
$$
Taking the limit $r\rightarrow\infty$ in the above equation, one gets $F_e\equiv0$, so $\Phi_e$ is harmonic. It follows from (\ref{asyE}) that 
$$
\p_\te A^\infty_e=\p_\te[\p_\te\Phi_e+A_e^\infty]-\p^2_\te\Phi_e=\p_\te[\p_\te\Phi_e+A_e^\infty]+r\p_r[r\p_r\Phi_e]=\mathit{o}(1),
$$   
as $r\rightarrow \infty$. It implies $A^\infty_e\equiv C$. Then the mass flux condition shows $\int_{-\beta}^\beta A_e^\infty\dd\te=-Q$, so $A^\infty_e\equiv-\frac{Q}{2\beta}$. The no-penetration condition of the Euler flow implies $\Phi_e|_{\Gamma^+}$ and $\Phi_e|_{\Gamma^-}$ are two constants, and the mass flux condition shows $\Phi_e|_{\Gamma^+}-\Phi_e|_{\Gamma^-}=Q$. Without loss of generality, one may assume $\Phi_e|_{\Gamma^+}=Q/2$, so $\Phi_e$ solves the following problem:
\begin{equation}\label{psie}
\left\{
\begin{aligned}
&-\Delta\Phi_e=0\quad\text{ in }\Omega,\\
&\Phi_e|_{\Gamma^+}=\frac{Q}{2}, \hspace{2mm} \Phi_e|_{\Gamma^-}=-\frac{Q}{2},\\
&\Phi_e|_{r\rightarrow\infty}=\frac{Q\te}{2\beta},\hspace{2mm} \Phi_e|_{r\rightarrow0}=\frac{Q\te}{2\al},
\end{aligned}
\right.
\end{equation}
namely, $\Phi_e=\frac{Q\psi}{2}$. The asymptotic behavior (\ref{asyE}) implies that the vorticity of the Euler flow is zero upstream, thus, it is expected that the limiting Euler flow is irrotational in the whole region.  

In terms of the variable $(\xi,\psi)\in(-\infty,\infty)\times(-1,1)$, the Navier-Stokes equation (\ref{NSPHI}) for $\Phi$ reads
\begin{equation}\label{NSP}
\left\{
\begin{aligned}
&J\Phi_\psi\p_\xi\big(J\Delta_{\xi,\psi}\Phi\big)-J\Phi_\xi\p_\psi\big(J\Delta_{\xi,\psi}\Phi\big)+\e J\Delta_{\xi,\psi}\big(J\Delta_{\xi,\psi}\Phi\big)=0, \\
&\Phi|_{\psi=1}=\frac{Q}{2}, \quad\Phi|_{\psi=-1}=-\frac{Q}{2},\\
&\Phi_\psi|_{\psi=1}=\Phi_\psi|_{\psi=-1}=0,
\end{aligned}
\right.
\end{equation}
where $J=\frac{\psi_\te^2}{r^2}+\psi_r^2$, and $\Delta_{\xi,\psi}=\p^2_\xi+\p^2_\psi$. The boundary condition (\ref{BC}) is equivalent to $[\Phi_\xi,\Phi_\psi]|_{\psi=\pm1}=0$, so $\Phi|_{\psi=1}$ and $\Phi|_{\psi=-1}$ are two constants independent on $\xi$. The mass flux condition (\ref{FX}) implies $\Phi|_{\psi=1}-\Phi|_{\psi=-1}=Q$, so the boundary conditions in the system (\ref{NSP}) follow. 

The mismatching of the boundary conditions between the corresponding Euler flow and Navier-Stokes flow, i.e. $\p_\psi\Phi_e|_{\psi=\pm1}=\frac{Q}{2}\neq0=\p_\psi\Phi|_{\psi=\pm1}$, gives rise of the viscous boundary layers. To study this, we introduce the following Prandtl variables:
$$
(\xi,\eta)=\big(\xi,\sqrt{\frac{Q}{\e}}(\psi+1)\big), \quad \sqrt{\e Q}\tilde{\Phi}(\xi,\eta)=\Phi(\xi,\psi).
$$
Then the equation (\ref{NSP}) now reads
\begin{align}\label{NSPP}
\begin{aligned}
&J\tilde{\Phi}_\eta\p_\xi\Big(J\big(\frac{\e}{Q}\tilde{\Phi}_{\xi\xi}+\tilde{\Phi}_{\eta\eta}\big)\Big)-J\tilde{\Phi}_\xi\p_\eta\Big(J\big(\frac{\e}{Q}\tilde{\Phi}_{\xi\xi}+\tilde{\Phi}_{\eta\eta}\big)\Big)\\
&+J\Big(\frac{\e}{Q}\p^2_\xi+\p^2_\eta\Big)\Big(J\big(\frac{\e}{Q}\tilde{\Phi}_{\xi\xi}+\tilde{\Phi}_{\eta\eta}\big)\Big)=0.
\end{aligned}
\end{align}
Taking the formal limit $\e\rightarrow0$, one obtains an equation for $\Phi_p$ as:
$$
J(\xi,-1)\Phi_{p\eta}\p_\xi\big(J(\xi,-1)\Phi_{p\eta\eta}\big)-J^2(\xi,-1)\Phi_{p\xi}\Phi_{p\eta\eta\eta}-J^2(\xi,-1)\Phi_{p\eta\eta\eta\eta}=0,
$$
where $(\xi,\eta)\in(-\infty,\infty)\times(0,\infty)$. Since $J$ is positive, the above equation can be rewritten as
$$
\Phi_{p\eta}\Phi_{p\xi\eta\eta}+\frac{J_{\xi}(\xi,-1)}{J(\xi,-1)}\Phi_{p\eta}\Phi_{p\eta\eta}-\Phi_{p\xi}\Phi_{p\eta\eta\eta}+\Phi_{p\eta\eta\eta\eta}=0.
$$
Thus,
$$
\p_\eta\Big(\Phi_{p\eta}\Phi_{p\xi\eta}-\Phi_{p\xi}\Phi_{p\eta\eta}+\frac{1}{2}\frac{J_\xi(\xi,-1)}{J(\xi,-1)}(\Phi_{p\eta})^2+\Phi_{p\eta\eta\eta}\Big)=0.
$$
Therefore, there exists a function $C(\xi)$ depending only on $\xi$ such that
\begin{align}
\Phi_{p\eta}\Phi_{p\xi\eta}-\Phi_{\xi}\Phi_{p\eta\eta}+\frac{1}{2}\frac{J_\xi(\xi,-1)}{J(\xi,-1)}(\Phi_{p\eta})^2+\Phi_{p\eta\eta\eta}+C(\xi)=0.
\end{align}
The boundary conditions of $\Phi_p$ are
$$
\Phi_{p\xi}|_{\eta=0}=\Phi_{p\eta}|_{\eta=0}=0, \quad Q\cdot\Phi_{p\eta}|_{\eta\rightarrow\infty}=\Phi_{e\psi}(\xi,-1)=\frac{Q}{2}.
$$
Set $[u_p,v_p]:=[-\Phi_{p\eta},\Phi_{p\xi}]$. Then we have to solve the following Prandtl's problem in $(-\infty,\infty)\times(0,\infty)$:
\begin{equation}\label{PRA}
\left\{
\begin{aligned}
&u_pu_{p\xi}+v_pu_{p\eta}+\frac{1}{2}\frac{J_\xi(\xi,-1)}{J(\xi,-1)}u^2_p-u_{p\eta\eta}+C(\xi)=0,\\
&u_{p\xi}+v_{p\eta}=0,\\
&u_p|_{\eta=0}=v_p|_{\eta=0}=0,\quad u_p|_{\eta\rightarrow\infty}=-\frac{1}{2},
\end{aligned}
\right.
\end{equation}
with $C(\xi)=-\frac{1}{8}\frac{J_\xi(\xi,-1)}{J(\xi,-1)}$ due to the Bernoulli equation.  

If $u_p$ is negative, $-\xi$ can be regarded as a time variable and $\eta$ as a space variable, and $[u_p,v_p]$ is an entire solution to a parabolic system. We will prove the existence of the entire solution to (\ref{PRA}) under the curvature-decreasing condition (\ref{CURC}). Now the main results can be stated as
\begin{theorem}\label{main} Consider the steady Navier-Stokes system (\ref{NSE}) on the symmetric domain $\Omega$ with the smooth boundaries $\Gamma^\pm$. If $\Gamma^+$ satisfies (\ref{CUR}) and (\ref{CURC}), and the mass flux of the viscous flow is a negative constant $-Q$,    
\noindent then there exists $\delta$ depending only on $\Gamma^\pm$, such that for $0<\e\leqslant\delta Q$, there exists a solution to the Navier-Stokes problem (\ref{NSP}) with no-slip boundary conditions on $\Gamma^\pm$, which satisfies 
\begin{align}\nonumber
\begin{aligned}
&\bigg\|U(\xi,\psi)-Q\Big[-\frac{1}{2}+\Big(u_p\big(\xi,\sqrt{\frac{Q}{\e}}(\psi+1)\big)+\frac{1}{2}\Big)+\Big(u_p\big(\xi,\sqrt{\frac{Q}{\e}}(1-\psi)\big)+\frac{1}{2}\Big)\Big]\bigg\|_{L^\infty}\leqslant C\sqrt{\e Q},\\
&\hspace{5cm}\Big\|V(\xi,\psi)\Big\|_{L^\infty}\leqslant C\sqrt{\e Q},
\end{aligned}
\end{align}
where $[U,V]=[-\Phi_\psi,\Phi_\xi]$, $[u_p,v_p]$ is the solution to the Prandtl problem (\ref{PRA}), and $C$ is a constant independent of $\e$ and $Q$. Moreover, $[U,V]$ has the asymptotic behaviors as $\xi\rightarrow\pm\infty$:
\begin{align}\label{SBX}
\begin{aligned}
&\underset{\xi\rightarrow\infty}{\lim}[U(\xi,\psi),V(\xi,\psi)]=[\beta A^\infty(\beta\psi),0],\\
&\underset{\xi\rightarrow-\infty}{\lim}[U(\xi,\psi),V(\xi,\psi)]=[\al A^0(\al\psi),0],
\end{aligned}
\end{align}
with $A^\infty$ and $A^0$ being the unique non-positive solutions to problems (\ref{INBE}) and (\ref{ZEBE}) respectively. 
\end{theorem}

\begin{remark}It is the first result on rigorous justification of the viscous boundary layer theory corresponding to the sink-type inviscid flow to our knowledge (except for the case in a straight nozzle which can be reduced to problem (\ref{INBE})). This result shows the global stability of the matching of the inviscid flow and the Prandtl flow which are both singular at the vertex of the nozzle. Furthermore, the leading Euler flow in $\Omega$ is non-shear (i.e. the velocity depends both on $r$ and $\te$) if the curvature of $\Gamma^\pm$ is not vanishing everywhere (i.e. $\kappa\not\equiv0$).
\end{remark}

\begin{remark}Besides posing the non-slip boundary condition (\ref{BC}) as usual, we prescribe the mass flux condition (\ref{FX}) for the solution to the Navier-Stokes equations in contrast with the case of plate boundary layers in \cite{GZ20}, \cite{GZ21} and \cite{GN}-\cite{IM20}, which prescribe the boundary data on $x=0$ and $x=L$ depending on the viscosity and changing sharply in the $y$-direction. We find that the mass flux $-Q$ completely determines the asymptotic behaviors of the solution at the vertex of the nozzle and infinity. In this sense, the given mass flux condition is consistent with the geometry of the underlying nozzle.
\end{remark}

\begin{remark}Since $\p_\xi\approx\beta r\p_r$, $\p_\psi\approx\beta\p_\te$ as $r\rightarrow\infty$; and $\p_\xi\approx\al r\p_r$, $\p_\psi\approx\al\p_\te$ as $r\rightarrow0$, thus,
(\ref{SBX}) implies
\[
\begin{split}
&[ru(r,\te),rv(r,\te)]\rightarrow[A^\infty(\te),0] \text{ as }r\rightarrow\infty,\\
&[ru(r,\te),rv(r,\te)]\rightarrow[A^0(\te),0] \text{ as }r\rightarrow0,
\end{split}
\]
where $A^\infty$ and $A^0$ are the unique non-positive solutions to (\ref{INBE}) and (\ref{ZEBE}) respectively. Note that 
\[
\begin{split}
&A^\infty(\te)=-\frac{Q}{2\beta }-\frac{Q}{2\beta }\Big[f'\big(\sqrt{\frac{Q}{2\beta\e}}(\te+\beta)\big)-1\Big]-\frac{Q}{2\beta }\Big[f'\big(\sqrt{\frac{Q}{2\beta\e}}(\beta-\te)\big)-1\Big]+\mathcal{O}(\sqrt{\e Q}),\\
&A^0(\te)=-\frac{Q}{2\al }-\frac{Q}{2\al }\Big[f'\big(\sqrt{\frac{Q}{2\al\e}}(\te+\al)\big)-1\Big]-\frac{Q}{2\al }\Big[f'\big(\sqrt{\frac{Q}{2\al\e}}(\al-\te)\big)-1\Big]+\mathcal{O}(\sqrt{\e Q}).
\end{split}
\]
It means that the superposition of the sink flow and its corresponding self-similar Prandtl flow describes the behavoirs of the viscous channel flow both at the vertex and infinity and the viscous flow indeed is singular at the vertex of the nozzle.
\end{remark}

\begin{remark}The curvature-decreasing condition (\ref{CURC}) implies that
\begin{align} \label{FPC}
\frac{1}{2}\frac{J_\xi(\xi,\pm1)}{J(\xi,\pm1)}<0 \text{ for any } \xi\in\mathbb{R}.
\end{align}
(\ref{FPC}) means that the inviscid flow accelerates along the boundaries, equivalently, the pressure is strictly favourable on the boundaries. It guarantees the global well-posedness for the Prandtl equations. It is observed by the physicists that if the curvature of $\Gamma^\pm$ increases too fast, the pressure of the inviscid flow becomes adverse on the boundaries.
\end{remark}

\begin{remark}The leading order inviscid flow (modulo a positive multiplicative constant) is determined by $\Gamma^\pm$. One can generalize the condition (\ref{CURC}) to the curves $\Gamma^\pm$ where (\ref{FPC}) holds true. It covers the case that $\Gamma^\pm$ are perturbations of the straight rays. Precisely, $\Gamma^+=\Big\{\big(x(s),y(s)\big)\Big\}$ satisfies (\ref{CUR}) with $\al,\beta\in (0,\frac{\pi}{2})$, and 
\begin{align}\nonumber
\Big\|\frac{\dd^k}{\dd s^k}\arctan\Big(\frac{\overset{.}{y}}{\overset{.}{x}}\Big)\Big\|_{L^\infty(0,\infty)}\leqslant \delta\ll 1,\text{ for any $1\leqslant k\leqslant M$,}
\end{align}
where $M$ is a positive large integer. Hence, our analysis indicates also that the sink flow is structural stable under suitable perturbation of the nozzle and the viscosity if corrected with some Prandtl's boundary layers.  
\end{remark}

\begin{remark} 
Similar result holds true for any non-symmetric domain enclosed by $\Gamma^\pm$ with $\Gamma^+$ satisfying (\ref{CUR}) and (\ref{CURC}) while  $\Gamma^-=\Big\{\big(x^-(s),y^-(s)\big)\Big\}$ such that 
\begin{align}\nonumber
\begin{aligned}
&\big(x^-(0),y^-(0)\big)=(0,0),\quad (\overset{.}{x}^-(0),\overset{.}{y}^-(0))=(\cos\al^-,\sin\al^-),\\
&\big(\overset{.}{x}^--\cos\beta^-,\overset{.}{y}^--\sin\beta^-\big)\text{ and its derivatives decay fast to 0 as $s\rightarrow\infty$},
\end{aligned}
\end{align}
with $-\frac{\pi}{2}\leqslant\beta^-\leqslant\al^-<0$, and furthermore $\kappa^-=\big|\frac{\dd}{\dd s}\arctan\big(\frac{\overset{.}{y}^-}{\overset{.}{x}^-}\big)\big|$ satisfying
\begin{align}\nonumber
\frac{\dd^2}{\dd s^2}\arctan\Big(\frac{\overset{.}{y}^-}{\overset{.}{x}^-}\Big)\geqslant0 \text{ for any $s\geqslant0$.}
\end{align}
\end{remark}

\begin{remark} 
For positive flux, even in $\Omega_\beta$, i.e. $\Gamma^\pm$ are two rays, the inviscid flow decelerates along the boundaries, so the boundary layer separation will occur (see \cite{LL} and \cite{SG}).   
\end{remark}

To prove the theorem, we construct an approximate solution to the Navier-Stokes equations and then analyze the stability of the linearized Navier-Stokes equations around the approximate solution. 

Due to the geometry of the infinitely long convergent nozzle and possible singularity of both the corresponding inviscid and Prandtl's flows, the construction of the approximate solution is non-trivial. We should first show the favourable pressure condition (\ref{FPC}) for the inviscid flow since it is necessary in general to ensure the global well-posedness of the Prandtl problem (\ref{PRA}). One can prove (\ref{FPC}) by analyzing the flow angle $q:=-\arctan(\frac{\psi_x}{\psi_y})$, where $(x,y)=(r\cos\te,r\sin\te)$ is the original rectangular coordinate, which coincides with $\pm\arctan\big(\frac{\overset{.}{y}}{\overset{.}{x}}\big)$ along the streamlines $\Gamma^\pm$. It will be shown in next section that $\frac{1}{2}\ln J-iq$ is a holomorphic function. Thus, it follows from the Cauchy-Riemann equations that $\frac{1}{2}\frac{J_\xi}{J}=-q_\psi$. We will show that $q_\psi$ is strictly positive by the maximum principle under the assumption of the curvature-decreasing condition (\ref{CURC}). Therefore, $-\frac{1}{2}\frac{J_\xi}{J}$ is strictly positive. Based on this favourable pressure condition, one can obtain the global well-posedness of the Prandtl equations by the method of Oleinik. Furthermore, we will carry out some delicate analysis to obtain the asymptotic behaviors of the solution to the Prandtl equations as $\xi\rightarrow\pm\infty$.

The are two main ingredients in the construction of a high-order approximate solution by Prandtl's expansion. One is the stability analysis of the linearized Prandtl system, which is not standard since it is an entire parabolic system. However, (\ref{FPC}) and the concavity of the Prandtl solution lead to a damping effect in this estimate. The other one is the elaborate analysis of the asymptotic behaviors of the approximate solution at both vertex and upstream. Indeed, it is necessary to modify the approximate solution to ensure that the remainder to vanish as $\xi\rightarrow\pm\infty$ in order to carry out the energy estimates for the remainder. However, the approximate solution constructed by the standard matched asymptotic expansion method does not possess this property. To overcome this difficulty, we need to give some careful stability analysis for the solutions to problems (\ref{INBE}) and (\ref{ZEBE}). One can view it as a special case of the stability of the linearized Navier-Stokes equations. Based on these analyses, we will obtain a high-order approximate solution with the mass flux $-Q$, which has the same asymptotic behaviors as the solution to the Navier-Stokes equations as $\xi\rightarrow\pm\infty$.  

At last, the critical step here is the stability analysis of the linearized Navier-Stokes equations around the approximate solution. We will estimate the quotient of the stream-function to the first component of the velocity for the Prandtl flow, introduced first in \cite{GZ20}. A key idea is to perform an energy estimate for this quantity with a suitable weight to achieve 
\begin{itemize}
\item uniform positivity due to the favourable pressure condition (\ref{FPC});
\item cancellation of terms caused by the non-shear Euler flow;
\item cancellation of terms caused by the normal derivatives of the Prandtl solution near the boundaries.
\end{itemize}
Based on these and the concavity of the Prandtl solution, we will obtain some positive terms in the energy estimate from the convect terms in the linearized Navier-Stokes equations, which leads to a stability estimate. Once the stability estimate holds true, one can establish the main theorem easily by the contraction mapping method.

{\bf Notation.}
For convenience, the following conventions will be used: $$\z\cdot,\cdot\y=\z\cdot,\cdot\y_{L^2_{\xi,\psi}\big(\mathbb{R}\times(-1,1)\big)}, $$ $$\z\cdot,\cdot\y|_{\psi=\pm1}=\z\cdot,\cdot\y_{L^2_\xi\big(\mathbb{R}\times\pm1\big)},$$ $$\|\cdot\|=\|\cdot\|_{L^2_{\xi,\psi}\big(\mathbb{R}\times(-1,1)\big)}$$ and $$\|\cdot\|_{\infty}=\|\cdot\|_{L^\infty_{\xi,\psi}}.$$ 
$a=\MO\big(b\big)$ means that there exists a positive constant $C$ depending only on $\Gamma^\pm$, such that $|a|\leqslant Cb$.

\section{The leading-order Euler flow and favourable pressure condition}\label{Euler-section}
In this section, we discuss the limiting Euler flow and its favorable pressure condition on the boundaries. Since the limiting Euler flow is determined by the curves $\Gamma^\pm$, so we start with the discussion on the properties of the boundary curves. Recall 
\[
\Gamma^+=\Big\{\big(x(s),y(s)\big)\in\mathbb{R}^2\Big|x(s)=r(s)cos\big(g\big(r(s)\big)\big),\hspace{1mm}y(s)=r(s)sin\big(g\big(r(s)\big)\big),\hspace{1mm} s\geqslant0\Big\}.
\]
The following lemma for $g(r)$ holds.
\begin{lemma}\label{g}
Assume that $\Gamma^+$ is a smooth curve satisfying (\ref{CUR}) and (\ref{CURC}). It then holds that
\[
\begin{split}
&g(0)=\al,\quad g(\infty)=\beta,\\
&g'(r)\geqslant0 \text{ for any $r\geqslant0$},
\end{split}
\]
and 
\begin{align}\label{BHG}
\begin{aligned}
&\Big|r^k\frac{\dd^k }{\dd r^k} g(r)\Big|\leqslant \frac{C_k}{r}, \text{ for $r\geqslant1$,}\\
&\Big|r^k\frac{\dd^k }{\dd r^k} g(r)\Big|\leqslant C_kr, \text{ for $r<1$,}
\end{aligned}
\end{align}
where $k$ is any positive integer.
\end{lemma}   
\begin{proof}
We first prove (\ref{BHG}). Since $\big(\overset{.}{x}(s)-\cos\beta,\overset{.}{y}(s)-\sin\beta\big)$ and its derivatives decay fast to 0 as $s\rightarrow\infty$, thus $\underset{s\rightarrow\infty}{\lim}\kappa(s)=0$. It follows from (\ref{CURC}) that 
$$\kappa(s)=\frac{\dd}{\dd s}\arctan\big(\frac{\overset{.}{y}}{\overset{.}{x}}\big)\geqslant0 \text{ for } s\geqslant0.$$
Set $\tilde{\te}(s)=\arctan\big(\frac{\overset{.}{y}(s)}{\overset{.}{x}(s)}\big)$. Then $\frac{\dd \tilde{\te}}{\dd s}=\kappa\geqslant0$. So $\tilde{\te}$ is a non-decreasing function with $\tilde{\te}(0)=\al$ and $\tilde{\te}(\infty)=\beta$. It follows that
\begin{align}
\begin{aligned}
&0\leqslant \cos\beta\leqslant \overset{.}{x}\leqslant \cos\al,\\
&0< \sin\al\leqslant \overset{.}{y}\leqslant \sin\beta.
\end{aligned}
\end{align}
For $r(s)=\sqrt{x^2(s)+y^2(s)}$, $\frac{\dd r}{\dd s}=\frac{x\overset{.}{x}+y\overset{.}{y}}{r}$. Due to $\overset{.}{x}^2(s)+\overset{.}{y}^2(s)=1$, one has $\frac{\dd r}{\dd s}\leqslant1$, and hence $r(s)\leqslant s$. Since 
\[
\begin{split}
&x(s)\overset{.}{x}(s)=\big(x(s)-x(0)\big)\overset{.}{x}(s)\geqslant 0,\\
&y(s)\overset{.}{y}(s)=\big(y(s)-y(0)\big)\overset{.}{y}(s)\geqslant s\overset{.}{y}(s)\sin\al\geqslant s\sin^2\al,
\end{split}
\]
one has 
$$
\frac{\dd r}{\dd s}\geqslant \frac{y\overset{.}{y}}{r}\geqslant \frac{s\sin^2\al}{r}\geqslant \sin^2\al>0.
$$
It follows that $\sin^2\al\leqslant \frac{\dd r}{\dd s}\leqslant 1$ and $s\sin^2\al \leqslant r\leqslant s$.

Recall $\te(s)=g(r(s))=\arctan(\frac{y(s)}{x(s)})$. Thus 
$$
rg'(r)=r\frac{\dd \te}{\dd s}\big/\frac{\dd r}{\dd s}=\frac{\overset{.}{y}x-\overset{.}{x}y}{x\overset{.}{x}+y\overset{.}{y}}\hspace{1mm}.
$$
For $s\leqslant \sin^{-2}\al$,
\[
\begin{split}
&\overset{.}{y}(s)x(s)-\overset{.}{x}(s)y(s)\\
=&\overset{.}{y}(0)\big[x(s)-s\overset{.}{x}(0)\big]-\overset{.}{x}(0)\big[y(s)-s\overset{.}{y}(0)\big]+\big[\overset{.}{y}(s)-\overset{.}{y}(0)\big]x(s)-\big[\overset{.}{x}(s)-\overset{.}{x}(0)\big]y(s),\\
\end{split}
\]
so $|\overset{.}{y}(s)x(s)-\overset{.}{x}(s)y(s)|\leqslant Cs^2$. It implies
$$
|r g'(r)|\leqslant C r \text{ for } r\leqslant 1.
$$
For $s$ large enough, 
$$
\overset{.}{y}(s)x(s)-\overset{.}{x}(s)y(s)=\big[\overset{.}{y}(s)-\sin\beta\big]x(s)-\big[\overset{.}{x}(s)-\cos\beta]y(s)+\sin\beta\cdot x(s)-\cos\beta\cdot y(s).
$$
Since $\big(\overset{.}{x}(s)-\cos\beta,\overset{.}{y}(s)-\sin\beta\big)$ decays fast to 0 as $s\rightarrow\infty$, one gets that
\[
\begin{split}
\sin\beta\cdot x(s)-\cos\beta\cdot y(s)=\sin\beta\cdot x(1)-\cos\beta\cdot y(1)+\int_{1}^{s}\big[\overset{.}{x}(\tau)\sin\beta-\overset{.}{y}(\tau)\cos\beta\big]\dd \tau
\end{split}
\]
is bounded as $s\rightarrow\infty$. Therefore, $\overset{.}{y}x-\overset{.}{x}y$ is bounded for $s\geqslant 1$, so
$$
|r g'(r)|\leqslant \frac{C}{r} \text{ for } r\geqslant 1.
$$
The estimates on high-order derivatives in (\ref{BHG}) are similar. 

Since $rg'(r)=\frac{\overset{.}{y}x-\overset{.}{x}y}{x\overset{.}{x}+y\overset{.}{y}}$, we write $h(s):=x-\frac{\overset{.}{x}}{\overset{.}{y}}y$, then
$$
\frac{\dd h}{\dd s}=\frac{\kappa y}{\overset{.}{y}^2}\geqslant0,
$$
so $h(s)\geqslant h(0)=0$, which implies $g'\geqslant0$.

Since
$\overset{.}{x}=\overset{.}{r}\Big(\cos\big(g(r)\big)-rg'(r)\sin\big(g(r)\big)\Big)$, $\overset{.}{y}=\overset{.}{r}\Big(\sin\big(g(r)\big)+rg'(r)\cos\big(g(r)\big)\Big)$, thus
$$\frac{\cos\big(g(r)\big)}{\sin\big(g(r)\big)}=\frac{\overset{.}{x}+\overset{.}{r}rg'\cos\big(g(r)\big)}{\overset{.}{y}-\overset{.}{r}rg'\sin\big(g(r)\big)}.
$$
Taking the limit as $s\rightarrow0$ and $s\rightarrow \infty$ yields $g(0)=\al$ and $g(\infty)=\beta$ immediately from (\ref{BHG}).
\end{proof}

Now we show the existences of the Euler flow. Set $\Phi_e=\frac{Q\psi}{2}$, where $\psi$ solves
\begin{equation}\label{psi-Euler}
\left\{
\begin{aligned}
&-\Delta\psi=0\quad\text{ in }\Omega,\\
&\psi|_{\Gamma^+}=1, \hspace{2mm} \psi|_{\Gamma^-}=-1,\\
&\psi|_{r\rightarrow\infty}=\frac{\te}{\beta},\hspace{2mm} \psi|_{r\rightarrow0}=\frac{\te}{\al}.
\end{aligned}
\right.
\end{equation}
Then the following lemma for $\psi$ holds.
\begin{lemma}\label{psi-E-est}
Under the assumption of Lemma \ref{g}, (\ref{psi-Euler}) admits a solution satisfying
\begin{align}\label{NEFE}
\begin{aligned}
&\Big|r^k\p_r^k\p_\te^l\Big[\psi(r,\te)-\frac{\te}{\beta}\Big]\Big|\leqslant\frac{C_{k,l}(\lambda)}{r^\lambda},\text{   for }r\geqslant1,\\
&\Big|r^k\p_r^k\p_\te^l\Big[\psi(r,\te)-\frac{\te}{\al}\Big]\Big|\leqslant C_{k,l}(\lambda)r^\lambda,\text{   for }r\leqslant1,
\end{aligned}
\end{align}
where $\lambda<\frac{1}{2}$ is a positive constant. And there exists a positive constant $b$, such that
\begin{align}\label{lob}
\underset{\Omega}{\inf}\hspace{1mm}\psi_\te\geqslant b>0.
\end{align}
\end{lemma}
\begin{proof}
In the coordinate $(t,\te)$, where $t=\ln r$, (\ref{psi-Euler}) reads
\begin{equation}\label{maxi}
\left\{
\begin{aligned}
&-\p_t^2\psi-\p_\te^2\psi=0, \quad t\in\mathbb{R},\hspace{1mm}|\te|<g(e^t),\\
&\psi|_{\te=g(e^t)}=1, \hspace{2mm} \psi|_{\te=-g(e^t)}=-1,\\
&\psi|_{t\rightarrow\infty}=\frac{\te}{\beta},\hspace{2mm} \psi|_{t\rightarrow-\infty}=\frac{\te}{\al}.
\end{aligned}
\right.
\end{equation}
Let $\tilde{\psi}=\psi-\frac{\te}{g(e^t)}$. Then $\tp$ solves the following problem
\begin{equation}\label{eqpsi}
\left\{
\begin{aligned}
&-\p_t^2\tp-\p_\te^2\tp=\tilde{f},\quad t\in\mathbb{R},\hspace{1mm}|\te|<g(e^t),\\
&\tp|_{\te=g(e^t)}=0, \hspace{2mm} \tp|_{\te=-g(e^t)}=0,\\
&\tp|_{t\rightarrow\infty}=0,\hspace{2mm} \tp|_{t\rightarrow-\infty}=0,
\end{aligned}
\right.
\end{equation}
where $\tilde{f}=\frac{2\te\big(e^tg'(e^t)\big)^2}{g^3}-\frac{\te e^tg'(e^t)}{g^2}-\frac{\te e^{2t}g''(e^t)}{g^2}$. It follows from (\ref{BHG}) that 
\[
\begin{split}
&|\tilde{f}(t,\te)|\leqslant Ce^{-t} \text{   for }t\geqslant 0,\\
&|\tilde{f}(t,\te)|\leqslant Ce^{t} \text{   for }t\leqslant 0.
\end{split}
\]
Multiplying the first equation in (\ref{eqpsi}) by $\tp \cosh(2\lambda t)$ and integrating in $\Big\{(t,\te)\big||\te|< g(e^t),\hspace{1mm}t\in\mathbb{R}\Big\}$ yield that
\begin{align}\nonumber
\begin{aligned}
&\int_{-\infty}^{\infty}\int_{-g(e^t)}^{g(e^t)}\Big[\big|\p_t\tp\big|^2+\big|\p_\te\tp\big|^2-4\lambda^2\big|\tp\big|^2\Big]\cosh(2\lambda t)\dd \te\dd t\\
=&\int_{-\infty}^{\infty}\int_{-g(e^t)}^{g(e^t)}\tilde{f}\tp \cosh(2\lambda t)\dd \te\dd t.
\end{aligned}
\end{align}
Since $\tp|_{\te=\pm g}=0$, so the Poincar\'{e} inequality gives  
$$
\int_{-g}^g \big|\tp\big|^2 \dd \te\leqslant \big(\frac{2g}{\pi}\big)^2\int_{-g}^g \big|\p_\te\tp\big|^2 \dd \te\leqslant\int_{-g}^g \big|\p_\te\tp\big|^2 \dd \te.
$$
For $\lambda<\frac{1}{2}$, it holds that 
$$
\int_{-\infty}^{\infty}\int_{-g(e^t)}^{g(e^t)}\Big[\big|\p_t\tp\big|^2+\big|\p_\te\tp\big|^2\Big]\cosh(2\lambda t)\dd \te\dd t\leqslant C(\lambda) \int_{-\infty}^{\infty}\int_{-g(e^t)}^{g(e^t)}|\tilde{f}|^2 \cosh(2\lambda t)\dd \te\dd t.
$$
This implies the existences of a solution to (\ref{eqpsi}). Moreover, by a similar argument, one can get
\begin{align}\nonumber
\begin{aligned}
\int_{-\infty}^{\infty}\int_{-g(e^t)}^{g(e^t)}\big|\p^k_t\p^l_\te\tp\big|^2\cosh(2\lambda t)\dd \te\dd t\leqslant C_{k,l}(\lambda)< \infty.
\end{aligned}
\end{align}
Thus, it holds that 
\begin{align}\label{wht}
\big|\p_t^k\p^l_\te \tp(t,\te)\big|\leqslant\frac{C_{k,l}(\lambda)}{\cosh(\lambda t)}, \text{   for }t\in\mathbb{R},\hspace{1mm}|\te|\leqslant g(e^t),
\end{align}
which implies (\ref{NEFE}). 

To obtain (\ref{lob}), we use the maximum principle. Since $\psi$ is harmonic respect to $(t,\te)$, and $\psi$ attains its strictly maximum on $\te=g(e^t)$, and the Hopf's lemma implies 
$$\psi_{\te}(t,g(e^t))>0 \text{  for any } t\in \mathbb{R}.
$$
(\ref{wht}) shows that $\psi_{\te}(t,\te)\rightarrow\frac{1}{\beta}$ uniformly as $t\rightarrow\infty$ and $\psi_{\te}(t,\te)\rightarrow\frac{1}{\al}$ uniformly as $t\rightarrow-\infty$, hence there exists a positive constant $b$ such that $\psi_{\te}(t,g(e^t))\geqslant b$ for any $t\in\mathbb{R}$. Similarly $\psi_{\te}(t,-g(e^t))\geqslant b$. Since $\psi_\te$ is a harmonic function, the maximum principle implies (\ref{lob}). The proof of Lemma \ref{psi-E-est} is completed.
\end{proof}

Recall the definition of $\xi$:
$$\xi=\frac{1}{\beta}\ln{r}-\int_{r}^\infty\frac{1}{\rho}\big(\psi_\te(\rho,\te)-\frac{1}{\beta}\big)\dd\rho.$$ 
Since $\psi$ is a harmonic function satisfying (\ref{NEFE}), it holds that
$$\p_r\xi=\frac{\p_\te\psi}{r}, \quad \frac{\p_\te\xi}{r}=-\p_r\psi.
$$ 
Since
\begin{align}\label{xib}
\xi\sim
\left\{
\begin{aligned}
&\frac{1}{\beta}\ln{r}, \quad r\rightarrow \infty,\\
&\frac{1}{\al}\ln{r}, \quad r\rightarrow0,
\end{aligned}
\right.
\end{align}
and $\underset{\Omega}{\inf}\hspace{1mm}\psi_\te\geqslant b>0,$ it follows that the mapping $(\xi,\psi): \hspace{1mm} \Omega\rightarrow (-\infty,\infty)\times(-1,1)$ is a conformal isomorphism. It is easy to see that $\xi$ and $\psi$ are smooth up to the boundaries $\Gamma^\pm$ except for the origin. 

Set $J=\frac{\psi_\te^2}{r^2}+\psi_r^2$. Then the magnitude of the velocity for the Euler flow is $\frac{Q}{2}J^\frac{1}{2}$. By choosing $\lambda=\frac{1}{4}$ in (\ref{NEFE}), one can obtain the following asymptotic behaviors for $J$ as $\xi\rightarrow\pm\infty$:  
\begin{align}\label{asybJ1}
\begin{aligned}
&\bigg|\p_\xi^k\p_\psi^l\Big[\frac{1}{J}\frac{\p^{j} J}{\p \xi^j}-\big(-2\beta\big)^j\Big]\bigg|\leqslant C_{k,l,j}e^{-\frac{\beta}{4}\xi} \quad\text{for } \xi\geqslant0,\\
&\bigg|\p_\xi^k\p_\psi^l\Big[\frac{1}{J}\frac{\p^{j} J}{\p \xi^j}-\big(-2\al\big)^j\Big]\bigg|\leqslant C_{k,l,j}e^{\frac{\al}{4}\xi} \quad\text{for } \xi<0,
\end{aligned}
\end{align}
\begin{align}\label{asybJ2}
\begin{aligned}
&\bigg|\p_\xi^k\p_\psi^l\Big[\frac{1}{J}\frac{\p^{j} J}{\p \psi^j}\Big]\bigg|\leqslant C_{k,l,j}e^{-\frac{\beta}{4}\xi} \quad\text{for } \xi\geqslant0,\\
&\bigg|\p_\xi^k\p_\psi^l\Big[\frac{1}{J}\frac{\p^{j} J}{\p \psi^j}\Big]\bigg|\leqslant C_{k,l,j}e^{\frac{\al}{4}\xi} \quad\text{for } \xi<0,
\end{aligned}
\end{align}
where $k\geqslant0$, $l\geqslant0$, and $j\geqslant1$ are integers. If choosing $k=0$, $l=0$, and $j=1$ in (\ref{asybJ1}) yields
\begin{align}\label{Jlimit}
\begin{aligned}
\underset{\xi\rightarrow\infty}{\lim}-\frac{1}{2}\frac{J_\xi}{J}=\beta,\quad \underset{\xi\rightarrow-\infty}{\lim}-\frac{1}{2}\frac{J_\xi}{J}=\al,
\end{aligned}
\end{align}
uniformly for $\psi\in[-1,1]$. Next lemma will show that the curvature-decreasing condition (\ref{CURC}) implies the inviscid flow accelerating on the boundaries. 
\begin{lemma}\label{keyE}
If $\Gamma^\pm$ are smooth curves satisfying (\ref{CUR}) and (\ref{CURC}), then there exists a constant $\mu>0$, such that
\begin{align}\label{key1}
-\frac{1}{2}\frac{J_\xi(\xi,\pm1)}{J(\xi,\pm1)}\geqslant \mu  \text{   for any }\xi\in\mathbb{R}.
\end{align}
Moreover, since $\frac{J_\xi}{J}$ is harmonic, it holds that
\begin{align}\label{key2}
\underset{\mathbb{R}\times(-1,1)}{\inf}-\frac{1}{2}\frac{J_\xi(\xi,\psi)}{J(\xi,\psi)}\geqslant \mu.
\end{align}
\end{lemma}
\begin{proof}
Let $(x,y)=(r\cos\te,r\sin\te)$ be the rectangular coordinate. Then $\xi$ and $\psi$ satisify the Cauchy-Riemann equations:
$$\xi_x=\psi_y,\quad \xi_y=-\psi_x.$$
Hence $\xi+i\psi$ is a holomorphic function respect to $x+iy$, so is $\xi_x+i\psi_x$. It is easy to see that $J=\psi_x^2+\psi_y^2=\psi_x^2+\xi_x^2$ and $J>0$ in $\Omega$, and then $\ln(\xi_x+i\psi_x)=\frac{1}{2}\ln J+i\arg (\xi_x+i\psi_x)$ is holomorphic.

Since $\psi$ attains its strictly maximum on $\Gamma^+$, and the outward normal vector of $\Gamma^+$ is $\big(-\overset{.}{y}(s),\overset{.}{x}(s)\big)$ with $\overset{.}{y}(s)> 0$ and $\overset{.}{x}(s)\geqslant 0$, the Hopf's lemma implies $-\psi_x|_{\Gamma^+}>0$ and $\psi_y|_{\Gamma^+}\geqslant0$. Similarly $\psi_x|_{\Gamma^-}>0$ and $\psi_y|_{\Gamma^-}\geqslant0$. It follows from (\ref{NEFE}) that $\psi_y\geqslant0$ as $r\rightarrow 0$ and $\infty$. The strong maximum principle implies $\psi_y>0$ in $\Omega$, hence $q:=\arctan\big(-\frac{\psi_x}{\psi_y}\big)$ is well-defined in $\Omega$. Recall $\tilde{\te}(s)=\arctan(\frac{\overset{.}{y}(s)}{\overset{.}{x}(s)})$. Since $\psi\big(x(s),y(s)\big)\equiv1$, thus $\overset{.}{x}\psi_x+\overset{.}{y}\psi_y=0$. $\overset{.}{y}>0$ and $-\psi_x|_{\Gamma^+}>0$, thus $ \frac{\overset{.}{x}(s)}{\overset{.}{y}(s)}=-\frac{\psi_y}{\psi_x}\big|_{\Gamma^+}$, which implies $q|_{\Gamma^+}=\tilde{\te}$. It is easy to check that $q$ is smooth up to the boundary except for the orgin because $J>0$. According to $\xi_x=\psi_y$, one gets $\ln(\xi_x+i\psi_x)=\frac{1}{2}\ln J+i\arg (\xi_x+i\psi_x)=\frac{1}{2}\ln J-iq$.

Since $(\xi,\psi): \hspace{1mm} \Omega\rightarrow (-\infty,\infty)\times(-1,1)$ is a conformal isomorphism, thus $\frac{1}{2}\ln J-iq$ is also a holomorphic function with respect to $\xi+i\psi$. The Cauchy-Riemann equations show 
$$\frac{1}{2}\frac{J_\xi}{J}=-q_\psi.
$$
So it suffices to prove $q_\psi\geqslant\mu>0$. That $\frac{1}{2}\ln J-iq$ is holomorphic implies that $q$ is harmonic. It follows from (\ref{Jlimit}) that $\underset{\xi\rightarrow\infty}{\lim} q_\psi=\beta$ and $\underset{\xi\rightarrow-\infty}{\lim} q_\psi=\al$ uniformly for $\psi\in[-1,1]$. Therefore $q$ solves the following problem:
\begin{equation}\label{q-equ}
\left\{
\begin{aligned}
&-\p_\xi^2q-\p^2_\psi q=0,\quad(\xi,\psi)\in\mathbb{R}\times(-1,1),\\
&q|_{\psi=\pm1}=\pm\tilde{\te},\\
&q|_{\xi\rightarrow\infty}=\beta\psi,\hspace{2mm} q|_{\xi\rightarrow-\infty}=\al\psi,
\end{aligned}
\right.
\end{equation}
where $\underset{\xi\rightarrow\infty}{\lim}\tilde{\te}=\beta$ and $\underset{\xi\rightarrow-\infty}{\lim}\tilde{\te}=\al$. Since $q|_{\psi=1}=\tilde{\te}$, one can get
$$q_\xi|_{\psi=1}=\frac{\dd \tilde{\te}}{\dd s}\frac{\dd s}{\dd \xi}=\kappa(s)\frac{\dd s}{\dd \xi}.$$
Noticing that on the boundary $\Gamma^+$, $\overset{.}{x}\psi_x+\overset{.}{y}\psi_y=0$, $\psi_x<0$, $\overset{.}{y}>0$ and $\overset{.}{x}^2+\overset{.}{y}^2=1$, one has
\[
\begin{split}
\frac{\dd \xi}{\dd s}&=\frac{\p \xi}{\p x}\overset{.}{x}+\frac{\p \xi}{\p y}\overset{.}{y}=\frac{\p \psi}{\p y}\overset{.}{x}-\frac{\p \psi}{\p x}\overset{.}{y}\\
                     &=\sqrt{\psi_y^2+\psi_x^2}=J^{\frac{1}{2}}.
\end{split}
\]
Thus, $q_\xi|_{\psi=1}=\frac{\kappa}{J^{\frac{1}{2}}}$. Taking the partial derivatives with respect to $\xi$ yields
\[
\begin{split}
q_{\xi\xi}|_{\psi=1}&=\frac{1}{J}\frac{\dd \kappa}{\dd s}+\frac{\kappa}{J^\frac{1}{2}}q_\psi.
\end{split}
\]
Due to $-\p_\xi^2q-\p^2_\psi q=0$, it holds that 
$$q_{\psi\psi}|_{\psi=1}=-\frac{1}{J}\frac{\dd \kappa}{\dd s}-\frac{\kappa}{J^\frac{1}{2}}q_\psi.$$ 
If $q_\psi$ attains its negative minimum on $(\xi_0,1)$, there are two possible cases, one is $\kappa(s(\xi_0))>0$, then (\ref{CURC}) implies $q_{\psi\psi}(\xi_0,1)\geqslant -\frac{K}{J^\frac{1}{2}}q_\psi>0$, which is impossible; the other is $\kappa(s(\xi_0))=0$, where $\frac{\dd \kappa}{\dd s}\leqslant0$ and $\underset{s\rightarrow\infty}{\lim}\kappa(s)=0$ imply $\kappa(s)\equiv0$ for $s\geqslant s(\xi_0)$. Since $q_\xi|_{\psi=1}=\frac{\kappa}{J^{\frac{1}{2}}}$, one gets $q(\xi_0,1)=\underset{\xi\rightarrow\infty}{\lim}q(\xi,1)=\beta$. The maximum principle shows $q\leqslant\beta$, so $q_\psi(\xi_0,1)\geqslant0$, which is a contradiction. Consequently $q_\psi|_{\psi=1}\geqslant0$. Similarly $q_\psi|_{\psi=-1}\geqslant0$. Note that $q_\psi$ is a harmonic function, and $q_\psi|_{\xi\rightarrow\infty}=\beta$, and $q_\psi|_{\xi\rightarrow-\infty}=\al$, by applying the maximum principle for $q_\psi$, one can get $q_\psi\geqslant 0$.

The strong maximum principle shows $q_\psi>0$ for $|\psi|<1$. If there exists $\xi_0$, such that $q_\psi(\xi_0,1)=0$. Applying the Hopf's lemma for $q_\psi$ shows $q_{\psi\psi}(\xi_0,1)<0$, which contradicts to $q_{\psi\psi}(\xi_0,1)=-\frac{1}{J}\frac{\dd \kappa}{\dd s}\geqslant 0$. So we have $q_\psi(\xi,1)>0$ for any $\xi\in\mathbb{R}$. Since $q_\psi|_{\xi\rightarrow\infty}=\beta$, and $q_\psi|_{\xi\rightarrow-\infty}=\al$, there exists a positive constant $\mu\leqslant\min\{\al,\beta\}$, such that $q_\psi(\xi,1)\geqslant \mu$ for any $\xi\in\mathbb{R}$. Same arguments work on $\psi=-1$. Therefore, we can obtain 
$$
\underset{\mathbb{R}\times(-1,1)}{\inf}q_\psi(\xi,\psi)\geqslant \mu,
$$
since $q_\psi$ is harmonic. So the proof is completed.
\end{proof}
\begin{remark}
Since the magnitude of the velocity for the Euler flow is $\frac{Q}{2}J^\frac{1}{2}$ and the $\xi$-direction is opposite to the velocity direction, so (\ref{key1}) means that the limiting Euler flow accelerates strictly on the boundaries, equivalently, the pressure is strictly favourable on the boundaries.
\end{remark}

\section{The existence of the global solution to the Prandtl equations}\label{Prantdl existence}
Now let us establish the global well-posedness of the Prandtl boundary layer problem under the assumption of 
\begin{align}\label{FPC2}
-\frac{1}{2}\frac{J_\xi(\xi,\pm1)}{J(\xi,\pm1)}\geqslant \mu  \text{   for any }\xi\in\mathbb{R}.
\end{align}
The Prandtl problem is posed in $\mathbb{R}\times(0,\infty)$:
\begin{equation}\label{PE}
\left\{
\begin{aligned}
&u_pu_{p\xi}+v_pu_{p\eta}+\frac{1}{2}\frac{J_\xi(\xi,-1)}{J(\xi,-1)}u^2_p-u_{p\eta\eta}-\frac{1}{8}\frac{J_\xi(\xi,-1)}{J(\xi,-1)}=0,\\
&u_{p\xi}+v_{p\eta}=0,\\
&u_p|_{\eta=0}=v_p|_{\eta=0}=0,\quad u_p|_{\eta\rightarrow\infty}=-\frac{1}{2}.
\end{aligned}
\right.
\end{equation}
Recall (\ref{Jlimit}), 
$$\underset{\xi\rightarrow\infty}{\lim}-\frac{1}{2}\frac{J_\xi}{J}=\beta,\quad \underset{\xi\rightarrow-\infty}{\lim}-\frac{1}{2}\frac{J_\xi}{J}=\al.
$$
And choosing $j=1$ in (\ref{asybJ1}) yields
\begin{align}\label{GD}
\begin{aligned}
&\bigg|\p_\xi^k\p_\psi^l\Big[-\frac{1}{2}\frac{J_\xi(\xi,\psi)}{J(\xi,\psi)}-\beta\Big]\bigg|\leqslant C_{k,l}e^{-\frac{\beta}{4}\xi} \quad\text{for } \xi\geqslant0,\\
&\bigg|\p_\xi^k\p_\psi^l\Big[-\frac{1}{2}\frac{J_\xi(\xi,\psi)}{J(\xi,\psi)}-\al\Big]\bigg|\leqslant C_{k,l}e^{\frac{\al}{4}\xi} \quad\text{for } \xi<0,
\end{aligned}
\end{align}
where $k$, $l$ are non-negative integers. Similar to the discussion on the Navier-Stokes equations and the Euler equations, we seek the solution $[u_p,v_p]$ to the Prandtl system satisfying 
\begin{align}\nonumber
&\underset{\xi\rightarrow\pm\infty}{\lim}[u_p(\xi,\eta),v_p(\xi,\eta)]=[A^\pm_p(\eta),B^\pm_p(\eta)].
\end{align}
The divergence-free condition shows $B^+_{p\eta}=-A^+_{p\xi}=0$, thus $B^+_p\equiv0$ follows from the boundary condition $v_p|_{\eta=0}=0$. Evaluating (\ref{PE}) as $\xi\rightarrow\infty$ yields
\begin{equation}\label{PEI}
\left\{
\begin{aligned}
&-\beta \big[A^+_p\big]^2 -\frac{\dd^2 }{\dd \eta^2}A^+_p+\frac{1}{4}\beta=0,\\
&A^+_p|_{\eta=0}=0,\quad A^+_p|_{\eta\rightarrow\infty}=-\frac{1}{2},
\end{aligned}
\right.
\end{equation}
thus $A^+_p(\eta)=-\frac{1}{2}f'\big(\sqrt{\frac{\beta}{2}}\eta\big)$ with $f$ being the solution to (\ref{FSE}). Similarly $\big[A^-_p(\eta),B^-_{p}(\eta)\big]=\Big[-\frac{1}{2}f'\big(\sqrt{\frac{\al}{2}}\eta\big),0\Big]$. 

The problem (\ref{FSE}) can be solved explicitly by 
$$ 
f'(\tilde{\eta})=3\tanh^2\Big(\frac{\tilde{\eta}}{\sqrt{2}}+\mathrm{artanh}\sqrt{\frac{2}{3}}\hspace{1mm}\Big)-2.
$$
It is easy to see $f''(\tilde{\eta})>0$ for $\tilde{\eta}\geqslant0$. $f'(0)=0$ and $f'(\infty)=1$ imply $0\leqslant f'\leqslant 1$. Hence $-f'''=1-(f')^2>0$. It follows  from the explicit expression of $f'$ that $1-f'$ and its derivatives decay exponentially as $\tilde{\eta}$ goes to infinity. The global well-posedness of the Prandtl problem is given in the following proposition.
\begin{proposition} \label{PRAE}
Assume that $J$ satisfies (\ref{FPC2}) and (\ref{GD}). Then the Prandtl problem (\ref{PE}) admits a smooth solution $[u_p,v_p]$ satisfying
\begin{align}\label{prapro}
\begin{aligned}
&-\frac{1}{2}<u_p(\xi,\eta)<0, \quad \text{for } \eta>0,\\
&u_{p\eta}(\xi,\eta)\leqslant 0, \quad u_{p\eta\eta}(\xi,\eta)\geqslant0, \quad \text{for } \eta\geqslant0,\\
&u_{p\eta}(\xi,\eta)\leqslant -m_0<0, \quad\text{for } \eta\leqslant\eta_0,
\end{aligned}
\end{align}
where $m_0$ and $\eta_0$ are positive constants depending on $J$. And 
\begin{align}\label{pinfe}
\begin{aligned}
\Big|\p_\xi^k\p_\eta^l\big[u_p(\xi,\eta)+\frac{1}{2}\big]\Big|\leqslant C_{k,l}e^{-c_0\eta},\quad \text{for } (\xi,\eta)\in\mathbb{R}\times[0,\infty),
\end{aligned}
\end{align}
where $k$, $l$ are non-negative integers and $c_0$ is a positive constant depending on $J$. Moverover, the solution has the asymptotic behaviors as $\xi\rightarrow\pm\infty$:
\begin{align}\label{infe}
&\Big|\p_\xi^k\p_\eta^l\big[u_p(\xi,\eta)-A^+_p(\eta)\big]\Big| \leqslant C_{k,l}e^{-c_0\eta}e^{-\frac{\beta\xi}{8}},\quad \text{for } (\xi,\eta)\in[0,\infty)\times[0,\infty),\\\label{zere}
&\Big|\p_\xi^k\p_\eta^l\big[u_p(\xi,\eta)-A^-_p(\eta)\big]\Big| \leqslant C_{k,l}e^{-c_0\eta}e^{\frac{\al\xi}{8}},\quad \text{for } (\xi,\eta)\in(-\infty,0)\times[0,\infty).
\end{align}
\end{proposition}
\begin{remark}
Since $u_p<0$, $-\xi$ can be regarded as a time variable and $\eta$ as a space variable, so the first equation in (\ref{PE}) is parabolic. $A^+_p$ is the initial data for this equation, thus (\ref{infe}) can be expected if $-\frac{1}{2}\frac{J_\xi}{J}-\beta$ decays fast. $A^-_p$ can be regarded as the `large time behavior' of the solution, hence we obtain the estimate (\ref{zere}) due to the effect of the damping term $\frac{1}{2}\frac{J_\xi(\xi,-1)}{J(\xi,-1)}u^2_p$ in the equation.
\end{remark}

The proof of Proposition \ref{PRAE} can be proved by modifying the main idea of Oleinik and Samokhin \cite{Oleinik}. The main ingredients are the estimates (\ref{infe}) and (\ref{zere}) as $\xi\rightarrow\pm\infty$. We use the von Mises variables:
$$X=\xi,\quad Z=-\sqrt{2}\int_0^\eta u_p(\xi,\eta')\dd \eta',
$$ 
with 
$$
u_p=-\frac{1}{\sqrt{2}}\frac{\p Z}{\p \eta},\quad v_p=\frac{1}{\sqrt{2}}\frac{\p Z}{\p \xi},\quad Z(\xi,0)=0,
$$
and a unknown function $\omega=\big(2u_p\big)^2$. Now (\ref{PE}) is reduced to the following problem in $\mathbb{R}\times(0,\infty)$:
\begin{equation}\label{VE}
\left\{
\begin{aligned}
&\frac{\p \om}{\p X}-A\om+\sqrt{\om}\frac{\p^2 \om}{\p Z^2}+A=0,\\
&\om|_{Z=0}=0,\quad \om|_{Z\rightarrow\infty}=1,
\end{aligned}
\right.
\end{equation}
where $A=-\frac{J_\xi(X,-1)}{J(X,-1)}$ is a function depending only on $X$.

Set $\om^+(Z)=\Big(2A^+_p\big(\eta(Z)\big)\Big)^2$, where the relationship between $\eta$ amd $Z$ is given by 
$$
Z(\eta)=-\sqrt{2}\int_0^\eta A_p^+(\eta')\dd\eta'. 
$$
It is easy to check that $\om^+$ satisfies
\begin{equation}\label{VEI}
\left\{
\begin{aligned}
&-2\beta\om^++\sqrt{\om^+}\frac{\dd^2 \om^+}{\dd Z^2}+2\beta=0,\quad Z>0,\\
&\om^+|_{Z=0}=0,\quad \om^+|_{Z\rightarrow\infty}=1,
\end{aligned}
\right.
\end{equation}
and
\begin{align}
\frac{\dd \om^+}{\dd Z}>0, \quad\sqrt{\om^+}\frac{\dd^2 \om^+}{\dd Z^2}<0,\quad \text{for } Z\geqslant0.
\end{align}
Define $\om^-=\Big(2A^-_p\Big)^2$ similarly. Then $\om^-$ satisfies
\begin{equation}\label{VE0}
\left\{
\begin{aligned}
&-2\al\om^-+\sqrt{\om^+}\frac{\dd^2 \om^-}{\dd Z^2}+2\al=0,\quad Z>0,\\
&\om^-|_{Z=0}=0,\quad \om^-|_{Z\rightarrow\infty}=1,
\end{aligned}
\right.
\end{equation}
and
\begin{align}
\frac{\dd \om^-}{\dd Z}>0, \quad\sqrt{\om^-}\frac{\dd^2 \om^-}{\dd Z^2}<0,\quad \text{for } Z\geqslant0.
\end{align}
We will prove that there exists a solution to (\ref{VE}) with the initial data $\om^+$, which converges to $\om^-$ as $X\rightarrow-\infty$. This is achieved by an approximation process as follows. Consider the following problem: 
\begin{equation}\label{VEA}
\left\{
\begin{aligned}
&\frac{\p \om}{\p X}-A\om+\sqrt{\om}\frac{\p^2 \om}{\p Z^2}+A=0,\quad (X,Z)\in(-\infty,X_0)\times(0,\infty),\\
&\om|_{Z=0}=0,\quad \om|_{Z\rightarrow\infty}=1,\\
&\om|_{X=X_0}=\om_0,
\end{aligned}
\right.
\end{equation}
where $\om_0=\Big(f'\big(\sqrt{\frac{A(X_0)}{4}}\eta(Z)\big)\Big)^2$ with $Z(\eta)=\frac{\sqrt{2}}{2}\int_0^\eta f'\big(\sqrt{\frac{A(X_0)}{4}}\eta'\big)\dd\eta'$. $\om_0$ is a non-negative solution to the following problem:
\begin{equation}\label{VEINI}
\left\{
\begin{aligned}
&-A(X_0)\om_0+\sqrt{\om_0}\frac{\dd^2 \om_0}{\dd Z^2}+A(X_0)=0,\quad Z>0,\\
&\om_0|_{Z=0}=0,\quad \om_0|_{Z\rightarrow\infty}=1.
\end{aligned}
\right.
\end{equation}
If $X_0$ is large enough, then $A(X_0)$ is close to $2\beta$, it follows that $\om_0$ is close to $\om^+$. The existences of the global solution to (\ref{VEA}) is studied in the following lemma.
\begin{lemma}\label{Ole}
Assume that $A(X)=-\frac{J_\xi(X,-1)}{J(X,-1)}$ satisfies (\ref{FPC2}) and (\ref{GD}). Then the problem (\ref{VEA}) adimits a solution $\om$ satisfying 
\begin{align}\label{Olee}
\begin{aligned}
&m_1\om^+\leqslant \om\leqslant M_1 \om^+, \quad \text{for } (X,Z)\in(-\infty,X_0]\times[0,\infty),\\
&\frac{\p \om}{\p Z}\geqslant m_2>0,\quad \text{for } 0\leqslant Z\leqslant Z^*,\\
&\Big|\frac{\p \om}{\p Z}\Big|\leqslant M_2,\hspace{2mm}\Big|\sqrt{\om}\frac{\p^2 \om}{\p Z^2}\Big|\leqslant M_3,\quad\text{for } (X,Z)\in(-\infty,X_0]\times[0,\infty),\\
&\Big|\om^{-\frac{3}{4}}\frac{\p \om}{\p X}\Big|\leqslant M_4,\quad \text{for } (X,Z)\in(-\infty,X_0]\times[0,\infty),
\end{aligned}
\end{align}
where the positive constants $m_k$, $M_k$ and $Z^*$ above depend only on $J$. 
\end{lemma}
\begin{proof}
The proof follows the same arguments of Oleinik and Samokhin in \cite{Oleinik}, which is just sketched here. Since $\omega^+$ is a solution to (\ref{VEI}) with $\frac{\dd^2 \om^+}{\dd Z}<0$, by the maximum principle, one can show $m_1\om^+\leqslant \om\leqslant M_1 \om^+$ for $m_1>0$ small enough and $M_1>0$ large enough. $m_1\om^+\leqslant \om$ implies $\frac{\p \om}{\p Z}|_{Z=0}\geqslant m_1 \frac{\dd \om^+}{\dd Z}|_{Z=0}>0$.

Set $h:=\frac{\p \om}{\p Z}$. Then 
\[
\frac{\p h}{\p X}+\sqrt{\om}\frac{\p^2 h}{\p Z^2}-Ah+\frac{h}{2\sqrt{\om}}h_Z=0,
\]
so $\frac{\p \om}{\p Z}$ is bounded by the maximum principle. 

Rewrite the equation for $\om$ in the divergence form:
\[
\frac{\p \om}{\p X}+\frac{\p}{\p Z}\big[\sqrt{\om}\frac{\p \om}{\p Z}\big]-\frac{h}{2\sqrt{\om}}\frac{\p \om}{\p Z}-A\om+A=0.
\]
Since $\sqrt{\om}$ is bounded below and $\frac{h}{2\sqrt{\om}}$ is bounded in the region $Z\geqslant Z_*$ for any $Z_*>0$, by the standard Nash-Moser-De Giorgi theory, there exists a $\gamma\in(0,1)$, such that $\om\in C^\gamma\big((-\infty,X_0]\times[Z_*,\infty)\big)$. Furthermore, by the Schauder theory for parabolic equations, $\|\om\|_{C^{2,\gamma}\big((-\infty,X_0]\times[Z_*,\infty)\big)}\leqslant M(Z_*)$. Thus, it suffices to estimate $\frac{\p \om}{\p X}$ and $\sqrt{\om}\frac{\p^2 \om}{\p Z^2}$ in the region $0\leqslant Z\leqslant Z_*$ for $Z_*$ small enough.

Since $r:=\frac{\p \om}{\p X}$ satisfies 
\[
\frac{\p r}{\p X}+\sqrt{\om}\frac{\p^2 r}{\p Z^2}-\frac{A}{2\om}r-\frac{3A}{2}r=\frac{r^2}{2\om}-A_X\big(1-\om\big)\geqslant-A_X\big(1-\om\big),
\]
by the maximum principle, $r\leqslant M$ for $M$ large enough. It follows from (\ref{VEA}) that $\sqrt{\om}\frac{\p ^2 \om}{\p Z^2}\geqslant -M'$ for large $M'$. And $\frac{\p  \om}{\p Z}-\frac{\p  \om}{\p Z}\big|_{Z=0}=\int_0^Z \frac{\p ^2 \om}{\p Z^2}\geqslant -M'\sqrt{Z}$ implies that $\frac{\p  \om}{\p Z}$ is strictly positive when $Z$ is small.

Set $\tilde{q}=\sqrt{\om}\frac{\p^2 \om}{\p Z^2}/(h-m')$, where $m'$ is a small positive constant. Then $\tilde{q}$ satisfies the following equation in $(-\infty,X_0)\times(0, Z^*)$ for small $Z^*$:
\begin{align}\nonumber
\begin{aligned}
&(h-m')\frac{\p \tilde{q}}{\p X}+(h-m')\sqrt{\om}\frac{\p^2 \tilde{q}}{\p Z^2}+2\sqrt{\om}h_Z\frac{\p \tilde{q}}{\p Z}\\
&-\frac{1}{2\om}\big(m'(h-m')\tilde{q}-2m'A\om-A(h-m')(1-\om)\big)\tilde{q}=0.
\end{aligned}
\end{align}
If $\tilde{q}$ does not attain its maximum on the parabolic boundary of $(-\infty,X_0)\times(0, Z^*)$, then at the maximum point, it holds, 
$$\big(m'(h-m')\tilde{q}-2m'A\om-A(h-m')(1-\om)\big)\tilde{q}\leqslant 0.$$
It yields $\sqrt{\om}\frac{\p ^2 \om}{\p Z^2}\leqslant M''$ in  $(-\infty,X_0)\times(0, Z^*)$. Due to (\ref{VEA}), it holds that $\frac{\p \om}{\p X}\geqslant M'''$ in  $(-\infty,X_0)\times(0, Z^*)$. 

Set $\tilde{g}=\om^{\frac{1}{4}}$ and $\tilde{r}=\frac{\p \tilde{g}}{\p X}$. Then 
\begin{align}\nonumber
\begin{aligned}
&4\tilde{g}^3\frac{\p \tilde{r}}{\p X}+4\tilde{g}^5\frac{\p^2 \tilde{r}}{\p Z^2}+24\tilde{g}^4\tilde{g}_Z\frac{\p \tilde{r}}{\p Z}+\big(-8\tilde{g}^2\tilde{r}+A\tilde{g}^3-5A/\tilde{g}-12\tilde{g}^3\tilde{g}_Z^2\big)\tilde{r}=A_X\tilde{g}^4-A_X.
\end{aligned}
\end{align}
The coefficient of $\tilde{r}$ in the equation above is 
\begin{align}\nonumber
\begin{aligned}
&-8\tilde{g}^2\tilde{r}+A\tilde{g}^3-5A/\tilde{g}-12\tilde{g}^3\tilde{g}_Z^2\\
=&\om^{-\frac{1}{4}}\big(-2\om_X+A\om-5A-\frac{3}{4}\om^{-\frac{1}{2}}h^2\big)\\
\leqslant& -m''\om^{-\frac{3}{4}}
\end{aligned}
\end{align}
for small $Z$, hence by the maximum principle, $\tilde{r}$ is bounded when $Z$ is small. 
\end{proof}

Base on (\ref{Olee}), we will obtain some estimates on $\om-\om^+$ for large $X_0$, and then take the limit $X_0\rightarrow\infty$. In fact, we have the following lemma.
\begin{lemma}\label{Gaoi}
Assume that $A(X)=-\frac{J_\xi(X,-1)}{J(X,-1)}$ satisfies (\ref{FPC2}) and (\ref{GD}). Then it holds in $(-\infty,X_0]\times[0,\infty)$ that
\begin{align}\label{Olee2}
\begin{aligned}
&\frac{\p \om}{\p Z}\geqslant0,\hspace{2mm}-\sqrt{\om}\frac{\p^2 \om}{\p Z^2}\geqslant0,\\
&\Big|\om^{-1}\frac{\p \om}{\p X}\Big|\leqslant M_5,\\
&0\leqslant1-\om\leqslant M_6e^{-\smu Z},\hspace{2mm} \frac{\p \om}{\p Z}\leqslant M_7e^{-\smu Z},\\
&\Big|\frac{\p \om}{\p X}\Big|\leqslant M_8e^{-\smu Z},\hspace{2mm} -\sqrt{\om}\frac{\p^2 \om}{\p Z^2}\leqslant M_9e^{-\smu Z},
\end{aligned}
\end{align}
and there exists a constant $M_{10}$ such that it holds in $[M_{10},X_0]\times[0,\infty)$,
\begin{align}\label{Gaoe}
\begin{aligned}
&\big| \om-\om_0\big|\leqslant M_{11} e^{-\smu Z}e^{-\frac{\beta X}{8}}\om_0,\hspace{2mm}\Big| \frac{\p \om}{\p X}\Big|\leqslant M_{12} e^{-\smu Z}e^{-\frac{\beta X}{8}}\om_0,\\
&\Big| \frac{\p \om}{\p Z}-\frac{\p \om_0}{\p Z}\Big|\leqslant M_{13} e^{-\smu Z}e^{-\frac{\beta X}{8}},\\
&\Big|\sqrt{\om_0}\Big[\frac{\p^2\om}{\p Z^2}-\frac{\p^2 \om_0}{\p Z^2}\Big]\Big|\leqslant M_{14} e^{-\smu Z}e^{-\frac{\beta X}{8}},
\end{aligned}
\end{align}
where the above constants $M_k$ depend only on $J$. 
\end{lemma}
\begin{proof}
We first prove (\ref{Olee2}). Differentiating the equation in (\ref{VEA}) with respect to $Z$, and setting $h=\frac{\p \om}{\p Z}$, one gets
\[
\frac{\p h}{\p X}+\sqrt{\om}\frac{\p^2 h}{\p Z^2}-Ah+\frac{h}{2\sqrt{\om}}h_Z=0.
\]
For $Z\geqslant Z^*$, $\frac{h}{2\sqrt{\om}}$ is bounded because $\om$ is bounded below. Considering $B_1=\frac{M_3}{Z_1}(1+Z)e^{-c(X-X_0)}$ in the domain $(-\infty,X_0)\times(Z^*,Z_1)$, one has
\[
\begin{split}
\ml_1(B_1):&=\frac{\p B_1}{\p X}+\sqrt{\om}\frac{\p^2 B_1}{\p Z^2}-AB_1+\frac{h}{2\sqrt{\om}}B_{1Z}\\
&=-AB_1+\Big(-c(1+Z)+\frac{h}{2\sqrt{\om}}\Big)\frac{M_3}{Z_1}e^{-c(X-X_0)}<0,
\end{split}
\]
provided that $c$ is large enough. Thus, $\ml_1(B_1+h)<0$ with $B_1+h\geqslant 0$ on $Z=Z^*$ and $Z=Z_1$. And $B_1+h>0$ follows from $h|_{X=X_0}>0$. Therefore the maximum principle implies that $B_1+h\geqslant0$ in $(-\infty,X_0)\times(Z^*,Z_1)$. Passing the limit $Z_1\rightarrow\infty$ yields $h\geqslant0$ in $(-\infty,X_0]\times[Z^*,\infty)$, which implies $\frac{\p \om}{\p Z}\geqslant0$.

Differentiating the equation in (\ref{VEA}) with respect to $X$ leads to
\[
\frac{\p r}{\p X}+\sqrt{\om}\frac{\p^2 r}{\p Z^2}-\frac{A+r}{2\om}r-\frac{3A}{2}r=-A_X\big(1-\om\big),
\]
where $r=\frac{\dd \om}{\dd Z}$. Since $A\geqslant 2\mu$ and $|r|\leqslant M_4\om^\frac{3}{4}\leqslant CZ^\frac{3}{4}$, thus $A+r\geqslant\mu$ for $0\leqslant Z\leqslant Z_2$ provided that $Z_2$ is small enough. Considering $B_2=M_5\om^+$ in the domain $(-\infty,X_0)\times(0,Z_2)$, one gets
\[
\begin{split}
\ml_2(B_2):&=\frac{\p B_2}{\p X}+\sqrt{\om}\frac{\p^2 B_2}{\p Z^2}-\frac{A+r}{2\om}B_2-\frac{3A}{2}B_2\\
&\leqslant-\mu M_5\frac{\om^+}{\om}\leqslant-\frac{\mu M_5}{M_1},\hspace{1mm}
\end{split}
\]
where $\sqrt{\om^+}\frac{\dd^2 \om^+}{\dd Z^2}\leqslant 0$ has been used. Choose $M_5$ large enough so that $\ml_2\big(B_2\pm r\big)<0$ and $B_2\pm r\geqslant0$ on $Z=0$ and $Z=Z_1$. It follows from the compatibility condition on $X=X_0$ that $r|_{X=X_0}=0$, hence the maximum principle implies that $|r|\leqslant M_5\om^+$ in $(-\infty,X_0]\times[0,Z_2]$.

Differentiating the equation in (\ref{VEA}) twice with respect to $Z$ and setting $p:=\sqrt{\om}\frac{\p^2 \om}{\p Z^2}$ give
\[
\frac{\p p}{\p X}+\sqrt{\om}\frac{\p^2 p}{\p Z^2}-Ap-\frac{\om_X}{2\om}p=0,\quad X<X_0,\hspace{1mm} Z>0.
\]
Let $P=e^{M_5X}p$. Then $P$ satisfies

\[
\frac{\p P}{\p X}+\sqrt{\om}\frac{\p^2 P}{\p Z^2}-AP-\Big[M_5+\frac{\om_X}{2\om}\Big]P=0, \quad X<X_0,\hspace{1mm} Z>0.
\]
Thus, the maximum principle shows $P\leqslant0$.

Note that $\om_0$ solves the following problem
\begin{equation}\nonumber
\left\{
\begin{aligned}
&-A(X_0)\om_0+\sqrt{\om_0}\frac{\dd^2 \om_0}{\dd Z^2}+A(X_0)=0,\quad Z>0,\\
&\om_0|_{Z=0}=0,\quad \om_0|_{Z\rightarrow\infty}=1,
\end{aligned}
\right.
\end{equation}
and $\frac{\dd w_0}{\dd Z}>0$, $\sqrt{w_0}\frac{\dd^2 w_0}{\dd Z^2}<0$.
$$\ml_3(e^{-\smu Z}):=\big[\sqrt{\om_0}\frac{\dd^2}{\dd Z^2}-A(X_0)\big]e^{-\smu Z}=[\mu\sqrt{\om_0}-A(X_0)]e^{-\mu Z}<0$$
since $A\geqslant2\mu$. Hence the maximum principle shows $1-\om_0\leqslant e^{-\smu Z}$. It follows from the equation for $\om_0$ that $\sqrt{\om_0}\frac{\dd^2\om_0}{\dd Z^2}\geqslant -C_1e^{-\smu Z}$, and thus $\frac{\dd \om_0}{\dd Z}=-\int_Z^\infty\frac{\dd^2 \om_0}{\dd Z^2}(Z')\dd Z'\leqslant C_2e^{-\smu Z}$.

It is easy to see $1-\om\leqslant M_6e^{-\smu Z}$ by the maximum principle. And 
$$\ml_1(M_7e^{-\smu Z})=\Big[\frac{\p }{\p X}+\sqrt{\om}\frac{\p^2 }{\p Z^2}-A+\frac{\om_Z}{2\sqrt{\om}}\frac{\p}{\p Z}\Big]\Big(M_7e^{-\smu Z}\Big)<0$$
implies $\frac{\p \om}{\p Z}\leqslant M_7e^{-\smu Z}$ for large $M_7$. Since
\[
\begin{split}
\ml_4(r):&=\frac{\p r}{\p X}+\sqrt{\om}\frac{\p^2 r}{\p Z^2}-\frac{A}{2\om}r-\frac{3A}{2}r=\frac{r^2}{2\om}-A_X(1-\om)\\
&\geqslant -A_X(1-\om)\geqslant Ce^{-\smu Z},
\end{split}
\] 
thus $\ml_4(M_8e^{-\smu Z}-r)<0$ for $M_8$ large enough, hence $r\leqslant M_8e^{-\smu Z}$. These and (\ref{VEA}) imply that $\frac{\p\om}{\p X}=-\sqrt{\om}\frac{\p^2 \om}{\p Z^2}+A(\om-1)\geqslant A(\om-1)\geqslant -C_3M_6e^{-\smu Z}$, and hence $\big|\frac{\p\om}{\p X}\big|\leqslant M_8e^{-\smu Z}$. The estimate of $\sqrt{\om}\frac{\p^2 \om}{\p Z^2}$ in (\ref{Olee2}) follows from (\ref{VEA}). And $\frac{\p\om}{\p Z}(Z)=-\int_Z^\infty\frac{\p^2\om}{\p Z^2}(Z')\dd Z'$ leads to the estimate for $\frac{\p\om}{\p Z}$. So (\ref{Olee2}) is proved.

Next we will prove (\ref{Gaoe}). Let $\om_1=\om-\om_0$, which satisfies
\begin{align}\label{omor}
\frac{\p\om_1}{\p X}+\sqrt{\om}\frac{\p^2 \om_1}{\p Z^2}-A\om_1+\frac{\om_1}{\sqrt{\om}+\sqrt{\om_0}}\frac{\dd^2 \om_0}{\dd Z^2}=\big(A(X)-A(X_0)\big)\big(\om_0-1\big).
\end{align}
Note that $\frac{\dd^2 \om_0}{\dd Z^2}<0$ and $\frac{\p\om}{\p Z}\geqslant0$, and set $\ml_5:=\frac{\p}{\p X}+\sqrt{\om}\frac{\p^2}{\p Z^2}-A+\frac{1}{\sqrt{\om}+\sqrt{\om_0}}\frac{\dd^2 \om_0}{\dd Z^2}$. Then 
\begin{align}\nonumber
\begin{aligned}
\ml_5(M_{11}e^{-\smu Z}e^{-\frac{\beta}{8}X}\om_0)=&\Big(-\frac{\beta}{8}-A+\mu\sqrt{\om}\Big)M_{11}e^{-\smu Z}e^{-\frac{\beta}{8}X}\om_0\\
&+\frac{1}{\sqrt{\om}+\sqrt{\om_0}}\frac{\dd^2 \om_0}{\dd Z^2}M_{11}e^{-\smu Z}e^{-\frac{\beta}{8}X}\om_0\\
&+M_{11}e^{-\smu Z}e^{-\frac{\beta}{8}X}\sqrt{\om}\Big(-\smu\frac{\dd \om_0}{\dd Z}+\frac{\dd^2 \om_0}{\dd Z^2}\Big)\\
&\leqslant-\mu M_{11}e^{-\smu Z}e^{-\frac{\beta}{8}X}\om_0+M_{11}e^{-\smu Z}e^{-\frac{\beta}{8}X}\sqrt{\om}\frac{\dd^2 \om_0}{\dd Z^2}\\
&=-\mu M_{11}e^{-\smu Z}e^{-\frac{\beta}{8}X}\om_0-M_{11}e^{-\smu Z}e^{-\frac{\beta}{8}X}\sqrt{\frac{\om}{\om_0}}A(X_0)\big(1-\om_0\big)\\
&\leqslant -m_3 M_{11}e^{-\smu Z}e^{-\frac{\beta}{8}X},
\end{aligned}
\end{align}
where $m_3\leqslant\min\{\mu,\sqrt{\frac{\om}{\om_0}}A(X_0)\}$ is a positive constant. (\ref{GD}) implies that $\big|A(X)-2\beta\big|\leqslant Ce^{-\frac{\beta}{4}X}$ and $\big|A(X_0)-2\beta\big|\leqslant Ce^{-\frac{\beta}{4}X_0}$ for large $X$ and $X_0$. So $\ml_5(M_{11}e^{-\smu Z}e^{-\frac{\beta}{8}X}\om_0\pm\om_1)<0$ for large enough $M_{11}$. Therefore the maximum principle shows $|\om_1|\leqslant M_{11}e^{-\smu Z}e^{-\frac{\beta}{8}X}\om_0$ in $[M_{10},X_0]\times[0,\infty)$.

Now we show the estimate of $r=\frac{\p \om}{\p X}$ in (\ref{Gaoe}). Recall that
\[
\frac{\p r}{\p X}+\sqrt{\om}\frac{\p^2 r}{\p Z^2}-\Big(\frac{A+r}{2\om}+\frac{3A}{2}\Big)r=-A_X\big(1-\om\big).
\]
Since $r=-\sqrt{\om}\frac{\p^2\om}{\p Z^2}+A\om-A$, it holds that 
\[
\ml_2(r)=\frac{\p r}{\p X}+\sqrt{\om}\frac{\p^2 r}{\p Z^2}-\Big(-\frac{1}{2\sqrt{\om}}\frac{\p^2\om}{\p Z^2}+2A\Big)r=-A_X\big(1-\om\big).
\]
Due to $-\frac{\p^2\om}{\p Z^2}\geqslant0$ and $\frac{\p\om}{\p Z}\geqslant0$, one can get
\begin{align}\nonumber
\begin{aligned}
\ml_2(M_{12}e^{-\smu Z}e^{-\frac{\beta}{8}X}\om_0)=&\Big(-\frac{\beta}{8}-2A+\mu\sqrt{\om}+\frac{1}{2\sqrt{\om}}\frac{\p^2\om}{\p Z^2}\Big)M_{12}e^{-\smu Z}e^{-\frac{\beta}{8}X}\om_0\\
&+M_{12}e^{-\smu Z}e^{-\frac{\beta}{8}X}\sqrt{\om}\Big(-\smu\frac{\dd \om_0}{\dd Z}+\frac{\dd^2 \om_0}{\dd Z^2}\Big)\\
&\leqslant-\mu M_{12}e^{-\smu Z}e^{-\frac{\beta}{8}X}\om_0+M_{12}e^{-\smu Z}e^{-\frac{\beta}{8}X}\sqrt{\om}\frac{\dd^2 \om_0}{\dd Z^2}\\
&=-\mu M_{12}e^{-\smu Z}e^{-\frac{\beta}{8}X}\om_0-M_{12}e^{-\smu Z}e^{-\frac{\beta}{8}X}\sqrt{\frac{\om}{\om_0}}A(X_0)\big(1-\om_0\big)\\
&\leqslant -m_3 M_{12}e^{-\smu Z}e^{-\frac{\beta}{8}X}.
\end{aligned}
\end{align}
So $\ml_2(M_{12}e^{-\smu Z}e^{-\frac{\beta}{8}X}\om_0\pm r)<0$ for large enough $M_{12}$. Due to the compatibility conditions on $X=X_0$, it holds that $r|_{X=X_0}=0$. Thus the maximum principle implies that $|r|\leqslant M_{12}e^{-\smu Z}e^{-\frac{\beta}{8}X}\om_0$ in $[M_{10},X_0]\times[0,\infty)$.

Finially, one can estimate $\sqrt{\om}\frac{\p^2 \om_1}{\p Z^2}$ in (\ref{Gaoe}) using (\ref{omor}). And the estimate of $\frac{\p\om_1}{\p Z}$ follows from $\frac{\p\om_1}{\p Z}(Z)=-\int_Z^\infty\frac{\p^2\om_1}{\p Z^2}(Z')\dd Z'$. So Lemma \ref{Gaoi} is proved.
\end{proof}

Similar analysis at $-\infty$ can be derived to get the following lemma.
\begin{lemma}\label{Gaon}
Assume that $A(X)=-\frac{J_\xi(X,-1)}{J(X,-1)}$ satisfies (\ref{FPC}) and (\ref{GD}). Then there exists a constant $M_{15}$ such that for $(X,Z)\in(-\infty,-M_{15}]\times[0,\infty)$, it holds that
\begin{align}\label{Gaoe2}
\begin{aligned}
&\big| \om-\om^-\big|\leqslant M_{16} e^{-\smu Z}e^{\frac{\al X}{8}}\om^-,\hspace{2mm}\Big| \frac{\p \om}{\p X}\Big|\leqslant M_{17} e^{-\smu Z}e^{\frac{\al X}{8}}\om^-,\\
&\Big| \frac{\p \om}{\p Z}-\frac{\p \om^-}{\p Z}\Big|\leqslant M_{18} e^{-\smu Z}e^{\frac{\al X}{8}},\\
&\Big|\sqrt{\om^-}\Big[\frac{\p^2\om}{\p Z^2}-\frac{\p^2 \om^-}{\p Z^2}\Big]\Big|\leqslant M_{19} e^{-\smu Z}e^{\frac{\al X}{8}},
\end{aligned}
\end{align}
where the above constants $M_k$ depend only on $J$. 
\end{lemma}

{\bf Proof of Proposition \ref{PRAE}:}
Taking the limit $X_0\rightarrow\infty$ in (\ref{VEA}), we can get a solution, $\om$, of the following Prandtl problem: 
\begin{equation}\label{VEFIN}
\left\{
\begin{aligned}
&\frac{\p \om}{\p X}-A\om+\sqrt{\om}\frac{\p^2 \om}{\p Z^2}+A=0,\quad (X,Z)\in(-\infty,\infty)\times(0,\infty),\\
&\om|_{Z=0}=0,\quad \om|_{Z\rightarrow\infty}=1,\\
&\om|_{X\rightarrow\infty}=\om^+,
\end{aligned}
\right.
\end{equation}
and $\om$ satisfies the following estimates
\begin{align}\label{Gaore1}
\begin{aligned}
&\big| \om-\om^+\big|\leqslant M_{11} e^{-\smu Z}e^{-\frac{\beta X}{8}}\om^+,\hspace{2mm}\Big| \frac{\p \om}{\p X}\Big|\leqslant M_{12} e^{-\smu Z}e^{-\frac{\beta X}{8}}\om^+,\\
&\Big| \frac{\p \om}{\p Z}-\frac{\p \om^+}{\p Z}\Big|\leqslant M_{13} e^{-\smu Z}e^{-\frac{\beta X}{8}},\\
&\Big|\sqrt{\om^+}\Big[\frac{\p^2\om}{\p Z^2}-\frac{\p^2 \om^+}{\p Z^2}\Big]\Big|\leqslant M_{14} e^{-\smu Z}e^{-\frac{\beta X}{8}}
\end{aligned}
\end{align}
in $[M_{10},\infty)\times[0,\infty)$,  and (\ref{Gaoe2}) in $(-\infty,M_{15}]\times[0,\infty)$. Furthermore, similar estimates for high-order derivatives hold. In terms of $[u_p,v_p]$ and the coordinate $(\xi,\eta)$, we have 
$$u_p=-\frac{1}{2}\sqrt{\om},\quad \eta(X,Z)=\sqrt{2}\int_0^Z\frac{1}{\sqrt{\om(X,Z')}}\dd Z'.$$
Direct computations give that
\begin{align}\nonumber
\begin{aligned}
&v_p=-\frac{\sqrt{2}}{4}\int_0^Z\frac{\om_X}{\om^{\frac{3}{2}}},\\
&\frac{\p u_p}{\p \eta}=-\frac{\sqrt{2}}{8}\frac{\p\om}{\p Z},\quad \frac{\p^2 u_p}{\p \eta^2}=-\frac{1}{8}\sqrt{\om}\frac{\p^2 \om}{\p Z^2},\\
&\frac{\p u_p}{\p \xi}=-\frac{1}{4\sqrt{\om}}\frac{\p\om}{\p X}+\frac{\sqrt{2}}{8}\frac{\p\eta}{\p X}\frac{\p \om}{\p Z}.
\end{aligned}
\end{align}
All the estimates in Proposition \ref{PRAE} follows from (\ref{Olee2}), (\ref{Gaoe2}) and (\ref{Gaore1}). Thus Proposition \ref{PRAE} is proved.
\qed

\section{The construction of a high-order Prandtl expansion}\label{high-order}
We construct a high-order Prandtl expansion in this section. By the Prandtl's viscous boundary layer theory, the approximate solution of the Navier-Stokes equation (\ref{NSP}) can take the following form:
\begin{align}\nonumber
\begin{aligned}
\Phi_s\approx&\underset{m=0}{\overset{n}{\sum}}\Big(\eq\Big)^m\Phi_e^m(\xi,\psi)+\underset{m=0}{\overset{n}{\sum}}\Big(\eq\Big)^{m+1} Q\Phi_b^m\big(\xi,\qe(\psi+1)\big)\\
&-\underset{m=0}{\overset{n}{\sum}}\Big(\eq\Big)^{m+1}Q\Phi_b^{m}\big(\xi,\qe(1-\psi)\big),
\end{aligned}
\end{align}
where $n=12$, $\big(\xi,\qe(\psi+1)\big)$ and $\big(\xi,\qe(1-\psi)\big)$ are the Prandtl variables near $\psi=-1$ and $\psi=1$. 

$\Phi_e^0$ is a stream-function of the leading-order Euler flow, namely, $\Phi_e^0=\Phi_e=\frac{Q}{2}\psi$. Let $u_b^0(\xi,\eta)=u_p(\xi,\eta)+\frac{1}{2}$ and $\Phi_b^0(\xi,\eta)=\int_\eta^\infty u_b^0(\xi,\eta')\dd\eta'$. The equations for $\Phi_e^m$ and $\Phi_b^m$ can be derived by the standard method of matched asymptotic expansions. We start with the case $m=1$. 

Away from the boundaries $\psi=\pm1$, $\Phi_s\approx \underset{k=0}{\overset{n}{\sum}}\Big(\eq\Big)^m\Phi_e^m(\xi,\psi)$. Recall the equation in (\ref{NSP}). Replacing $\Phi$ by $\Phi_s$ and comparing the order $\mathcal{O}(\eq)$ in the equation in (\ref{NSP}), one finds
$$-\Delta_{\xi,\psi}\hspace{1mm}\Phi_e^1=0.$$
To match the boundary conditions on $\psi=\pm1$, one requires
$$
\Phi_e^1|_{\psi=\pm1}=\pm Q\Phi_b^0(\xi,0)=\pm Q\int_0^\infty u_b^0(\xi,\eta)\dd\eta.
$$
Since $\underset{\xi\rightarrow\pm\infty}{\lim}\frac{\p^2 \Phi_e^1}{\p \psi^2}=0$, thus the following asymptotic behavior is required,
$$
\underset{\xi\rightarrow\pm\infty}{\lim}\Phi_e^1=QI^{0,\pm}\psi,
$$
where $I^{0,\pm}=\underset{\xi\rightarrow\pm\infty}{\lim}\Phi_b^0(\xi,0)$ are two constants. Thus $\Phi_e^1$ should be a solution to the following problem in $(-\infty,\infty)\times(-1,1)$:
\begin{equation}\label{Euler1}
\left\{
\begin{aligned}
&-\Delta_{\xi,\psi}\hspace{1mm}\Phi_e^1=0,\\
&\Phi_e^1|_{\psi=\pm1}=\pm Q\Phi_b^0(\xi,0),\\
&\underset{\xi\rightarrow\pm\infty}{\lim}\Phi_e^1=QI^{0,\pm}\psi.
\end{aligned}
\right.
\end{equation}
It is easy to see that $|\Phi_b^0(\xi,0)-I^{0,+}|\leqslant Ce^{-\frac{\beta\xi}{8}}$ and $|\Phi_b^0(\xi,0)-I^{0,-}|\leqslant Ce^{\frac{\al\xi}{8}}$ due to (\ref{infe}) and (\ref{zere}). The standard theory for elliptic equations shows the following lemma.
\begin{lemma}
(\ref{Euler1}) admits a solution $\Phi_e^1$ satisfying 
\begin{align}\label{Euler1e}
\begin{aligned}
\int_{-\infty}^\infty\int_{-1}^1\bigg|\p^k_\xi\p^l_\psi\Big[\Phi_e^1(\xi,\psi)-Q\Phi_b^0(\xi,0)\psi\Big]\bigg|^2\cosh\Big(\frac{\lambda\xi}{8}\Big)\dd\psi\dd\xi\leqslant C_{k,l}(\lambda)Q^2,
\end{aligned}
\end{align}
where $0<\lambda<2\al$ is a constant.
\end{lemma}
In the following, we choose $\lambda=\mu\leqslant\al$.

Next we will construct the boundary layer profile $\Phi_b^1$. Only the case $\psi$ near $-1$ will be dealt with since the case near $\psi=1$ is similar. Set $\tilde{\Phi}_s(\xi,\eta)=\frac{1}{\sqrt{\e Q}}\Phi_s(\xi,\psi)$, where $\eta=\qe(\psi+1)$. Near $\psi=-1$, $\tilde{\Phi}_s(\xi,\eta)\approx \underset{m=0}{\overset{n}{\sum}}\Big(\eq\Big)^{m-1}\frac{1}{Q}\Phi_e^m\big(\xi,-1+\eq\eta\big)+\underset{m=0}{\overset{n}{\sum}}\Big(\eq\Big)^m\Phi_b^m(\xi,\eta)$.  Recall the Navier-Stokes system in the Prandtl variable, (\ref{NSPP}). In (\ref{NSPP}), replacing $\tilde{\Phi}$ by $\tilde{\Phi}_s$ and comparing the order $\mathcal{O}(\eq)$ lead to
\begin{align}\label{lipra}
\begin{aligned}
u_pu_{b\xi\eta}^1+u_b^1u_{p\xi\eta}-Au_pu_{b\eta}^1-Au_b^1u_{p\eta}+v_pu_{b\eta\eta}^1+\big[v_b^1-v_b^1|_{\eta=0}\big]u_{p\eta\eta}-u_{b\eta\eta\eta}^1=f^1_b,
\end{aligned}
\end{align}
where $u_b^1=-\Phi_{b\eta}^1$, $v_b^1-v_b^1|_{\eta=0}=-\int_0^\eta u_{b\xi}^1(\xi,\eta')\dd\eta'$, and $f^1_b$ is the force term:
\begin{align}\nonumber
\begin{aligned}
f^1_b=&\frac{1}{Q}\p_\psi\Phi_{e}^1|_{\psi=-1} u_{b\xi\eta}^0-\frac{A}{Q}\p_\psi\Phi_{e}^1|_{\psi=-1}u_{b\eta}^0-\p_\psi\Big[\frac{J_\xi}{J}\Big]\Big|_{\psi=-1}\eta u_pu_{b\eta}^0\\
      &-\frac{J_{\psi}}{J}\Big|_{\psi=-1}v_pu_{b\eta}^0+2\frac{J_{\psi}}{J}\Big|_{\psi=-1}u_{b\eta\eta}^0.
\end{aligned}
\end{align}
$u_b^1$ should satisfy the boundary conditions $u_b^1|_{\eta=0}=\frac{1}{Q}\Phi_{e\psi}^1|_{\psi=-1}$, and $\underset{\eta\rightarrow\infty}{\lim}u_b^1=0$. Rewrite (\ref{lipra}) as
\begin{align}\nonumber
\begin{aligned}
\p_\eta\Big[u_pu_{b\xi}^1+u_b^1u_{p\xi}-Au_pu_{b}^1+v_pu_{b\eta}^1+\big[v_b^1-v_b^1|_{\eta=0}\big]u_{p\eta}-u_{b\eta\eta}^1\Big]=f^1_b.
\end{aligned}
\end{align}
Since $f^1_b$ decays fast to $0$ as $\eta\rightarrow\infty$, one has derived the following problem for $[u_b^1,v_b^1]$.
\begin{align}\label{brandtl1}
\left\{
\begin{aligned}
&u_pu_{b\xi}^1+u_b^1u_{p\xi}-Au_pu_{b}^1+v_pu_{b\eta}^1+\big[v_b^1-v_b^1|_{\eta=0}\big]u_{p\eta}-u_{b\eta\eta}^1=F^1_b,\\
&u_{b\xi}^1+v_{b\eta}^1=0,\\
&u_b^1|_{\eta=0}=\frac{1}{Q}\Phi_{e\psi}^1|_{\psi=-1},\quad \underset{\eta\rightarrow\infty}{\lim}u_b^1=0,
\end{aligned}
\right.
\end{align}
where $F^1_b=-\int_\eta^\infty f^1_b(\xi,\eta')\dd \eta'$. The equation in (\ref{brandtl1}) is a linear parabolic equation for $u_b^1$ when $-\xi$ is regarded as a time variable and $\eta$ as a space variable.

Now we discuss the asymptotic behavior $\big[A_b^{1,\pm}(\eta),0\big]:=\underset{\xi\rightarrow\pm\infty}{\lim}\big[u_b^1(\xi,\eta),v_b^1(\xi,\eta)\big]$. Due to (\ref{asybJ2}), it holds that
\begin{align}\label{GD2}
\begin{aligned}
&\bigg|\p_\xi^k\p_\psi^l\Big[\frac{J_\psi(\xi,\psi)}{J(\xi,\psi)}\Big]\bigg|\leqslant C_{k,j}e^{-\frac{\beta}{4}\xi} \quad\text{for } \xi\geqslant0,\\
&\bigg|\p_\xi^k\p_\psi^l\Big[\frac{J_\psi(\xi,\psi)}{J(\xi,\psi)}\Big]\bigg|\leqslant C_{k,j}e^{\frac{\al}{4}\xi} \quad\text{for } \xi<0,
\end{aligned}
\end{align}
where $k$, $l$ are non-negative integers. It follows that $\underset{\xi\rightarrow\pm\infty}{\lim}\frac{J_\psi(\xi,\psi)}{J(\xi,\psi)}=0$. And $\underset{\xi\rightarrow\pm\infty}{\lim}\p_\psi\Big[\frac{J_\xi}{J}\Big]=0$ due to (\ref{asybJ1}). It follows from (\ref{infe}) and (\ref{zere}) that $\underset{\xi\rightarrow\pm\infty}{\lim}[u_p,v_p]=[A^\pm_p,0]$. Taking the limits as $\xi\rightarrow\pm\infty$ in the equation in (\ref{brandtl1}), one gets that
\begin{align}\label{brainf1}
\begin{aligned}
-2\beta A^+_pA_{b}^{1,+}-A_{b\eta\eta}^{1,+}=-2\beta I^{0,+}A_b^{0,+},
\end{aligned}
\end{align} 
\begin{align}\label{braminf1}
\begin{aligned}
-2\al A^-_pA_{b}^{1,-}-A_{b\eta\eta}^{1,-}=-2\al I^{0,-}A_b^{0,-},
\end{aligned}
\end{align} 
where $A_b^{0,\pm}=\underset{\xi\rightarrow\pm\infty}{\lim}u_b^0=A_p^\pm+\frac{1}{2}$. The boundary conditions of $A_b^{1,\pm}$ are $A_b^{1,\pm}|_{\eta=0}=I^{0,\pm}$ and $\underset{\eta\rightarrow\infty}{\lim}A_b^{1,\pm}=0$. Let $\chi_1(\xi)$ be a smooth cut-off function satisfying $\chi_1|_{(-\infty,0]}=0$, $\chi_1|_{[1,\infty)}=1$, and $0\leqslant\chi_1\leqslant1$. Then the following lemma for $[u_b^1,v_b^1]$ holds.
\begin{lemma}\label{brandtl linear}
(\ref{brandtl1}) admits a solution $[u_b^1,v_b^1]$ satisfying 
\begin{align}\label{ube}
\begin{aligned}
\int_{-\infty}^\infty\int_{0}^\infty\bigg|\p^k_\xi\p^l_\eta\Big[u_b^1-\chi_1(\xi)A_b^{1,+}-\chi_1(-\xi)A_b^{1,-} \Big]\bigg|^2\cosh\Big(\frac{\mu\xi}{8}\Big)e^{2c_1\eta}\dd\eta\dd\xi\leqslant C_{k,l},
\end{aligned}
\end{align}
where $k$ and $l$ are non-negative integers, $c_1$ is a positive constant. And $A_b^{1,\pm}$ solve the equations (\ref{brainf1}) and (\ref{braminf1}), respectively, satisfying 
\begin{align}\label{Abe}
\begin{aligned}
\int_{0}^\infty\bigg|\frac{\dd^l}{\dd \eta^l}A_b^{1,\pm}\bigg|^2e^{2c_1\eta}\dd\eta\dd\xi\leqslant C_{l},
\end{aligned}
\end{align}
where $l$ is non-negative integer.
\end{lemma}
\begin{proof}
We first estimate $A_b^{1,+}$. Indeed, one can solve the ordinary differential equation (\ref{brainf1}) explicitly. Note that $\varphi_1:=A^+_{p\eta}$ satisfies
\begin{align}\label{ODE}
\begin{aligned}
-2\beta A^+_p\varphi_1-\varphi_{1\eta\eta}=0.
\end{aligned}
\end{align} 
Recall $A^+_p=-\frac{1}{2}f'\big(\sqrt{\frac{\beta}{2}\eta}\big)$ with
$$
f'(\tilde{\eta})=3\tanh^2\Big(\frac{\tilde{\eta}}{\sqrt{2}}+\mathrm{artanh}\sqrt{\frac{2}{3}}\hspace{1mm}\Big)-2.
$$
Therefore $\varphi_1(\eta)<0$ for $\eta\geqslant0$ and $\underset{\eta\rightarrow\infty}{\lim}\varphi_1(\eta) e^{\sqrt{\beta}\eta}=-c$, where $c$ is a positive constant. The other independent solution to (\ref{ODE}) can be $\varphi_2(\eta)=\varphi_1(\eta)\int_0^\eta\frac{1}{\varphi_1^2(\eta')}\dd\eta'$. It is easy to check that $\underset{\eta\rightarrow\infty}{\lim}\varphi_2(\eta) e^{-\sqrt{\beta}\eta}=-c'$, where $c'$ is a positive constant. Let $F_b^{1,+}=-2\beta I^{0,+}A_b^{0,+}$. Then $A_b^{1,+}$ can be represented as
\begin{align}\label{ODEsolu}
\begin{aligned}
A_b^{1,+}(\eta)=C_1\varphi_1(\eta)-\varphi_1(\eta)\int_0^\eta\varphi_2(\eta')F_b^{1,+}(\eta')\dd\eta'-\varphi_2(\eta)\int_\eta^\infty\varphi_1(\eta')F_b^{1,+}(\eta')\dd\eta'.
\end{aligned}
\end{align}
Here $C_1$ is a constant given by $C_1\varphi_1(0)=I^{0,+}$ to match the boundary condition  $A_b^{1,+}|_{\eta=0}=I^{0,+}$. (\ref{ODEsolu}) shows that if $F_b^{1,+}(\eta)e^{\lambda\eta}$ is bounded for $0<\lambda\leqslant\sqrt{\beta}$, then $A_b^{1,+}(\eta)e^{\lambda\eta}$ is bounded. The estimates on high-order derivatives are similar, thus, (\ref{Abe}) holds for $c_1<\min\{\sqrt{\al},\sqrt{\beta}\}$.

Now we estimate $u_b^1$. Set
\begin{align}\nonumber
\begin{aligned}
\tilde{u}:=&u_b^1-u_{c}\\
         :=&u_b^1-\frac{1}{Q}\Phi^1_{e\psi}|_{\psi=-1}\chi_2(\eta)-\chi_1(\xi)\Big(A_b^{1,+}-I^{0,+}\chi_2(\eta)\Big)-\chi_1(-\xi)\Big(A_b^{1,-}-I^{0,-}\chi_2(\eta)\Big),
\end{aligned}
\end{align}
where $\chi_2\in C^\infty[0,\infty)$ is a cut-off function satisfying $\chi_2|_{[0,\frac{1}{2}]}=1$, $\chi_2|_{[1,\infty)}=0$, and $0\leqslant\chi_2\leqslant1$. It follows that $\tilde{u}|_{\eta=0}=\tilde{u}|_{\eta\rightarrow\infty}=0$ and $\underset{\xi\rightarrow\infty}{\lim}\tilde{u}=0$. Denote $\phi=-\int_0^\eta \tilde{u}(\xi,\eta')\dd\eta'$. It then holds in $(-\infty,\infty)\times(0,\infty)$,
\begin{align}\label{lpraphi}
\left\{
\begin{aligned}
&-u_p\phi_{\xi\eta}-u_{p\xi}\phi_\eta+Au_p\phi_\eta-v_p\phi_{\eta\eta}+u_{p\eta}\phi_\xi+\phi_{\eta\eta\eta}=\f,\\
&\phi|_{\xi\rightarrow\infty}=0,\\
&\phi|_{\eta=0}=\phi_\eta|_{\eta=0}=0,\quad \underset{\eta\rightarrow\infty}{\lim}\phi_\eta=0,
\end{aligned}
\right.
\end{align}
where
\begin{align}\nonumber
\begin{aligned}
\f=F_b^1-u_p u_{c\xi}-u_{p\xi}u_c+Au_pu_c-v_pu_{c\eta}+u_{p\eta}\int_0^\eta u_{c\xi}\dd\eta'+u_{c\eta\eta}.
\end{aligned}
\end{align}
It follows from the estimates (\ref{infe}), (\ref{zere}), (\ref{Euler1e}), (\ref{GD2}) and (\ref{Abe}) that
$$
\int_{-\infty}^\infty\int_{0}^\infty\Big|\p^k_\xi\p^l_\eta \f\Big|^2\cosh\Big(\frac{\mu\xi}{8}\Big)e^{2c_1\eta}\dd\eta\dd\xi< \infty,
$$
for $c_1<\min\{c_0,\sqrt{\mu}\}$, $k\geqslant0$, and $l\geqslant0$. 

Now we derive a prior estimate for $\phi$. Let $[\bar{u},\bv]=[u_p,v_p]$, $\hh=\frac{\phi}{\bar{u}}$. (\ref{prapro}) implies $-u_p\geqslant m_0\eta$ when $\eta\leqslant\eta_0$, and  $-u_p\geqslant m_0\eta_0$ when $\eta\geqslant\eta_0$. Thus, $\hh$ is well-defined due to $\phi|_{\eta=0}=0$. Then (\ref{lpraphi}) now reads 
\begin{align}\label{lprahh}
\left\{
\begin{aligned}
&-\big[\bu^2\hh_\eta\big]_\xi+\p^3_\eta\big[\bu\hh\big]-\bv\big[\bu\hh\big]_{\eta\eta}+A\bu\big[\bu\hh\big]_\eta+\bu\bv_{\eta\eta}\hh=\f,\\
&\hh|_{\xi\rightarrow\infty}=0,\\
&\hh|_{\eta=0}=0,\quad \underset{\eta\rightarrow\infty}{\lim}\hh_\eta=0,
\end{aligned}
\right.
\end{align}
where $\underset{\eta\rightarrow\infty}{\lim}\hh_\eta=0$ follows from $\underset{\eta\rightarrow\infty}{\lim}\phi_\eta=0$ and $\bu_{\eta}$ decays fast. Set $\rho=\cosh\Big(\frac{\mu\xi}{8}\Big)e^{2c_1\eta}$ with $c_1$ being a small constant to be chosen later. Multiplying the first equation in (\ref{lprahh}) by $\hh_\eta\rho$ and integrating in $(\xi^*,\infty)\times(0,\infty)$ give
\begin{align}\label{hhall}
\begin{aligned}
&\int_{\xi^*}^\infty\int_0^\infty\Big(-\big[\bu^2\hh_\eta\big]_\xi+\p^3_\eta\big[\bu\hh\big]-\bv\big[\bu\hh\big]_{\eta\eta}+A\bu\big[\bu\hh\big]_\eta+\bu\bv_{\eta\eta}\hh\Big)\hh_\eta\rho\dd\eta\dd\xi\\
=&\int_{\xi^*}^\infty\int_0^\infty \f\hh_\eta\rho\dd\eta\dd\xi.
\end{aligned}
\end{align}
We deal with the left hand side of (\ref{hhall}) term by terms. One can get by integrating by parts that
\begin{align}\label{hhes1}
\begin{aligned}
&\int_{\xi^*}^\infty\int_0^\infty-\big[\bu^2\hh_\eta\big]_\xi\hh_\eta\rho\dd\eta\dd\xi\\
=&\frac{1}{2}\int_0^\infty \bu^2\hh_\eta^2\rho\dd\eta\bigg|_{\xi=\xi^*}+\int_{\xi^*}^\infty\int_0^\infty \Big(-\bu\bu_\xi\hh_\eta^2\rho+\frac{1}{2}\bu^2\hh_\eta^2\rho_\xi\Big)\dd\eta\dd\xi.
\end{aligned}
\end{align}
\begin{align}\label{hhes2}
\begin{aligned}
&\int_{\xi^*}^\infty\int_0^\infty\p^3_\eta\big[\bu\hh\big]\hh_\eta\rho\dd\eta\dd\xi\\
=&\int_{\xi^{*}}^\infty -\bu_{\eta}\hh_\eta^2\rho\dd\xi\bigg|_{\eta=0}+\int_{\xi^*}^\infty\int_0^\infty \Big(-\bu\hh_{\eta\eta}^2\rho+2\bu_{\eta\eta}\hh_\eta^2\rho-\frac{1}{2}\bu_{\eta\eta\eta\eta}\hh^2\rho\Big)\dd\eta\dd\xi\\
&+\int_{\xi^*}^\infty\int_0^\infty \Big(-\frac{1}{2}\bu_\eta\hh_\eta^2\rho_\eta+\frac{1}{2}\bu\hh_\eta^2\rho_{\eta\eta}-\frac{1}{2}\bu_{\eta\eta\eta}\hh^2\rho_\eta\Big)\dd\eta\dd\xi.
\end{aligned}
\end{align}
\begin{align}\label{hhes3}
\begin{aligned}
&\int_{\xi^*}^\infty\int_0^\infty-\bv\big[\bu\hh\big]_{\eta\eta}\hh_\eta\rho\dd\eta\dd\xi\\
=&\int_{\xi^*}^\infty\int_0^\infty \Big(-\frac{3}{2}\bv\bu_\eta\hh_\eta^2\rho+\frac{1}{2}\bu\bv_\eta\hh_\eta^2\rho+\frac{1}{2}\bv\bu_{\eta\eta\eta}\hh^2\rho+\frac{1}{2}\bv_\eta\bu_{\eta\eta}\hh^2\rho\Big)\dd\eta\dd\xi\\
 &+\int_{\xi^*}^\infty\int_0^\infty \Big(\frac{1}{2}\bv\bu\hh_\eta^2\rho_\eta+\frac{1}{2}\bv\bu_{\eta\eta}\hh^2\rho_\eta\Big)\dd\eta\dd\xi.\\
\end{aligned}
\end{align}
\begin{align}\label{hhes4}
\begin{aligned}
&\int_{\xi^*}^\infty\int_0^\infty A\bu\big[\bu\hh\big]_\eta\hh_\eta\rho\dd\eta\dd\xi\\
=&\int_{\xi^*}^\infty\int_0^\infty \Big(A\bu^2\hh_\eta^2\rho-\frac{1}{2}A\bu\bu_{\eta\eta}\hh^2\rho-\frac{1}{2}A\bu_\eta^2\hh^2\rho\Big)\dd\eta\dd\xi\\
 &+\int_{\xi^*}^\infty\int_0^\infty -\frac{1}{2}A\bu\bu_\eta\hh^2\rho_\eta\dd\eta\dd\xi.
\end{aligned}
\end{align}
\begin{align}\label{hhes5}
\begin{aligned}
&\int_{\xi^*}^\infty\int_0^\infty \bu\bv_{\eta\eta}\hh\hh_\eta\rho\dd\eta\dd\xi\\
=&\int_{\xi^*}^\infty\int_0^\infty \Big(-\frac{1}{2}\bu\bv_{\eta\eta\eta}\hh^2\rho-\frac{1}{2}\bu_\eta\bv_{\eta\eta}\hh^2\rho\Big)\dd\eta\dd\xi+\int_{\xi^*}^\infty\int_0^\infty -\frac{1}{2}\bu\bv_{\eta\eta}\hh^2\rho_\eta\dd\eta\dd\xi.
\end{aligned}
\end{align}
Collect (\ref{hhes1}) - (\ref{hhes5}) to get
\begin{align}\nonumber
&\int_{\xi^*}^\infty\int_0^\infty\Big(-\big[\bu^2\hh_\eta\big]_\xi+\p^3_\eta\big[\bu\hh\big]-\bv\big[\bu\hh\big]_{\eta\eta}+A\bu\big[\bu\hh\big]_\eta+\bu\bv_{\eta\eta}\hh\Big)\hh_\eta\rho\dd\eta\dd\xi\\\nonumber
=&\frac{1}{2}\int_0^\infty \bu^2\hh_\eta^2\rho\dd\eta\bigg|_{\xi=\xi^*}+\int_{\xi^{*}}^\infty -\bu_{\eta}\hh_\eta^2\rho\dd\xi\bigg|_{\eta=0}+\int_{\xi^*}^\infty\int_0^\infty-\bu\hh_{\eta\eta}^2\rho\dd\eta\dd\xi\\\nonumber
&+\int_{\xi^*}^\infty\int_0^\infty\Big(-\frac{3}{2}\bu\bu_\xi-\frac{3}{2}\bv\bu_{\eta}+2\bu_{\eta\eta}+A\bu^2\Big)\hh_\eta^2\rho\dd\eta\dd\xi\\\nonumber
&+\int_{\xi^*}^\infty\int_0^\infty\Big(-\bu_{\eta\eta\eta\eta}+\bv\bu_{\eta\eta\eta}+\bv_\eta\bu_{\eta\eta}-A\bu\bu_{\eta\eta}-A\bu_\eta^2-\bu\bv_{\eta\eta\eta}-\bu_\eta\bv_{\eta\eta}\Big)\frac{\hh^2\rho}{2}\dd\eta\dd\xi\\\nonumber
&+\int_{\xi^*}^\infty\int_0^\infty\Big(\bu^2\rho_\xi-\bu_\eta\rho_\eta+\bu\rho_{\eta\eta}+\bv\bu\rho_\eta\Big)\frac{\hh_\eta^2}{2}\dd\eta\dd\xi\\\nonumber
&+\int_{\xi^*}^\infty\int_0^\infty\Big(\bu_{\eta\eta\eta}\rho_\eta+\bv\bu_{\eta\eta}\rho_\eta-A\bu\bu_\eta\rho_\eta-\bu\bv_{\eta\eta}\rho_\eta\Big)\frac{\hh^2}{2}\dd\eta\dd\xi.
\end{align}
Since $-\bu_\eta\geqslant0$ and $-\bu\geqslant0$, thus $-\bu_{\eta}\hh_\eta^2\rho$ and $-\bu\hh_{\eta\eta}^2\rho$ are non-negative. $[\bu,\bv]$ satisfies the Prandtl equation:
$$ \bu\bu_\xi+\bv\bu_\eta-\frac{1}{2}A\bu^2-\bu_{\eta\eta}+\frac{1}{8}A=0,
$$
and $\bar{u}_{\eta\eta}\geqslant0$.
Thus,
\begin{align}\nonumber
\begin{aligned}
-\frac{3}{2}\bu\bu_\xi-\frac{3}{2}\bv\bu_{\eta}+2\bu_{\eta\eta}+A\bu^2=\frac{1}{4}A\bu^2+\frac{1}{2}\bu_{\eta\eta}+\frac{3}{16}A\geqslant \frac{3}{16}A\geqslant\frac{3}{8}\mu.
\end{aligned}
\end{align}
And
\begin{align}\nonumber
\begin{aligned}
&-\bu_{\eta\eta\eta\eta}+\bv\bu_{\eta\eta\eta}+\bv_\eta\bu_{\eta\eta}-A\bu\bu_{\eta\eta}-A\bu_\eta^2-\bu\bv_{\eta\eta\eta}-\bu_\eta\bv_{\eta\eta}\\
=&\p^2_\eta\Big[\bu\bu_\xi+\bv\bu_\eta-\frac{1}{2}A\bu-\bu_{\eta\eta}\Big]=0.
\end{aligned}
\end{align}
So one can obtain
\begin{align}\nonumber
&\int_{\xi^*}^\infty\int_0^\infty\Big(-\big[\bu^2\hh_\eta\big]_\xi+\p^3_\eta\big[\bu\hh\big]-\bv\big[\bu\hh\big]_{\eta\eta}+A\bu\big[\bu\hh\big]_\eta+\bu\bv_{\eta\eta}\hh\Big)\hh_\eta\rho\dd\eta\dd\xi\\\nonumber
\geqslant&\int_{\xi^*}^\infty\int_0^\infty-\bu\hh_{\eta\eta}^2\rho\dd\eta\dd\xi+\int_{\xi^*}^\infty\int_0^\infty \frac{3\mu}{8}\hh_\eta^2\rho\dd\eta\dd\xi\\\nonumber
&+\int_{\xi^*}^\infty\int_0^\infty\Big(\bu^2\rho_\xi-\bu_\eta\rho_\eta+\bu\rho_{\eta\eta}+\bv\bu\rho_\eta\Big)\frac{\hh_\eta^2}{2}\dd\eta\dd\xi\\\nonumber
&+\int_{\xi^*}^\infty\int_0^\infty\Big(\bu_{\eta\eta\eta}\rho_\eta+\bv\bu_{\eta\eta}\rho_\eta-A\bu\bu_\eta\rho_\eta-\bu\bv_{\eta\eta}\rho_\eta\Big)\frac{\hh^2}{2}\dd\eta\dd\xi.
\end{align}
Note that $|\bu^2\rho_\xi|\leqslant\frac{\mu}{32}\rho$ and $|-\bu_\eta\rho_\eta+\bu\rho_{\eta\eta}+\bv\bu\rho_\eta|\leqslant \big(\|\bu_\eta\|_{L^\infty}c_1+\frac{c_1^2}{2}+\|\bv\|_{L^\infty}\frac{c_1}{2}\big)\rho$. One has
\begin{align}\nonumber
&\bigg|\int_{\xi^*}^\infty\int_0^\infty\Big(\bu^2\rho_\xi-\bu_\eta\rho_\eta+\bu\rho_{\eta\eta}+\bv\bu\rho_\eta\Big)\frac{\hh_\eta^2}{2}\dd\eta\dd\xi\bigg|\\\nonumber
\leqslant &(\frac{3\mu}{32}+Mc_1)\int_{\xi^*}^\infty\int_0^\infty \hh_\eta^2\rho\dd\eta\dd\xi,
\end{align}
where $M$ is independent of $c_1$. If $c_1\leqslant \frac{c_0}{2}$, then
$$\big|\bu_{\eta\eta\eta}\rho_\eta+\bv\bu_{\eta\eta}\rho_\eta-A\bu\bu_\eta\rho_\eta-\bu\bv_{\eta\eta}\rho_\eta\big|\leqslant Mc_1e^{-c_0\eta}e^{c_1\eta}\leqslant \frac{M'c_1}{\eta^2},
$$
which yields
\begin{align}\nonumber
\begin{aligned}
&\bigg|\int_{\xi^*}^\infty\int_0^\infty\Big(\bu_{\eta\eta\eta}\rho_\eta+\bv\bu_{\eta\eta}\rho_\eta-A\bu\bu_\eta\rho_\eta-\bu\bv_{\eta\eta}\rho_\eta\Big)\frac{\hh^2}{2}\dd\eta\dd\xi\bigg|\\
\leqslant &\frac{M'c_1}{2}\int_{\xi^*}^\infty\int_0^\infty \frac{\hh^2}{\eta^2}\dd\eta\dd\xi\\
\leqslant &2M'c_1\int_{\xi^*}^\infty\int_0^\infty\hh_\eta^2\dd\eta\dd\xi\\
\leqslant &2M'c_1\int_{\xi^*}^\infty\int_0^\infty\hh_\eta^2\rho\dd\eta\dd\xi.
\end{aligned}
\end{align}
Choosing $c_1$ small enough gives
\begin{align}\nonumber
&\int_{\xi^*}^\infty\int_0^\infty\Big(-\big[\bu^2\hh_\eta\big]_\xi+\p^3_\eta\big[\bu\hh\big]-\bv\big[\bu\hh\big]_{\eta\eta}+A\bu\big[\bu\hh\big]_\eta+\bu\bv_{\eta\eta}\hh\Big)\hh_\eta\rho\dd\eta\dd\xi\\\nonumber
\geqslant&\int_{\xi^*}^\infty\int_0^\infty-\bu\hh_{\eta\eta}^2\rho\dd\eta\dd\xi+\int_{\xi^*}^\infty\int_0^\infty \frac{\mu}{4}\hh_\eta^2\rho\dd\eta\dd\xi.
\end{align}
Thus, obtains 
\begin{align}\nonumber
\begin{aligned}
\int_{\xi^*}^\infty\int_0^\infty-\bu\hh_{\eta\eta}^2\rho\dd\eta\dd\xi+\int_{\xi^*}^\infty\int_0^\infty \hh_\eta^2\rho\dd\eta\dd\xi\leqslant C(\mu)\int_{\xi^*}^\infty\int_0^\infty \f^2\rho\dd\eta\dd\xi.
\end{aligned}
\end{align}
Taking the limit $\xi^*\rightarrow-\infty$ in tha above inequality yields 
\begin{align}\nonumber
\begin{aligned}
\int_{-\infty}^\infty\int_0^\infty-\bu\hh_{\eta\eta}^2\rho\dd\eta\dd\xi+\int_{-\infty}^\infty\int_0^\infty \hh_\eta^2\rho\dd\eta\dd\xi\leqslant C(\mu)\int_{-\infty}^\infty\int_0^\infty \f^2\rho\dd\eta\dd\xi.
\end{aligned}
\end{align}
Differentiating the first equation in (\ref{lprahh}) in $\xi$ and applying a similar analysis to $\hh_\xi$ imply
\begin{align}\nonumber
\begin{aligned}
\int_{-\infty}^\infty\int_0^\infty-\bu\hh_{\xi\eta\eta}^2\rho\dd\eta\dd\xi+\int_{-\infty}^\infty\int_0^\infty \hh_{\xi\eta}^2\rho\dd\eta\dd\xi\leqslant C(\mu)\int_{-\infty}^\infty\int_0^\infty \Big(\f^2+\f_\xi^2\Big)\rho\dd\eta\dd\xi.
\end{aligned}
\end{align}
Since $\phi=\bu\hh$, it is easy to check
\begin{align}\nonumber
\begin{aligned}
\int_{-\infty}^\infty\int_0^\infty \Big(\phi_\eta^2+\phi_{\eta\eta}^2\Big)\rho\dd\eta\dd\xi\leqslant C\int_{-\infty}^\infty\int_0^\infty \f^2\rho\dd\eta\dd\xi,
\end{aligned}
\end{align}
and 
\begin{align}\nonumber
\begin{aligned}
\int_{-\infty}^\infty\int_0^\infty \Big(\phi_{\xi\eta}^2+\phi_{\xi\eta\eta}^2\Big)\rho\dd\eta\dd\xi\leqslant C\int_{-\infty}^\infty\int_0^\infty \Big(\f^2+\f_\xi^2\Big)\rho\dd\eta\dd\xi.
\end{aligned}
\end{align}
The estimate on $\phi_{\eta\eta\eta}$ follows directly from the (\ref{lpraphi}). The estimates on high-order derivatives in (\ref{ube}) are similar, so we proved Lemma \ref{brandtl linear}.
\end{proof}

With $u_b^1$ at hand, one can set $\Phi_b^1(\xi,\eta)=\int_\eta^\infty u_b^1(\xi,\eta')\dd\eta'$.
 
Next we construct $\Phi_e^m$ and $\Phi_b^m$ for $2\leqslant m\leqslant n$. $\Phi_e^m$ solves the following problem:
\begin{equation}\label{Eulerm}
\left\{
\begin{aligned}
&-\Delta_{\xi,\psi}\hspace{1mm}\Phi_e^m=0,\\
&\Phi_e^m|_{\psi=\pm1}=\pm Q\Phi_b^{m-1}(\xi,0),\\
&\underset{\xi\rightarrow\pm\infty}{\lim}\Phi_e^m=QI^{m-1,\pm}\psi,
\end{aligned}
\right.
\end{equation} 
where $I^{m-1,\pm}=\underset{\xi\rightarrow\pm\infty}{\lim}\Phi_b^{m-1}(\xi,0)$. 

Let $[u_b^m,v_b^m]=[-\Phi_{b\eta}^m,\Phi_{b\xi}^m]$. Then $[u_b^m,v_b^m]$ solves the following problem:
\begin{align}\label{brandtlm}
\left\{
\begin{aligned}
&u_pu_{b\xi}^m+u_b^mu_{p\xi}-Au_pu_{b}^m+v_pu_{b\eta}^m+\big[v_b^m-v_b^m|_{\eta=0}\big]u_{p\eta}-u_{b\eta\eta}^m=F^m_b,\\
&u_{b\xi}^m+v_{b\eta}^m=0,\\
&u_b^m|_{\eta=0}=\frac{1}{Q}\Phi_{e\psi}^{m}|_{\psi=-1},\quad \underset{\eta\rightarrow\infty}{\lim}u_b^m=0,
\end{aligned}
\right.
\end{align}
where the explicit form of $F_b^m$ is given in the appendix. Set $[A_b^{m,\pm},0]=\underset{\xi\rightarrow\pm\infty}{\lim}[u_b^m,v_b^m]$. It follows from taking the limits as $\xi\rightarrow\pm\infty$ in (\ref{brandtlm}) that
\begin{equation}\label{Abmeq}
\left\{
\begin{aligned}
&-2\beta A^+_pA_{b}^{m,+}-A_{b\eta\eta}^{m,+}=F_b^{m,+},\\
&A_{b}^{m,+}|_{\eta=0}=I^{m-1,+},\quad \underset{\eta\rightarrow\infty}{\lim}A_b^{m,+}=0,
\end{aligned}
\right.
\end{equation} 
where 
\begin{align}\nonumber
\begin{aligned}
F_b^{m,+}=&-\beta\overset{m-1}{\underset{j=1}{\sum}}\Big(I^{m-1-j,+}\big(A_b^{j,+}-I^{j-1,+}\big)-I^{m-1-j,+}A_b^{j,+}\Big)\\
&-2\beta I^{m-1,+}A_b^0+4\beta^2A_b^{m-2,+},
\end{aligned}
\end{align}
and
\begin{equation}\label{Abmeq2}
\left\{
\begin{aligned}
&-2\al A^-_pA_{b}^{m,-}-A_{b\eta\eta}^{m,-}=F_b^{m,-},\\
&A_{b}^{m,-}|_{\eta=0}=I^{m-1,-},\quad \underset{\eta\rightarrow\infty}{\lim}A_b^{m,-}=0,
\end{aligned}
\right.
\end{equation} 
where 
\begin{align}\nonumber
\begin{aligned}
F_b^{m,-}=&-\al\overset{m-1}{\underset{j=1}{\sum}}\Big(I^{m-1-j,-}\big(A_b^{j,-}-I^{j-1,-}\big)-I^{m-1-j,-}A_b^{j,-}\Big)\\
&-2\al I^{m-1,-}A_b^0+4\al^2A_b^{m-2,-}.
\end{aligned}
\end{align}
$u_b^m$ can be analyzed exactly as for $u_b^1$. With $u_b^m$ so determined, one can set $\Phi_b^m(\xi,\eta)=\int_\eta^\infty u_b^m(\xi,\eta')\dd\eta'$ and $v_b^m=\Phi_{b\xi}^m$. We conclude the following proposition for the approximate profiles.
\begin{proposition} \label{approxi}
For $1\leqslant m\leqslant n$, (\ref{Eulerm}) and (\ref{brandtlm}) admit smooth solutions $\Phi_e^m$ and $[u_b^m,v_b^m]$ respectively, which satisfy the following estimates:
\begin{align}\label{Eulermes}
\begin{aligned}
&\Big|\p_\xi^k\p_\psi^l\big[\Phi_e^m(\xi,\psi)-QI^{m-1,+}\psi\big]\Big| \leqslant C_{k,l,m}e^{-\frac{\mu\xi}{16}}Q,\quad \text{for } (\xi,\psi)\in[0,\infty)\times[-1,1],\\
&\Big|\p_\xi^k\p_\psi^l\big[\Phi_e^m(\xi,\psi)-QI^{m-1,-}\psi\big]\Big| \leqslant C_{k,l,m}e^{\frac{\mu\xi}{16}}Q,\quad \text{for } (\xi,\psi)\in(-\infty,0)\times[-1,1],
\end{aligned}
\end{align}
where $k$ and $l$ non-negative integers, and 
\begin{align}\label{brandtlmes}
\begin{aligned}
&\Big|\p_\xi^k\p_\eta^l\big[u_b^m(\xi,\eta)-A_b^{m,+}(\eta)\big]\Big| \leqslant C_{k,l,m}e^{-c_m\eta}e^{-\frac{\mu\xi}{16}},\quad \text{for } (\xi,\eta)\in[0,\infty)\times[0,\infty),\\
&\Big|\p_\xi^k\p_\eta^l\big[u_b^m(\xi,\eta)-A_b^{m,-}(\eta)\big]\Big| \leqslant C_{k,l,m}e^{-c_m\eta}e^{\frac{\mu\xi}{16}},\quad \text{for } (\xi,\eta)\in(-\infty,0)\times[0,\infty),
\end{aligned}
\end{align}
where $c_m$ is positive constant, and $A_b^{m,\pm}$ are the solutions to the equations (\ref{Abmeq}) and (\ref{Abmeq2}) respectively, such that
\begin{align}
\begin{aligned}
\Big|\frac{\dd^l}{\dd \eta^l}A_b^{m,\pm}\Big|\leqslant C_{l}e^{-c_m\eta},\quad \text{for } \eta\in [0,\infty).
\end{aligned}
\end{align}
\end{proposition}
Now an approximate solution is constructed in the following form:
\begin{align}\nonumber
\begin{aligned}
\Phi_s=&\underset{m=0}{\overset{n}{\sum}}\Big(\eq\Big)^m\Phi_e^m(\xi,\psi)+\underset{m=0}{\overset{n-1}{\sum}}\Big(\eq\Big)^{m+1}Q\Phi_b^m\big(\xi,\qe(\psi+1)\big)\chi_2(\psi+1)\\
&-\underset{m=0}{\overset{n-1}{\sum}}\Big(\eq\Big)^{m+1}Q\Phi_b^{m}\big(\xi,\qe(1-\psi)\big)\chi_2(1-\psi)\\
&-\Big(\eq\Big)^{n+1}\chi_2(\psi+1)\int_0^{\qe(\psi+1)}Qu_b^n(\xi,\eta')\dd\eta'\\
&+\Big(\eq\Big)^{n+1}\chi_2(1-\psi)\int_0^{\qe(1-\psi)}Qu_b^n(\xi,\eta')\dd\eta',
\end{aligned}
\end{align}
where $\chi_2\in C^\infty[0,\infty)$ is a cut-off function satisfying $\chi_2|_{[0,\frac{1}{2}]}=1$, $\chi_2|_{[1,\infty)}=0$. It is easy to check that $\Phi_s$ satisfies 
$$\Phi_s|_{\psi=\pm1}=\pm\frac{Q}{2},\quad \Phi_{s\psi}|_{\psi=\pm1}=0.$$
Set 
\begin{align}\nonumber
\begin{aligned}
R_s:=&\Phi_{s\psi}\Delta_{\xi,\psi}\Phi_{s\xi}+\frac{J_\xi}{J}\Phi_{s\psi}\Delta_{\xi,\psi}\Phi_{s}-\Phi_{s\xi}\Delta_{\xi,\psi}\Phi_{s\psi}-\frac{J_\psi}{J}\Phi_{s\xi}\Delta_{\xi,\psi}\Phi_{s}\\
&+\frac{\e}{J}\Delta_{\xi,\psi}\big(J\Delta_{\xi,\psi}\Phi_s\big).
\end{aligned}
\end{align}
Then 
$$\|R_s\|_{L^\infty\big(\mathbb{R}\times(-1,1)\big)}\leqslant C\Big(\eq\Big)^nQ^2.$$
However, it must be noted that $R_s$ does not vanish as $\xi\rightarrow\pm\infty$. Actually,
\begin{align}\nonumber
 &\underset{\xi\rightarrow\infty}{\lim}\frac{J_\xi}{J}=-2\beta,\quad  \underset{\xi\rightarrow\infty}{\lim}\frac{J_{\xi\xi}}{J}=4\beta^2\\\nonumber
&\underset{\xi\rightarrow\infty}{\lim}\frac{J_\psi}{J}=0,\quad  \underset{\xi\rightarrow\infty}{\lim}\frac{J_{\xi\xi}}{J}=0,\\\nonumber
&\underset{\xi\rightarrow\infty}{\lim}\Phi_{s\xi}=0,
\end{align}
hence 
\begin{align}\nonumber
 &\underset{\xi\rightarrow\infty}{\lim}R_s=-2\beta\Phi^+_{s\psi}\Phi^+_{s\psi\psi}+\e\Phi^+_{s\psi\psi\psi\psi}+4\beta\e\Phi_{s\psi\psi}^+,
\end{align}
where $\Phi^+_s=\underset{\xi\rightarrow\infty}{\lim}\Phi_s\not\equiv0$. Thus, $R_s\notin L^2{\big(\mathbb{R}\times(-1,1)\big)}$. It leads to our analysis at $\xi\rightarrow\pm\infty$ of the Navier-Stokes equations in the next section.

\section{The analysis at $\xi\rightarrow\pm\infty$ of the Navier-Stokes equations}\label{infty behavior}
In this section, we analyze the Navier-Stokes equations as $\xi\rightarrow\pm\infty$. Recall the equivalent formulation of the Navier-Stokes problem, (\ref{NSP}) in $(-\infty,\infty)\times(-1,1)$:
\begin{equation}\nonumber
\left\{
\begin{aligned}
&J\Phi_\psi\p_\xi\big(J\Delta_{\xi,\psi}\Phi\big)-J\Phi_\xi\p_\psi\big(J\Delta_{\xi,\psi}\Phi\big)+\e J\Delta_{\xi,\psi}\big(J\Delta_{\xi,\psi}\Phi\big)=0, \\
&\Phi|_{\psi=1}=\frac{Q}{2}, \quad\Phi|_{\psi=-1}=-\frac{Q}{2},\\
&\Phi_\psi|_{\psi=1}=\Phi_\psi|_{\psi=-1}=0.
\end{aligned}
\right.
\end{equation}
Multiplying the first equation above by $\frac{1}{J^2}$ yields
\begin{align}\label{Phiinf}
\begin{aligned}
\Phi_{\psi}\Delta_{\xi,\psi}\Phi_{\xi}+\frac{J_\xi}{J}\Phi_{\psi}\Delta_{\xi,\psi}\Phi-\Phi_{\xi}\Delta_{\xi,\psi}\Phi_{\psi}-\frac{J_\psi}{J}\Phi_{\xi}\Delta_{\xi,\psi}\Phi+\frac{\e}{J}\Delta_{\xi,\psi}\big(J\Delta_{\xi,\psi}\Phi\big)=0.
\end{aligned}
\end{align}
(\ref{asybJ1}) and (\ref{asybJ2}) imply
\begin{align}\nonumber
\begin{aligned}
\underset{\xi\rightarrow\infty}{\lim}\frac{J_\xi}{J}=-2\beta,\quad \underset{\xi\rightarrow\infty}{\lim}\frac{J_{\xi\xi}}{J}=4\beta^2, \quad\underset{\xi\rightarrow\infty}{\lim}\frac{J_\psi}{J}=\underset{\xi\rightarrow\infty}{\lim}\frac{J_{\psi\psi}}{J}=0.
\end{aligned}
\end{align}
Taking the limit in (\ref{Phiinf}) as $\xi\rightarrow\infty$ gives
\begin{equation}\label{Ape}
\left\{
\begin{aligned}
&-2\beta A^+\frac{\dd A^+}{\dd \psi}-\e\Big(\frac{\dd^3 A^+}{\dd\psi^3}+4\beta^2\frac{\dd A^+}{\dd \psi}\Big)=0,\\
&A^+(-1)=A^+(1)=0, \\
&\int_{-1}^1 A^+(\psi)\dd\psi=-Q,
\end{aligned}
\right.
\end{equation}
where $[A^+,0]=\underset{\xi\rightarrow\infty}{\lim}[-\Phi_\psi,\Phi_\xi]$. Recall the construction of the approximate solution in Section \ref{high-order}, there exists an approximate solution to the equation (\ref{Ape}) of the form: 
\begin{align}\nonumber
\begin{aligned}
A_s^+(\psi)=&Q\Big(-\frac{1}{2}+A_b^{0,+}\big(\qe(\psi+1)\big)+A_b^{0,+}\big(\qe(1-\psi)\big)\Big)\\
            &+\underset{m=1}{\overset{n}{\sum}}Q\Big(-I^{m-1,+}+A_b^{m,+}\big(\qe(\psi+1)\big)+A_b^{m,+}\big(\qe(1-\psi)\big)\Big).
\end{aligned}
\end{align}
Now the main result in this section can be stated as follows: 
\begin{proposition} \label{inftyest}
There exists a constant $\delta>0$, such that for $0<\e\leqslant \delta Q$, the problem (\ref{Ape}) has a smooth solution $A^+$ satisfying
\begin{align}\label{apesti}
\begin{aligned}
&\Big\|\frac{\dd^l}{\dd \psi^l}\big[A^+-A_s^+\big]\Big\|_{L^\infty}\leqslant C\Big(\eq\Big)^{n+1-l}Q,
\end{aligned}
\end{align}
where $l=0$, $1$, $2$, $3$, and $C$ is a constant independent of $\e$ and $Q$. 
\end{proposition}
In order to prove Proposition \ref{inftyest}, we construct a higher-order approximate solution to (\ref{Ape}). For $1\leqslant m'\leqslant n'$, where $n'$ is a positive integer to be chosen later, denote by $A_b^{n+m',+}$ the solution to the following problem:
\begin{equation}\label{higherap}
\left\{
\begin{aligned}
&-2\beta A^+_pA_{b}^{n+m',+}-A_{b\eta\eta}^{n+m',+}=F_b^{n+m',+},\\
&A_{b}^{n+m',+}|_{\eta=0}=I^{n+m'-1,+},\quad \underset{\eta\rightarrow\infty}{\lim}A_b^{n+m',+}=0,
\end{aligned}
\right.
\end{equation} 
where 
\begin{align}\nonumber
\begin{aligned}
I^{n+m'-1}=\int_0^\infty A_b^{n+m'-1,+}(\eta)\dd\eta,
\end{aligned}
\end{align}
and
\begin{align}\nonumber
\begin{aligned}
F_b^{n+m',+}=&-\beta\overset{n+m'-1}{\underset{j=1}{\sum}}\Big(I^{n+m'-1-j,+}\big(A_b^{j,+}-I^{j-1,+}\big)-I^{n+m'-1-j,+}A_b^{j,+}\Big)\\
&-2\beta I^{n+m'-1,+}A_b^0+4\beta^2A_b^{n+m'-2,+}.
\end{aligned}
\end{align}
Now we define 
\begin{align}\nonumber
\bar{A}_s^+(\psi):=&Q\Big(-\frac{1}{2}+A_b^{0,+}\big(\qe(\psi+1)\big)\chi_2(\psi+1)+A_b^{0,+}\big(\qe(1-\psi)\big)\chi_2(1-\psi)\Big)\\\nonumber
           &+\underset{m=1}{\overset{n+n'}{\sum}}-\Big(\eq\Big)^{m}QI^{m-1,+}+\underset{m=1}{\overset{n+n'}{\sum}}\Big(\eq\Big)^{m}QA_b^{m,+}\big(\qe(\psi+1)\big)\chi_2(\psi+1)\\\nonumber
            &+\underset{m=1}{\overset{n+n'}{\sum}}\Big(\eq\Big)^{m}QA_b^{m,+}\big(\qe(1-\psi)\big)\chi_2(1-\psi)\\\nonumber
            &-\underset{m=0}{\overset{n+n'-1}{\sum}}\Big(\eq\Big)^{m+1}Q\chi_2'(\psi+1)\int_{\qe(\psi+1)}^\infty A_b^{m,+}(\eta')\dd\eta'\\\nonumber
            &-\underset{m=0}{\overset{n+n'-1}{\sum}}\Big(\eq\Big)^{m+1}Q\chi_2'(1-\psi)\int_{\qe(1-\psi)}^\infty A_b^{m,+}(\eta')\dd\eta'\\\nonumber
            &+\Big(\eq\Big)^{n+n'+1}Q\chi_2'(\psi+1)\int_0^{\qe(\psi+1)}A_b^{m,+}(\eta')\dd\eta'\\\nonumber
            &+\Big(\eq\Big)^{n+n'+1}Q\chi_2'(1-\psi)\int_0^{\qe(1-\psi)}A_b^{m,+}(\eta')\dd\eta' .
\end{align}
Then it can be checked that $\bar{A}_s^+|_{\psi=\pm1}=0$ and 
$$\int_{-1}^1 \bar{A}_s^+(\psi)\dd\psi=-Q.$$
Let
$$\bar{R}_s^+:=-2\beta \bar{A}_s^+\frac{\dd }{\dd \psi}\bar{A}_s^+-\e\big(\frac{\dd^3 }{\dd\psi^3}\bar{A}_s^++4\beta^2\frac{\dd}{\dd \psi}\bar{A}_s^+\big).$$
Then it holds that
$$\big\|\bar{R}_s^+\big\|_{L^\infty}\leqslant C\Big(\eq\Big)^{n+n'}Q^2.$$
It is enough to choose $n'=2$ in the following. Note that $A_b^{m,+}(\eta)$ decays fast as $\eta\rightarrow\infty$, $\chi_2$ is support on $[0,1]$, and $\chi_2|_{[0,\frac{1}{2}]}=1$. One can get from the construction that   
\begin{align}\label{apesti2}
\begin{aligned}
&\Big\|\frac{\dd^l}{\dd \psi^l}\big[\bar{A}_s^+-A_s^+\big]\Big\|_{L^\infty}\leqslant C_l\Big(\eq\Big)^{n+1-l}Q,
\end{aligned}
\end{align}
when $l$ is any non-negative integer. Now we construct a solution $A^+$ to the equation (\ref{Ape}), which is close to $\bar{A}_s^+$. Let $\mathcal{A}=A^+-\bar{A}_s^+$. $\mathcal{A}$ solves the following problem in $(-1,1)$:
\begin{equation}\label{mApe}
\left\{
\begin{aligned}
&-2\beta \bar{A}_s^+\frac{\dd \mathcal{A}}{\dd \psi}-2\beta \mathcal{A}\frac{\dd\bar{A}_s^+ }{\dd \psi}-\e\Big(\frac{\dd^3 \mathcal{A}}{\dd\psi^3}+4\beta^2\frac{\dd\mathcal{A}}{\dd \psi}\Big)=-\bar{R}_s^++2\beta \mathcal{A}\frac{\dd\mathcal{A} }{\dd \psi},\\
&\mathcal{A}(-1)=\mathcal{A}(1)=0, \\
&\int_{-1}^1 \mathcal{A}(\psi)\dd\psi=0.
\end{aligned}
\right.
\end{equation}
Set $\Xi=-\int_{-1}^\psi \mathcal{A}(\psi')\dd\psi' $. Then it holds that
\begin{equation}\label{Xien}
\left\{
\begin{aligned}
&-2\beta \bar{A}_s^+\frac{\dd^2 \Xi}{\dd \psi^2}-2\beta \frac{\dd \bar{A}_s^+}{\dd \psi}\frac{\dd\Xi }{\dd \psi}-\e\Big(\frac{\dd^4\Xi }{\dd\psi^4}+4\beta^2\frac{\dd^2\Xi}{\dd \psi^2}\Big)=\bar{R}_s^+-2\beta \frac{\dd\Xi }{\dd \psi}\frac{\dd^2 \Xi}{\dd \psi^2},\\
&\Xi(\pm1)=\Xi'(\pm1)=0.
\end{aligned}
\right.
\end{equation}
Consider the linearized problem: 
\begin{equation}\label{Xie}
\left\{
\begin{aligned}
&-2\beta \bar{A}_s^+\frac{\dd^2 \Xi}{\dd \psi^2}-2\beta \frac{\dd \bar{A}_s^+}{\dd \psi}\frac{\dd\Xi }{\dd \psi}-\e\Big(\frac{\dd^4\Xi }{\dd\psi^4}+4\beta^2\frac{\dd^2\Xi}{\dd \psi^2}\Big)=F^+,\\
&\Xi(\pm1)=\Xi'(\pm1)=0.
\end{aligned}
\right.
\end{equation}
The key ingredient of this section is the following lemma. 
\begin{lemma}\label{inflinest}
There exists a constant $\delta>0$, such that for $0<\e\leqslant \delta Q$, ({\ref{Xie}}) admits a solution $\Xi$ satisfying
 \begin{align}\nonumber
\begin{aligned}
\big\|\Xi_\psi\big\|_{L^2(-1,1)}+\eq \big\|\Xi_{\psi\psi}\big\|_{L^2(-1,1)}\leqslant \frac{C}{Q}\big\|F^+\big\|_{L^2(-1,1)},
\end{aligned}
\end{align}
where $C$ is a constant independent of $\e$ and $Q$.
\end{lemma}
\begin{proof}
At first, we show that $\BA$ is negative. Since $\BA$ is symmetric about $\psi=0$, it suffices to deal with the case $\psi\leqslant0$ only. For $\delta$ is small enough, in the region $0<\psi+1\leqslant \eq$, it holds that
 \begin{align}\nonumber
\begin{aligned}
&Q\Big(-\frac{1}{2}+A_b^{0,+}\big(\qe(\psi+1)\big)\chi_2(\psi+1)\Big)\\
=&Q\Big(-\frac{1}{2}+A_b^{0,+}\big(\qe(\psi+1)\big)\Big)\\
=&QA_p^+\big(\qe(\psi+1)\big)\\
\leqslant &-cQ\qe(\psi+1)<0,
\end{aligned}
\end{align}
thus, $\BA(\psi)\leqslant -cQ\qe(\psi+1)+C Q(\psi+1)<0$; while in the region $\eq\leqslant\psi+1\leqslant 0$, since $A_b^{0,+}(\eta)$ decaying fast as $\eta\rightarrow\infty$ and $\chi_2|_{[0,\frac{1}{2}]}=1$, one has
 \begin{align}\nonumber
\begin{aligned}
&Q\Big(-\frac{1}{2}+A_b^{0,+}\big(\qe(\psi+1)\big)\chi_2(\psi+1)\Big)\\
=&Q\Big(A_p^+\big(\qe(\psi+1)\big)+A_b^{0,+}\big(\qe(\psi+1)\big)\big[\chi_2(\psi+1)-1\big]\Big)\\
\leqslant&Q\Big(-c+A_b^{0,+}\big(\qe(\psi+1)\big)\big[\chi_2(\psi+1)-1\big]\Big)\\
\leqslant &-\frac{1}{2}cQ<0,
\end{aligned}
\end{align}
thus, $\BA(\psi)\leqslant -\frac{1}{2}cQ+C Q\eq<0$. And it is easy to check that $\pm\frac{\dd \BA}{\dd \psi}|_{\psi=\pm1}>0$ due to $A_{p\eta}^{+}(0)<0$. Set $H=\frac{\Xi}{\BA}$. Then $H$ is well-defined, and $H(\pm1)=0$ follows from $\Xi|_{\psi=\pm1}=\Xi_\psi|_{\psi=\pm1}=0$. To estimate $H$, one can multiply the equation in (\ref{Xie}) by $H$, and integrate in $(-1,1)$ to obtain
 \begin{align}\label{Hint}
\begin{aligned}
\int_{-1}^1\Big(-2\beta \bar{A}_s^+ \Xi_{\psi\psi}-2\beta \bar{A}_{s\psi}^+\Xi_{\psi}-\e\big(\Xi_{\psi\psi\psi\psi}+4\beta^2\Xi_{\psi\psi}\big)\Big)H\dd\psi=\int_{-1}^1 HF^+\dd\psi.
\end{aligned}
\end{align}
Using integration by parts repeatedly and the suitable boundary conditions, one can get that
 \begin{align}\label{Hint1}
\begin{aligned}
\int_{-1}^1-2\beta \bar{A}_s^+\Xi_{\psi\psi}H\dd\psi=&\int_{-1}^12\beta\Big(\big(\BA H_\psi\big)^2-\BA\bar{A}_{s\psi\psi}^+H^2\Big) \dd\psi,
\end{aligned}
\end{align}
 \begin{align}\label{Hint2}
\begin{aligned}
\int_{-1}^1-2\beta\bar{A}_{s\psi}^+\Xi_{\psi} H\dd\psi=&\int_{-1}^1\beta\Big(\BA\bar{A}_{s\psi\psi}^+H^2-\big(\bar{A}_{s\psi}^+\big)^2H^2 \Big) \dd\psi,
\end{aligned}
\end{align}
\begin{align}\label{Hint3}
\begin{aligned}
&-\e\int_{-1}^1 \big[\BA H\big]_{\psi\psi\psi\psi} H\dd\psi\\
=&\e\int_{-1}^1 \Big(-\BA H_{\psi\psi}^2+2\bar{A}_{s\psi\psi}^+H_\psi^2-\frac{1}{2}\bar{A}_{s\psi\psi\psi\psi}^+H^2\Big)\dd\psi+\e \bar{A}_{s\psi}^+H_\psi^2\bigg|_{\psi=-1}^{\psi=1},
\end{aligned}
\end{align}
\begin{align}\label{Hint4}
\begin{aligned}
-\e\int_{-1}^14\beta^2\Xi_{\psi\psi} H\dd\psi=&\e\int_{-1}^14\beta^2 \Xi_\psi H_\psi \dd\psi.
\end{aligned}
\end{align}
Collecting (\ref{Hint1}) - (\ref{Hint4}) yields 
 \begin{align}\label{Hintall}
\begin{aligned}
&\int_{-1}^1\Big(-2\beta \bar{A}_s^+ \Xi_{\psi\psi}-2\beta \bar{A}_{s\psi}^+\Xi_{\psi}-\e\big(\Xi_{\psi\psi\psi\psi}+4\beta^2\Xi_{\psi\psi}\big)\Big)H\dd\psi\\
=&\int_{-1}^1 -\e\BA H_{\psi\psi}^2\dd\psi +\int_{-1}^1\Big(2\e\bar{A}_{s\psi\psi}^++2\beta\big(\BA \big)^2\Big) H_\psi^2\dd\psi\\
  &+\int_{-1}^1\Big(-\frac{\e}{2}\bar{A}_{s\psi\psi\psi\psi}^+ -\beta\BA\bar{A}_{s\psi\psi}^+-\beta\big(\bar{A}_{s\psi}^+\big)^2 \Big) H^2\dd\psi\\
&+\e \bar{A}_{s\psi}^+H_\psi^2\bigg|_{\psi=-1}^{\psi=1}+\e\int_{-1}^14\beta^2 \Xi_\psi H_\psi \dd\psi.
\end{aligned}
\end{align}
Since $\BA$ is negative and $\pm\frac{\dd \BA}{\dd \psi}|_{\psi=\pm1}>0$, one has 
$$\int_{-1}^1 -\e\BA H_{\psi\psi}^2\dd\psi\geqslant0,\quad \e \bar{A}_{s\psi}^+H_\psi^2\bigg|_{\psi=-1}^{\psi=1}\geqslant 0.$$ 
In the region $\psi\leqslant0$,
$$\Big|2\frac{\e}{Q}\frac{\bar{A}_{s\psi\psi}^+}{Q}+2\beta\Big(\frac{\BA}{Q} \Big)^2-2A_{p\eta\eta}^+-2\beta \big(A_p^+\big)^2\Big|\leqslant C\eq,$$
it then follows that 
\begin{align}\nonumber
\begin{aligned}
2\e\bar{A}_{s\psi\psi}^++2\beta\big(\BA \big)^2\geqslant &2Q^2\Big(A_{p\eta\eta}^++\beta\big(A_p^+\big)^2\Big)-C\eq Q^2\\
 =&\frac{\beta Q^2}{2}-C\eq Q^2\geqslant \frac{\beta Q^2}{4}
\end{aligned}
\end{align}
for $\eq$ small enough. The similar estimate holds for $\psi\geqslant 0$. Thus, 
$$\int_{-1}^1\Big(2\e\bar{A}_{s\psi\psi}^++2\beta\big(\BA \big)^2\Big) H_\psi^2\dd\psi\geqslant \frac{\beta Q^2}{4}\int_{-1}^1 H_\psi^2\dd\psi.$$
In the region $\psi\leqslant0$, it holds that
\begin{align}\nonumber
\begin{aligned}
&\bigg|-\big(\psi+1\big)^2\Big[\frac{\e}{2Q}\frac{\bar{A}_{s\psi\psi\psi\psi}^+}{Q} +\beta\frac{\BA}{Q}\frac{\bar{A}_{s\psi\psi}^+}{Q}+\beta\Big(\frac{\bar{A}_{s\psi}^+}{Q}\Big)^2\Big]\\
&+\eta^2\Big(\frac{1}{2}A_{p\eta\eta\eta\eta}^++\beta A_{p}^+A_{p\eta\eta}^++\beta\big(A_{p\eta}^+\big)^2\Big)\bigg|\leqslant C\eq,
\end{aligned}
\end{align}
consequently,
\begin{align}\nonumber
\begin{aligned}
&-\big(\psi+1\big)^2\Big(\frac{\e}{2}\bar{A}_{s\psi\psi\psi\psi}^+ -\beta\BA\bar{A}_{s\psi\psi}^+-\beta\big(\bar{A}_{s\psi}^+\big)^2\Big)\\
\geqslant&-Q^2\eta^2\Big(\frac{1}{2}A_{p\eta\eta\eta\eta}^++\beta A_{p}^+A_{p\eta\eta}^++\beta\big(A_{p\eta}^+\big)^2\Big)-C\eq Q^2\\
=&-C\eq Q^2.
\end{aligned}
\end{align}
Similarly, in the region $\psi\geqslant0$, one has
\begin{align}\nonumber
\begin{aligned}
&-\big(1-\psi\big)^2\Big(\frac{\e}{2}\bar{A}_{s\psi\psi\psi\psi}^+ -\beta\BA\bar{A}_{s\psi\psi}^+-\beta\big(\bar{A}_{s\psi}^+\big)^2\Big)\geqslant-C\eq Q^2.
\end{aligned}
\end{align}
Thus, it follows from the Hardy inequality that
\begin{align}\nonumber
\begin{aligned}
&\int_{-1}^1\Big(-\frac{\e}{2}\bar{A}_{s\psi\psi\psi\psi}^+ -\beta\BA\bar{A}_{s\psi\psi}^+-\beta\big(\bar{A}_{s\psi}^+\big)^2 \Big) H^2\dd\psi\\
&\geqslant -C\eq Q^2 \int_{-1}^1\bigg(\frac{ H}{(\psi+1)(1-\psi)}\bigg)^2\dd\psi\\
&\geqslant -C\eq Q^2 \int_{-1}^1H_\psi^2\dd\psi.
\end{aligned}
\end{align}
Note also that
\begin{align}\nonumber
\begin{aligned}
\e\int_{-1}^14\beta^2 \Xi_\psi H_\psi \dd\psi\geqslant &-C\e \Big(\big\|\BA H_\psi\big\|_{L^2}+ \big\|\bar{A}_{s\psi}^+H\big\|_{L^2}\Big)\|H\|_{L^2}\\
\geqslant &-C\frac{\e}{Q}Q^2\|H_\psi\|_{L^2}\|H\|_{L^2}-C\eq Q^2\|H\|^2_{L^2}\\
\geqslant &-C\eq Q^2\|H_\psi\|^2_{L^2}.
\end{aligned}
\end{align}
Therefore we can conclude from (\ref{Hintall}) and the estimates above that 
 \begin{align}\nonumber
\begin{aligned}
&\int_{-1}^1\Big(-2\beta \bar{A}_s^+ \Xi_{\psi\psi}-2\beta \bar{A}_{s\psi}^+\Xi_{\psi}-\e\big(\Xi_{\psi\psi\psi\psi}+4\beta^2\Xi_{\psi\psi}\big)\Big)H\dd\psi\\
\geqslant &\int_{-1}^1 -\e\BA H_{\psi\psi}^2\dd\psi+\frac{\beta Q^2}{4}\int_{-1}^1 H_\psi^2\dd\psi-C\eq Q^2\int_{-1}^1 H_\psi^2\dd\psi\\
\geqslant &\int_{-1}^1 -\e\BA H_{\psi\psi}^2\dd\psi+\frac{\beta Q^2}{8}\int_{-1}^1 H_\psi^2\dd\psi
\end{aligned}
\end{align}
for $\eq$ small enough. Thus, 
$$ \int_{-1}^1 -\e\BA H_{\psi\psi}^2\dd\psi+ Q^2\int_{-1}^1 H_\psi^2\dd\psi\leqslant \frac{C}{Q^2}\int_{-1}^1\big(F^+\big)^2\dd\psi.$$
Finally, observe that
 \begin{align}\nonumber
\begin{aligned}
\big\|\Xi_\psi\big\|_{L^2}\leqslant&\big\|\BA H_\psi\big\|_{L^2}+\big\| \bar{A}_{s\psi}^+ H\big\|_{L^2}\\
\leqslant&\big\|\BA\big\|_{L^\infty}\big\| H_\psi\big\|_{L^2}+\big\|(\psi+1)(1-\psi) \bar{A}_{s\psi}^+ \big\|_{L^\infty}\Big\|\frac{H}{(\psi+1)(1-\psi)}\Big\|_{L^2}\\
\leqslant &CQ\big\| H_\psi\big\|_{L^2}\\
\leqslant &\frac{C}{Q}\big\| F^+\big\|_{L^2},
\end{aligned}
\end{align}
 \begin{align}\nonumber
\begin{aligned}
\eq\big\|\Xi_{\psi\psi}\big\|_{L^2}\leqslant&\eq\big\|\BA H_{\psi\psi}\big\|_{L^2}+2\eq\big\| \bar{A}_{s\psi}^+ H_{\psi}\big\|_{L^2}+\eq\big\| \bar{A}_{s\psi\psi}^+ H\big\|_{L^2}\\
\leqslant&\big\|\sqrt{-\e\BA} H_{\psi\psi}\big\|_{L^2}+2\eq\big\|\bar{A}_{s\psi}^+\big\|_{L^\infty}\big\|H_{\psi}\big\|_{L^2}\\
         &+\eq\big\|(\psi+1)(1-\psi) \bar{A}_{s\psi\psi}^+ \big\|_{L^\infty}\Big\|\frac{H}{(\psi+1)(1-\psi)}\Big\|_{L^2}\\
\leqslant &\big\|\sqrt{-\e\BA} H_{\psi\psi}\big\|_{L^2}+CQ\big\| H_\psi\big\|_{L^2}\\
\leqslant &\frac{C}{Q}\big\| F^+\big\|_{L^2},
\end{aligned}
\end{align}
so the inequality in Lemma \ref{inflinest} follows.
\end{proof}

{\bf Proof of Proposition \ref{inftyest}:} We solve the problem (\ref{Xien}) by the method of the contraction mapping. Define
 \begin{align}\label{XiUpe}
\begin{aligned}
\big\|\Xi\big\|_{\X}=&\big\|\Xi_\psi\big\|_{L^2(-1,1)}+\eq\big\|\Xi_{\psi\psi}\big\|_{L^2(-1,1)}+\frac{\e}{Q}\big\|\Xi_{\psi\psi\psi}\big\|_{L^2(-1,1)}\\
&+\Big(\frac{\e}{Q}\Big)^{\frac{3}{2}}\big\|\Xi_{\psi\psi\psi\psi}\big\|_{L^2(-1,1)}.
\end{aligned}
\end{align}
Define a map $\mathcal{T}:H^4(-1,1)\rightarrow H^4(-1,1)$ as $\T(\Xi)=\Up$, where $\Up$ solves
\begin{equation}\label{Xieup}
\left\{
\begin{aligned}
&-2\beta \bar{A}_s^+\frac{\dd^2 \Up}{\dd \psi^2}-2\beta \frac{\dd \bar{A}_s^+}{\dd \psi}\frac{\dd\Up }{\dd \psi}-\e\Big(\frac{\dd^4\Up }{\dd\psi^4}+4\beta^2\frac{\dd^2\Up}{\dd \psi^2}\Big)=\bar{R}_s^+-2\beta \frac{\dd\Xi }{\dd \psi}\frac{\dd^2 \Xi}{\dd \psi^2},\\
&\Up(\pm1)=\Up'(\pm1)=0.
\end{aligned}
\right.
\end{equation}
Let 
$$B=\bigg\{\Xi\in H^4(-1,1):\big\|\Xi\big\|_{\X}\leqslant \Big(\eq\Big)^{n+\frac{3}{2}}Q\bigg\}.$$
We will show that $\T$ is a contractive mapping on $B$ provided that $\big\|\bar{R}_s^+\big\|_{L^\infty}\leqslant C\Big(\eq\Big)^{n+2}Q^2.$ Set $F^+=\bar{R}_s^+-2\beta\Xi_\psi \Xi_{\psi\psi}$. By virtue of Lemma \ref{inflinest}, it holds
$$\big\|\Up_\psi\big\|_{L^2(-1,1)}+\eq\big\|\Up_{\psi\psi}\big\|_{L^2(-1,1)}\leqslant \frac{C}{Q}\big\|F^+\big\|_{L^2(-1,1)}.$$
It follows from (\ref{Xieup}) that
 \begin{align}\nonumber
\begin{aligned}
\frac{\e}{Q}\big\|\Up_{\psi\psi\psi\psi}\big\|_{L^2(-1,1)}\leqslant& \frac{2\beta}{Q}\big\|\BA\big\|_{L^\infty(-1,1)}\big\|\Up_{\psi\psi}\big\|_{L^2(-1,1)}+\frac{2\beta}{Q}\big\|\bar{A}_{s\psi}^+\big\|_{L^\infty(-1,1)}\big\|\Up_{\psi}\big\|_{L^2(-1,1)}\\
&+\frac{4\beta^2\e}{Q}\big\|\Up_{\psi\psi}\big\|_{L^2(-1,1)}+\frac{1}{Q}\big\|F^+\big\|_{L^2(-1,1)}\\
\leqslant& C\big\|\Up_{\psi\psi}\big\|_{L^2(-1,1)}+\qe\big\|\Up_{\psi}\big\|_{L^2(-1,1)}\\
&+\frac{C\e}{Q}\big\|\Up_{\psi\psi}\big\|_{L^2(-1,1)}+\frac{1}{Q}\big\|F^+\big\|_{L^2(-1,1)}\\
\leqslant& \frac{C}{Q}\qe\big\|F^+\big\|_{L^2(-1,1)},
\end{aligned}
\end{align}
namely, $\big(\frac{\e}{Q}\big)^\frac{3}{2}\big\|\Up_{\psi\psi\psi\psi}\big\|_{L^2(-1,1)}\leqslant \frac{C}{Q}\big\|F^+\big\|_{L^2(-1,1)}$. The Gagliardo-Nirenberg inequality implies
 \begin{align}\nonumber
\begin{aligned}
\frac{\e}{Q}\big\|\Up_{\psi\psi\psi}\big\|_{L^2(-1,1)}\leqslant &\frac{C\e}{Q}\big\|\Up_{\psi\psi}\big\|_{L^2(-1,1)}+\frac{C\e}{Q}\big\|\Up_{\psi\psi}\big\|_{L^2(-1,1)}^\frac{1}{2}\big\|\Up_{\psi\psi\psi\psi}\big\|_{L^2(-1,1)}^\frac{1}{2}\\
\leqslant &\frac{C}{Q}\big\|F^+\big\|_{L^2(-1,1)}.
\end{aligned}
\end{align}
Hence one gets
$$\big\|\Up\big\|_{\X}\leqslant\frac{C}{Q}\big\|F^+\big\|_{L^2(-1,1)}.$$
Note that
 \begin{align}\nonumber
\begin{aligned}
2\beta\big\|\Xi_\psi\Xi_{\psi\psi}\big\|_{L^2(-1,1)}\leqslant& C\big\|\Xi_\psi\big\|_{L^\infty(-1,1)}\big\|\Xi_{\psi\psi}\big\|_{L^2(-1,1)}\\
 \leqslant &C\Big(\big\|\Xi_\psi\big\|_{L^2(-1,1)}+\big\|\Xi_{\psi}\big\|^\frac{1}{2}_{L^2(-1,1)}\big\|\Xi_{\psi\psi}\big\|^\frac{1}{2}_{L^2(-1,1)}\Big)\big\|\Xi_{\psi\psi}\big\|_{L^2(-1,1)} \\
\leqslant& C\Big(\frac{\e}{Q}\Big)^{-\frac{3}{4}}\big\|\Xi\big\|_{\X}^2.
\end{aligned}
\end{align}
If $\Xi\in B$, then
 \begin{align}\nonumber
\begin{aligned}
\big\|\Up\big\|_{\X}\leqslant &\frac{1}{Q}\big\|\bar{R}_s^+\big\|_{L^2(-1,1)}+\frac{2\beta}{Q}\big\|\Xi_\psi\Xi_{\psi\psi}\big\|_{L^2(-1,1)}\\
\leqslant &\frac{C}{Q}\big\|\bar{R}_s^+\big\|_{L^\infty(-1,1)}+\frac{C}{Q}\Big(\frac{\e}{Q}\Big)^{-\frac{3}{4}}\big\|\Xi\big\|_{\X}^2\\
\leqslant&C\Big(\eq\Big)^{n+2}Q+C\Big(\eq\Big)^{2n+\frac{3}{2}}Q\\
\leqslant&\Big(\eq\Big)^{n+\frac{3}{2}}Q
\end{aligned}
\end{align}
for $\eq$ small enough. Thus $\T(B)\subset B$. If $\Xi_1$, $\Xi_2\in B$,
\begin{align}\nonumber
\begin{aligned}
\big\|\T(\Xi_1-\Xi_2)\big\|_{\X}\leqslant &\frac{C}{Q}\big\|\Xi_{1\psi}\Xi_{1\psi\psi}-\Xi_2\Xi_{2\psi\psi}\big\|_{L^2(-1,1)}\\
\leqslant&\frac{C}{Q}\big\|\Xi_{1\psi}\big(\Xi_{1\psi\psi}-\Xi_{2\psi\psi}\big)\big\|_{L^2(-1,1)}+\frac{C}{Q}\big\|\big(\Xi_{1\psi}-\Xi_{2\psi}\big)\Xi_{2\psi\psi}\big\|_{L^2(-1,1)}\\
\leqslant&\frac{C}{Q}\big\|\Xi_{1\psi}\big\|_{L^\infty(-1,1)}\big\|\Xi_{1\psi\psi}-\Xi_{2\psi\psi}\big\|_{L^2(-1,1)}\\
 &+\frac{C}{Q}\big\|\Xi_{1\psi}-\Xi_{2\psi}\big\|_{L^\infty(-1,1)}\big\|\Xi_{2\psi\psi}\big\|_{L^2(-1,1)}\\
\leqslant& \frac{C}{Q}\Big(\frac{\e}{Q}\Big)^{-\frac{3}{4}}\Big(\big\|\Xi_1\big\|_{\X}+\big\|\Xi_2\big\|_{\X}\Big)\big\|\Xi_{1}-\Xi_{2}\big\|_{\X}\\
\leqslant& C\Big(\eq\Big)^{n}\big\|\Xi_{1}-\Xi_{2}\big\|_{\X},
\end{aligned}
\end{align}
so $\T$ is a contractive mapping on $B$ if $\eq$ is small enough. Thus, the contraction mapping theorem implies that there exists a $\Xi\in B$ such that $\T(\Xi)=\Xi$. Thus the problem (\ref{Xien}) admits a solution $\Xi\in B$. Finally, note that
\begin{align}\nonumber
\begin{aligned}
&\big\|\mathcal{A}\big\|_{L^\infty(-1,1)}= \big\|\Xi_\psi\big\|_{L^\infty(-1,1)}\leqslant C\Big(\frac{\e}{Q}\Big)^{-\frac{1}{4}}\big\|\Xi\big\|_{\X}\leqslant C\Big(\eq\Big)^{n+1}Q,\\
&\big\|\mathcal{A}_\psi\big\|_{L^\infty(-1,1)}= \big\|\Xi_{\psi\psi}\big\|_{L^\infty(-1,1)}\leqslant C\Big(\frac{\e}{Q}\Big)^{-\frac{3}{4}}\big\|\Xi\big\|_{\X}\leqslant C\Big(\eq\Big)^{n}Q,\\
&\big\|\mathcal{A}_{\psi\psi}\big\|_{L^\infty(-1,1)}= \big\|\Xi_{\psi\psi\psi}\big\|_{L^\infty(-1,1)}\leqslant C\Big(\frac{\e}{Q}\Big)^{-\frac{5}{4}}\big\|\Xi\big\|_{\X}\leqslant C\Big(\eq\Big)^{n-1}Q,
\end{aligned}
\end{align}
and by the equation in (\ref{mApe}),
\begin{align}\nonumber
\begin{aligned}
&\frac{\e}{Q}\big\|\mathcal{A}_{\psi\psi\psi}\big\|_{L^\infty(-1,1)}\\
=&\frac{1}{Q}\big\|-\bar{R}_s^++2\beta\mathcal{A}\mathcal{A}_{\psi}+2\beta \BA\mathcal{A}_{\psi}+2\beta \bar{A}_{s\psi}^+\mathcal{A}+4\e\beta^2\mathcal{A}_\psi\big\|_{L^\infty(-1,1)}\\
\leqslant&C\Big(\eq\Big)^{n}Q.
\end{aligned}
\end{align}
Hence we conclude
\begin{align}\nonumber
\begin{aligned}
&\Big\|\frac{\dd^l}{\dd \psi^l}\big[A-\bar{A}_s^+\big]\Big\|_{L^\infty}\leqslant C\Big(\eq\Big)^{n+1-l}Q,
\end{aligned}
\end{align}
for $l=0$, 1, 2, 3. The above estimate and (\ref{apesti2}) imply (\ref{apesti}), which completes the proof.
\qed

Similar result holds for $\xi\rightarrow-\infty$, and denote by $A^-$ a solution to the following problem:
\begin{equation}\label{Ame}
\left\{
\begin{aligned}
&-2\al A^-\frac{\dd A^-}{\dd \psi}-\e\Big(\frac{\dd^3 A^-}{\dd\psi^3}+4\al^2\frac{\dd A^-}{\dd \psi}\Big)=0,\\
&A^-(-1)=A^-(1)=0, \\
&\int_{-1}^1 A^-(\psi)\dd\psi=-Q.
\end{aligned}
\right.
\end{equation}
Now we can modify the approximate solution as
\begin{align}\nonumber
\begin{aligned}
\bar{\Phi}_s(\xi,\psi)=&\Phi_s(\xi,\psi)+\chi_1(\xi)\Big(-\frac{Q}{2}-\int_0^\psi A^+(\psi')\dd\psi'-\Phi_s^+(\psi)\Big)\\
&+\chi_1(-\xi)\Big(-\frac{Q}{2}-\int_0^\psi A^-(\psi')\dd\psi'-\Phi_s^-(\psi)\Big),
\end{aligned}
\end{align}
where $\Phi^{\pm}_s=\underset{\xi\rightarrow\pm\infty}{\lim}\Phi_s$ and $\chi_1$ is a smooth cut-off function satisfying $\chi_1|_{(-\infty,0]}=0,$ $\chi_1|_{[1,\infty)}=1$, $0\leqslant\chi_1\leqslant1$. $\bar{\Phi}_s$ satisfies also the boundary conditions:
$$\bar{\Phi}_s|_{\psi=\pm1}=\pm\frac{Q}{2},\quad \bar{\Phi}_{s\psi}|_{\psi=\pm1}=0.$$
And it has the asymptotic behaviors $\underset{\xi\rightarrow\pm\infty}{\lim}[-\bar{\Phi}_{s\psi},\bar{\Phi}_{s\xi}]=[A^{\pm},0]$. Define
\begin{align}\nonumber
\begin{aligned}
\bar{R}_s:=&\bar{\Phi}_{s\psi}\Delta_{\xi,\psi}\bar{\Phi}_{s\xi}+\frac{J_\xi}{J}\bar{\Phi}_{s\psi}\Delta_{\xi,\psi}\bar{\Phi}_{s}-\bar{\Phi}_{s\xi}\Delta_{\xi,\psi}\bar{\Phi}_{s\psi}-\frac{J_\psi}{J}\bar{\Phi}_{s\xi}\Delta_{\xi,\psi}\bar{\Phi}_{s}\\
&+\frac{\e}{J}\Delta_{\xi,\psi}\big(J\Delta_{\xi,\psi}\bar{\Phi}_s\big).
\end{aligned}
\end{align}
Direct computations give that for $\xi\geqslant 1$,
\begin{align}\nonumber
\bar{R}_s=&R_s+2\beta\Phi_{s\psi}^+\Phi_{s\psi\psi}^+-\e\Phi_{s\psi\psi\psi\psi}^+-4\beta^2\e\Phi_{s\psi\psi}^+\\\nonumber
  &-2\beta A^+A^+_{\psi}-\e A^+_{\psi\psi\psi}-4\beta^2\e A^+_{\psi}\\\nonumber
  &+\big(-\Phi_{s\psi}^+-A^+\big)\Delta_{\xi,\psi}\Phi_{s\xi}+\Big(\frac{J_\xi}{J}+2\beta\Big)\big(-\Phi_{s\psi}^+\Phi_{s\psi\psi}^++A^+A_{\psi}^+\big)\\\nonumber
&+\frac{J_\xi}{J}\big(\Phi_{s\psi}-\Phi_{s\psi}^+\big)\big(-\Phi_{s\psi\psi}^+-A_\psi^+\big)+\frac{J_\xi}{J}\big(-\Phi_{s\psi}^+-A^+\big)\big(\Phi_{s\psi\psi}-\Phi_{s\psi\psi}^+\big)\\\nonumber
&+\Phi_{s\xi}\Delta_{\xi,\psi}\big(-\Phi_{s\psi}^+-A^+\big)-\frac{J_\psi}{J}\Phi_{s\xi}\big(-\Phi_{s\psi\psi}^+-A^+_\psi\big)\\\nonumber
&+2\e \frac{J_\psi}{J}\big(-\Phi_{s\psi\psi\psi}^+-A^+_{\psi\psi}\big)+\e \Big(\frac{J_{\xi\xi}}{J}-4\beta^2+\frac{J_{\psi\psi}}{J}\Big)\big(-\Phi_{s\psi\psi}^+-A^+_{\psi}\big).
\end{align}
Recall the facts that
\begin{align}\nonumber
 &\underset{\xi\rightarrow\infty}{\lim}R_s=-2\beta\Phi^+_{s\psi}\Phi^+_{s\psi\psi}+\e\Phi^+_{s\psi\psi\psi\psi}+4\beta\e\Phi_{s\psi\psi}^+,
\end{align}
and 
\begin{align}\nonumber
-2\beta A^+A^+_{\psi}-\e A^+_{\psi\psi\psi}-4\beta^2\e A^+_{\psi}=0.
\end{align}
Combining these with that $\frac{J_\xi}{J}+2\beta$, $\frac{J_{\xi\xi}}{J}-4\beta^2$, $\frac{J_\psi}{J}$, $\frac{J_{\psi\psi}}{J}$, $\Phi_{s\xi}$ and $\Phi_{s\psi}-\Phi_{s\psi}^+$ decay to $0$ exponentially as $\xi\rightarrow\infty$, and $\big\|\frac{\dd^l}{\dd \psi^l}[-\Phi_{s\psi}^+-A^+]\big\|_\infty\leqslant  C_l\Big(\eq\Big)^{n+1-l}Q$, we conclude 
$$\|\bar{R}_s\|_{L^2\big(\mathbb{R}\times(-1,1)\big)}\leqslant C\Big(\eq\Big)^nQ^2.$$

\section{Estimates of the remainder and stability analysis}\label{stability}
We will establish the nonlinear stability analysis to estimate the deviation of the viscous flow from the approximate solution. We start with studying some properties of the approximate solution. Let $[\UB,\VB]=[-\bar{\Phi}_{s\psi},\bar{\Phi}_{s\xi}]$, and we just focus on their behaviors in the region $-1\leqslant\psi\leqslant0$. It follows from the construction that
\begin{align}\nonumber
\begin{aligned}
\UB(\xi,\psi)=&-\Phi_{e\psi}^0(\xi,\psi)-Q\Phi_{b\eta}^0\big(\xi,\qe(\psi+1)\big)\chi_2(\psi+1)+\MO\Big(\eq Q\Big)\\
             =&-Q\Big(\frac{1}{2}-u_b^0\big(\xi,\qe(\psi+1)\big)\chi_2(\psi+1)\Big)+\MO\Big(\eq Q\Big).
\end{aligned}
\end{align}
Since $\chi_2(\psi+1)=1$ for $0\leqslant\psi+1\leqslant\frac{1}{2}$ and $u_b^0(\xi,\eta)=\MO(e^{-c_0\eta})$, one has
$$u_b^0\big(\xi,\qe(\psi+1)\big)\big(\chi_2(\psi+1)-1\big)=\MO(e^{-c_0\qe(\psi+1)})=\MO(e^{-\frac{c_0}{2}\qe}),$$
thus, $$\Big|u_b^0\big(\xi,\qe(\psi+1)\big)\chi_2(\psi+1)-u_b^0\big(\xi,\qe(\psi+1)\big)\Big|\leqslant C_M\Big(\eq\Big)^M,$$
for any $M>0$. It thus follows that
\begin{align}\nonumber
\begin{aligned}
\UB(\xi,\psi)=&-Q\Big(\frac{1}{2}-u_b^0\big(\xi,\qe(\psi+1)\big)\Big)+\MO\Big(\eq Q\Big)\\
=&Q u_p\big(\xi,\qe(\psi+1)\big)+\MO\Big(\eq Q\Big).
\end{aligned}
\end{align}
For $\VB$, since $\Phi_e^0=\frac{Q\psi}{2}$, one can get
\begin{align}\nonumber
\begin{aligned}
\VB(\xi,\psi)=&\Phi_{e\xi}^0(\xi,\psi)+\eq Q\Phi_{b\xi}^0\big(\xi,\qe(\psi+1)\big)\chi_2(\psi+1)\\
              &+\eq \Phi_{e\xi}^1(\xi,\psi)+\MO\bigg(\Big(\eq\Big)^2Q\bigg)\\
              =&\eq Q\Phi_{b\xi}^0\big(\xi,\qe(\psi+1)\big)\big(\chi_2(\psi+1)-1\big)\\
              &+\eq \Big(Q\Phi_{b\xi}^0\big(\xi,\qe(\psi+1)\big)+\Phi_{e\xi}^1(\xi,\psi)\Big)+\MO\bigg(\Big(\eq\Big)^2Q\bigg)\\
              =&\eq \Big(Q\Phi_{b\xi}^0\big(\xi,\qe(\psi+1)\big)+\Phi_{e\xi}^1(\xi,\psi)\Big)+\MO\bigg(\Big(\eq\Big)^2Q\bigg).
\end{aligned}
\end{align}
Since $\Phi_e^1(\xi,-1)+Q\Phi_b^0(\xi,0)=0$, one can get
\begin{align}\nonumber
\begin{aligned}
&Q\Phi_{b\xi}^0\big(\xi,\qe(\psi+1)\big)+\Phi_{e\xi}^1(\xi,\psi)\\
=&Q\Phi_{b\xi}^0\big(\xi,\qe(\psi+1)\big)-Q\Phi_{b\xi}^0\big(\xi,0\big)+\Phi_{e\xi}^1(\xi,\psi)-\Phi_{e\xi}^1(\xi,0)+\MO\Big(\eq Q\Big)\\
=&Q\int_0^{\qe(\psi+1)}\Phi_{b\xi\eta}^0(\xi,\eta')\dd\eta'+\Phi_{e\xi}^1(\xi,\psi)-\Phi_{e\xi}^1(\xi,0)+\MO\Big(\eq Q\Big)\\
=&-Q\int_0^{\qe(\psi+1)}u_{p\xi}(\xi,\eta')\dd\eta'+\Phi_{e\xi}^1(\xi,\psi)-\Phi_{e\xi}^1(\xi,0)+\MO\bigg(\Big(\eq\Big)^2Q\bigg)\\
=&Qv_p\big(\xi,\qe(\psi+1)\big)+\Phi_{e\xi}^1(\xi,\psi)-\Phi_{e\xi}^1(\xi,0)+\MO\Big(\eq Q\Big).
\end{aligned}
\end{align}
Hence
\begin{align}\nonumber
\begin{aligned}
\VB(\xi,\psi)=&\eq Qv_p\big(\xi,\qe(\psi+1)\big)+\eq\big(\Phi_{e\xi}^1(\xi,\psi)-\Phi_{e\xi}^1(\xi,0)\big)\\
&+\MO\bigg(\Big(\eq\Big)^2Q\bigg).
\end{aligned}
\end{align}
For $\UB_\psi$ and $\VB_\psi$, it holds that $\UB_{\psi}=\MO\Big(\qe Q\Big)$ and $\VB_{\psi}=\MO\big( Q\big)$. And in the region $\psi\leqslant0$, 
\begin{align}\nonumber
\begin{aligned}
(\psi+1)\UB_{\psi}(\xi,\psi)=&Q(\psi+1)\qe  u_{p\eta}\big(\xi,\qe(\psi+1)\big)+\MO\big(Q\big)\\
=&Q\eta u_{p\eta}(\xi,\eta)+\MO\big(Q\big)\\
=&\MO\big(Q\big)
\end{aligned}
\end{align}
since $u_{p\eta}$ decays fast in $\eta$ as $\eta$ goes to $\infty$, where $\eta=\qe(\psi+1)$. In the whole region $-1\leqslant\psi\leqslant1$,
$$(\psi+1)(1-\psi)\UB_{\psi}(\xi,\psi)=\MO\big(Q\big).$$
For $\UB_\xi$, in the region $\psi\leqslant0$, one has
\begin{align}\nonumber
\begin{aligned}
(\psi+1)\UB_{\xi}(\xi,\psi)=&Q(\psi+1)\qe  u_{p\xi}\big(\xi,\qe(\psi+1)\big)+\MO\Big(\eq Q\Big)\\
=&Q\eta u_{p\xi}(\xi,\eta)+\MO\Big(\eq Q\Big)\\
=&\MO\Big(\eq Q\Big).
\end{aligned}
\end{align}
Consequently, in the the whole region $-1\leqslant\psi\leqslant1$, it holds that
$$(\psi+1)(1-\psi)\UB_{\xi}(\xi,\psi)=\MO\Big(\eq Q\Big),$$
and it follows from the divergence-free condition that,
$$(\psi+1)(1-\psi)\VB_{\psi}(\xi,\psi)=-(\psi+1)(1-\psi)\UB_{\xi}(\xi,\psi)=\MO\Big(\eq Q\Big).$$
Similarly, one has that 
\begin{align}\nonumber
\begin{aligned}
&(\psi+1)^2(1-\psi)^2\UB_{\psi}(\xi,\psi)=\MO\Big(\eq Q\Big),\\
&(\psi+1)^2(1-\psi)^2\VB_{\psi}(\xi,\psi)=-(\psi+1)^2(1-\psi)^2\UB_{\xi}(\xi,\psi)=\MO\bigg(\Big(\eq\Big)^2 Q\bigg).
\end{aligned}
\end{align}
Since $\bar{\Phi}_s|_{\psi=\pm1}=\pm\frac{Q}{2}$ and $\bar{\Phi}_{s\psi}|_{\psi=\pm1}=0$, thus $[\UB,\VB]_{\psi=\pm1}=0$. Finally, we show that $\UB$ is non-positive and $\pm\UB_{\psi}|_{\psi=\pm1}>0$. Noticing that $Qu_b^m(\xi,0)=\Phi_{e\psi}^m|_{\psi=-1}$, for $\psi+1\leqslant\eta_0\eq$, namely, $\eta\leqslant\eta_0$, one can get
\begin{align}\nonumber
\begin{aligned}
\UB=&Qu_p\big(\xi,\qe(\psi+1)\big)+\MO\Big(Q(\psi+1)\Big)\\
\leqslant&-m_0Q\qe(\psi+1)+\MO\Big(Q(\psi+1)\Big)\\
\leqslant&-\frac{m_0}{2}Q\qe(\psi+1),
\end{aligned}
\end{align}
where one has used the fact that $u_{p\eta}(\xi,\eta)\leqslant -m_0<0$ for $\eta\leqslant\eta_0$. And in the region $\eta_0\eq\leqslant\psi+1\leqslant0$, since $u_p(\xi,\eta)\leqslant-m_0\eta_0<0$ for $\eta\geqslant\eta_0$, it hold that
\begin{align}\nonumber
\begin{aligned}
\UB=&Qu_p\big(\xi,\qe(\psi+1)\big)+\MO\Big(\eq Q\Big)\\
\leqslant&-m_0\eta_0Q+\MO\Big(\eq Q\Big)\\
\leqslant&-\frac{m_0\eta_0}{2}Q.
\end{aligned}
\end{align}
For $\UB_\psi$ at $\psi=-1,$ one has
\begin{align}\nonumber
\begin{aligned}
\UB_{\psi}(\xi,-1)=&Q\qe u_{p\eta}(\xi,0)+\MO\big(Q\big)\\
\leqslant&-m_0Q\qe+\MO\big(Q\big)\\
\leqslant&-\frac{m_0}{2}Q\qe.
\end{aligned}
\end{align}
Similar estimates hold on the boundary $\psi=1$.

Now we are ready to establish the stability analysis by estimating the remainder which is the deviation of the viscous flow from the approximation solution. Thus set $\Pi=\Phi-\bar{\Phi}_s$. Then $\Pi$ solves the following problem in $(-\infty,\infty)\times(-1,1)$:
\begin{equation}\label{NPie}
\left\{
\begin{aligned}
&J\bar{\Phi}_{s\psi}\p_\xi\big(J\Delta_{\xi,\psi}\Pi\big)+J\Pi_{\psi}\p_\xi\big(J\Delta_{\xi,\psi}\bar{\Phi}_s\big)-J\bar{\Phi}_{s\xi}\p_\psi\big(J\Delta_{\xi,\psi}\Pi\big)-J\Pi_\xi\p_\psi\big(J\Delta_{\xi,\psi}\bar{\Phi}_s\big)\\
&+\e J\Delta_{\xi,\psi}\big(J\Delta_{\xi,\psi}\Pi\big)=-J^2\bar{R}_s-J^2N(\Pi), \\
&\Pi|_{\psi=1}=\Pi|_{\psi=-1}=0,\\
&\Pi_\psi|_{\psi=1}=\Pi_\psi|_{\psi=-1}=0,
\end{aligned}
\right.
\end{equation}
where $N(\Pi)$ is the nonlinear term:
$$N(\Pi)=\frac{1}{J}\Pi_\psi\p_\xi\big(J\Delta_{\xi,\psi}\Pi\big)-\frac{1}{J}\Pi_\xi\p_\psi\big(J\Delta_{\xi,\psi}\Pi\big).$$
Multiply the above equation by $\frac{1}{J^2}$ and then consider the linearized problem:
\begin{equation}\label{Pie}
\left\{
\begin{aligned}
&\frac{1}{J}\bar{\Phi}_{s\psi}\p_\xi\big(J\Delta_{\xi,\psi}\Pi\big)+\frac{1}{J}\Pi_{\psi}\p_\xi\big(J\Delta_{\xi,\psi}\bar{\Phi}_s\big)-\frac{1}{J}\bar{\Phi}_{s\xi}\p_\psi\big(J\Delta_{\xi,\psi}\Pi\big)-\frac{1}{J}\Pi_\xi\p_\psi\big(J\Delta_{\xi,\psi}\bar{\Phi}_s\big)\\
&+\e \frac{1}{J}\Delta_{\xi,\psi}\big(J\Delta_{\xi,\psi}\Pi\big)=F, \\
&\Pi|_{\psi=1}=\Pi|_{\psi=-1}=0,\\
&\Pi_\psi|_{\psi=1}=\Pi_\psi|_{\psi=-1}=0.
\end{aligned}
\right.
\end{equation}
One of the key ingredients in our stability analysis is the following estimate.
\begin{proposition}\label{Keyp}
Under the assumptions in Theorem \ref{main}, if $\Pi\in H^4\big(\mathbb{R}\times(-1,1)\big)$ solves (\ref{Pie}) with $F\in L^2\big(\mathbb{R}\times(-1,1)\big)$, then for $\eq$ suitably small, it holds that
\begin{align}\label{KeyPie}
\begin{aligned}
\big\|\nabla_{\xi,\psi}\Pi\big\|+\eq\Big\| \nabla^2_{\xi,\psi}\Pi\Big\|\leqslant \frac{C}{Q}\big\|F\big\|.
\end{aligned}
\end{align}
\end{proposition}
\begin{proof}
Set $G:=\Pi/\UB$. Since $\UB$ is negative for $-1<\psi<1$ and $\pm \UB_\psi|_{\psi=\pm1}>0$, $G$ is well-defined. And $G|_{\psi=\pm1}=0$ follows from $\Pi|_{\psi=\pm1}=\Pi_\psi|_{\psi=\pm1}=0$. Take the $L^2\big(\mathbb{R}\times(-1,1)\big)$ inner product of the first equation in (\ref{Pie}) with $-G$ to get
\begin{align}\nonumber
\begin{aligned}
&\Bz\frac{1}{J}\bar{\Phi}_{s\psi}\p_\xi\big(J\Delta_{\xi,\psi}\Pi\big),-G\By+\Bz\frac{1}{J}\Pi_{\psi}\p_\xi\big(J\Delta_{\xi,\psi}\bar{\Phi}_s\big),-G\By-\Bz\frac{1}{J}\bar{\Phi}_{s\xi}\p_\psi\big(J\Delta_{\xi,\psi}\Pi\big),-G\By\\
&-\Bz\frac{1}{J}\Pi_\xi\p_\psi\big(J\Delta_{\xi,\psi}\bar{\Phi}_s\big),-G\By+\e\Bz \frac{1}{J}\Delta_{\xi,\psi}\big(J\Delta_{\xi,\psi}\Pi\big),-G\By=\bz F,-G\by.
\end{aligned}
\end{align}
By integrating by parts, one has 
\begin{align}\nonumber
\begin{aligned}
I_1:=&\Bz\frac{1}{J}\bar{\Phi}_{s\psi}\p_\xi\big(J\Delta_{\xi,\psi}\Pi\big),-G\By=\Bz\frac{1}{J}\UB\p_\xi\big(J\Delta_{\xi,\psi}\Pi\big),\frac{\Pi}{\UB}\By\\
=&-\Bz\frac{J_\xi}{J},\Pi_\xi^2\By-\Bz\frac{J_\xi}{J},\Pi_\psi^2\By-\Bz\frac{1}{2}\Delta_{\xi,\psi}\Big[\frac{J_\xi}{J}\Big],\Pi^2\By.
\end{aligned}
\end{align}
Since $\frac{J_\xi}{J}$ is a harmonic function, it holds that 
\begin{align}\nonumber
I_1=&-\Bz\frac{J_\xi}{J},\Pi_\xi^2\By-\Bz\frac{J_\xi}{J},\Pi_\psi^2\By\\\nonumber
=&-\Bz\frac{J_\xi}{J}, \UB^2G_\xi^2+\UB_\xi^2G^2+2\UB\UB_\xi  GG_\xi\By\\\nonumber
&-\Bz\frac{J_\xi}{J},\UB^2G_\psi^2+\UB_\psi^2G^2+2\UB\UB_\psi  GG_\psi\By\\\nonumber
=&-\Bz\frac{J_\xi}{J}\UB^2,G_\xi^2\By-\Bz\frac{J_\xi}{J}\UB^2,G_\psi^2\By+\Bz\frac{J_\xi}{J}\UB\UB_{\psi\psi}, G^2\By\\\nonumber
&+\MO\Big(\Big\|\frac{J_\xi}{J}\Big\|_\infty\big\|(\psi+1)(1-\psi)\UB_\xi\big\|_\infty^2\Big\|\frac{G}{(\psi+1)(1-\psi)}\Big\|^2\Big)\\\nonumber
&+\MO\Big(\Big\|\frac{J_\xi}{J}\UB\Big\|_\infty\big\|(\psi+1)(1-\psi))\UB_\xi\big\|_\infty\big\|G_\xi\big\|\Big\|\frac{G}{(\psi+1)(1-\psi)}\Big\|\Big)\\\nonumber
&+\MO\Big(\Big\|\Big[\frac{J_\xi}{J}\Big]_\psi\UB\Big\|_\infty\big\|(\psi+1)^2(1-\psi)^2\UB_\psi\big\|_\infty\Big\|\frac{G}{(\psi+1)(1-\psi)}\Big\|^2\Big).\\\nonumber
\end{align}
It follows from the Hardy inequality that $\Big\|\frac{G}{(\psi+1)(1-\psi)}\Big\|\leqslant C\big\|G_\psi\big\|$, therefore
\begin{align}\nonumber
I_1=&-\Bz\frac{J_\xi}{J}\UB^2,G_\xi^2\By-\Bz\frac{J_\xi}{J}\UB^2,G_\psi^2\By+\Bz\frac{J_\xi}{J}\UB\UB_{\psi\psi}, G^2\By\\\nonumber
&+\MO\Big(\frac{\e}{Q} Q^2\big\|G_\psi\big\|^2+\eq Q^2\big\|G_\xi\big\|\big\|G_\psi\big\|+\eq Q^2\big\|G_\psi\big\|^2\Big)\\\label{I1}
=&-\Bz\frac{J_\xi}{J}\UB^2,G_\xi^2\By-\Bz\frac{J_\xi}{J}\UB^2,G_\psi^2\By+\Bz\frac{J_\xi}{J}\UB\UB_{\psi\psi}, G^2\By+\MO\Big(\eq Q^2\big\|\nabla G\big\|^2\Big).
\end{align}
Next,
\begin{align}\nonumber
\begin{aligned}
I_2:=&\Bz\frac{1}{J}\Pi_{\psi}\p_\xi\big(J\Delta_{\xi,\psi}\bar{\Phi}_s\big),-G\By\\
=&\Bz \Pi_{\psi}\VB_{\psi\psi},-G\By+\Bz \frac{J_\xi}{J}\Pi_{\psi}\UB_\psi,G\By+\Bz \Pi_{\psi}\VB_{\xi\xi},-G\By+\Bz \frac{J_\xi}{J}\Pi_{\psi}\VB_\xi,-G\By\\
=&\Bz \UB\VB_{\psi\psi},-G G_\psi\By+\Bz \UB_\psi\VB_{\psi\psi},-G^2\By+\Bz \frac{J_\xi}{J}\UB\UB_\psi,GG_\psi\By+\Bz \frac{J_\xi}{J}\UB^2_\psi,G^2\By\\
&+\MO\Big(\big\|\VB_{\xi\xi}\big\|_\infty\big\|\Pi_\psi\big\|\big\|G\big\|+\Big\|\frac{J_\xi}{J}\VB_{\xi}\Big\|_\infty\big\|\Pi_\psi\big\|\big\|G\big\|\Big)\\
=&\Bz \frac{1}{2}\UB\VB_{\psi\psi\psi}-\frac{1}{2}\UB_\psi\VB_{\psi\psi},G^2\By+\Bz \frac{1}{2}\frac{J_\xi}{J}\Big(\UB^2_\psi-\UB\UB_{\psi\psi}\Big),G^2\By\\
&-\Bz\frac{1}{2}\Big[\frac{J_\xi}{J}\Big]_\psi \UB\UB_\psi, G^2\By+\MO\Big(\big\|\VB_{\xi\xi}\big\|_\infty\big\|\Pi_\psi\big\|\big\|G\big\|+\Big\|\frac{J_\xi}{J}\VB_{\xi}\Big\|_\infty\big\|\Pi_\psi\big\|\big\|G\big\|\Big).
\end{aligned}
\end{align}
Noticing that
\begin{align}\nonumber
\begin{aligned}
\big\|\Pi_\psi\|\leqslant& C\big\|\UB\big\|_\infty\big\|G_\psi\big\|+C\big\|(1-\psi)(1+\psi)\UB_\psi\big\|_\infty\Big\|\frac{G}{(\psi+1)(1-\psi)}\Big\|\\
                       \leqslant &CQ\big\|G_\psi\big\|,
\end{aligned}
\end{align}
and $\big\|\VB_\xi\big\|_\infty+\big\|\VB_{\xi\xi}\big\|_\infty\leqslant C\eq Q,$ one can obtain
\begin{align}\nonumber
I_2=&\Bz \frac{1}{2}\UB\VB_{\psi\psi\psi}-\frac{1}{2}\UB_\psi\VB_{\psi\psi},G^2\By+\Bz \frac{1}{2}\frac{J_\xi}{J}\Big(\UB^2_\psi-\UB\UB_{\psi\psi}\Big),G^2\By\\\nonumber
&+\MO\Big(\Big\|\Big[\frac{J_\xi}{J}\Big]_\psi \UB\Big\|_\infty\big\|(\psi+1)^2(1-\psi)^2\UB_\psi\big\|_\infty\Big\|\frac{G}{(\psi+1)(1-\psi)}\Big\|^2\Big)\\\nonumber
&+\MO\Big(\eq Q^2\big\|G_\psi\big\|^2\Big)\\\label{I2}
=&\Bz \frac{1}{2}\UB\VB_{\psi\psi\psi}-\frac{1}{2}\UB_\psi\VB_{\psi\psi},G^2\By+\Bz \frac{1}{2}\frac{J_\xi}{J}\Big(\UB^2_\psi-\UB\UB_{\psi\psi}\Big),G^2\By+\MO\Big(\eq Q^2\big\|G_\psi\big\|^2\Big).
\end{align}
Next, we turn to estimate 
\begin{align}\nonumber
\begin{aligned}
I_3:=&-\Bz\frac{1}{J}\bar{\Phi}_{s\xi}\p_\psi\big(J\Delta_{\xi,\psi}\Pi\big),-G\By\\
    =&\Bz \VB \Delta_{\xi,\psi}\Pi_\psi,G\By+\Bz \frac{J_\psi}{J}\VB\Delta_{\xi,\psi}\Pi,G\By\\
    =&:I_{3,1}+I_{3,2}.
\end{aligned}
\end{align}
It will be shown that $I_{3,2}$ is bounded by $\eq Q^2\|\nabla G\|^2$. Indeed,
\begin{align}\nonumber
\begin{aligned}
I_{3,2}:=&\Bz \frac{J_\psi}{J}\VB\Delta_{\xi,\psi}\Pi,G\By\\
=&-\Bz \frac{J_\psi}{J}\VB\Pi_\xi,G_\xi\By-\Bz \Big[\frac{J_\psi}{J}\VB\Big]_\xi\Pi_\xi,G\By\\
&-\Bz \frac{J_\psi}{J}\VB\Pi_\psi,G_\psi\By-\Bz \Big[\frac{J_\psi}{J}\VB\Big]_\psi\Pi_\psi,G\By\\
=&\MO\Big(\Big\|\frac{J_\psi}{J}\VB\Big\|_\infty\big\|\Pi_\xi\big\|\big\|G_\xi\big\|+\Big\|\Big[\frac{J_\psi}{J}\VB\Big]_\xi\Big\|_\infty\big\|\Pi_\xi\big\|\big\|G\big\|\Big)\\
&+\MO\Big(\Big\|\frac{J_\psi}{J}\VB\Big\|_\infty\big\|\Pi_\psi\big\|\big\|G_\psi\big\|+\Big\|(\psi+1)(1-\psi)\Big[\frac{J_\psi}{J}\VB\Big]_\psi\Big\|_\infty\big\|\Pi_\xi\big\|\Big\|\frac{G}{(\psi+1)(1-\psi)}\Big\|\Big)\\
=&\eq Q^2\|\nabla G\|^2,
\end{aligned}
\end{align}
where $\VB=\MO\big(\eq Q\big)$ and $(\psi+1)(1-\psi)\VB_{\psi}=\MO\big(\eq Q\big)$ have been used. $I_{3,1}$ can be estimated as
\begin{align}\nonumber
I_{3,1}=&\Bz \VB\Pi_{\xi\xi\psi},G\By+\Bz \VB\Pi_{\psi\psi\psi},G\By\\\nonumber
=&-\Bz \VB\big(\UB G_{\xi\psi}+\UB_\psi G_\xi+\UB_\xi G_\psi+\UB_{\xi\psi}G\big),G_\xi\By\\\nonumber
&+\Bz \VB_\xi\Pi_\psi,G_\xi\By+\Bz \VB_{\xi\xi}\Pi_\psi,G\By\\\nonumber
&-\Bz \VB\big(\UB G_{\psi\psi}+2\UB_\psi G_\psi+\UB_{\psi\psi} G\big),G_\psi\By\\\nonumber
&+\Bz \VB_\psi\Pi_{\psi},G_\psi\By+\Bz \VB_{\psi\psi}\Pi_{\psi},G\By\\\nonumber
=&\Bz\frac{1}{2}\big(\VB\UB\big)_\psi, G^2_\xi\By-\Bz\VB\UB_\psi, G^2_\xi\By+\MO\Big(\big\|\VB\UB_\xi\big\|_\infty\big\|G_\xi\big\|\big\|G_\psi\big\|\Big)\\\nonumber
&+\MO\Big(\big\|(\psi+1)(1-\psi)\VB\UB_{\xi\psi}\big\|_\infty\big\|G_\xi\big\|\Big\|\frac{G}{(\psi+1)(1-\psi)}\Big\|\Big)\\\nonumber
&+\MO\Big(\big\|\VB_\xi\big\|_\infty\big\|\Pi_\psi\big\|\big\|G_\xi\big\|+\big\|\VB_{\xi\xi}\big\|_\infty\big\|\Pi_\psi\big\|\big\|G\big\|\Big)\\\nonumber
&+\Bz\frac{1}{2}\big(\VB\UB\big)_\psi, G^2_\psi\By-2\Bz\VB\UB_\psi, G^2_\psi\By+\Bz\frac{1}{2}\big(\VB\UB_{\psi\psi}\big)_\psi, G^2\By\\\nonumber
&+\Bz \VB_\psi\UB,G^2_\psi\By+\Bz \VB_\psi\UB_{\psi},GG_\psi\By+\Bz \VB_{\psi\psi}\UB,GG_\psi\By+\Bz \VB_{\psi\psi}\UB_{\psi},G^2\By\\\nonumber
=&\Bz\frac{1}{2}\UB\VB_\psi-\frac{1}{2}\VB\UB_\psi,G_\xi^2\By+\Bz\frac{3}{2}\UB\VB_\psi-\frac{3}{2}\VB\UB_\psi,G_\psi^2\By\\\nonumber
&+\Bz\frac{1}{2}\VB_\psi\UB_{\psi\psi}+\frac{1}{2}\VB\UB_{\psi\psi\psi}, G^2\By-\Bz\frac{1}{2} \big(\VB_\psi\UB_{\psi}\big)_\psi,G^2\By\\\nonumber
&-\Bz \frac{1}{2}\big(\VB_{\psi\psi}\UB\big)_\psi,G^2\By+\Bz \VB_{\psi\psi}\UB_{\psi},G^2\By+\MO\Big(\eq Q^2\|\nabla G\|^2\Big)\\\nonumber
=&\Bz\frac{1}{2}\UB\VB_\psi-\frac{1}{2}\VB\UB_\psi,G_\xi^2\By+\Bz\frac{3}{2}\UB\VB_\psi-\frac{3}{2}\VB\UB_\psi,G_\psi^2\By\\\nonumber
&+\Bz\frac{1}{2}\VB\UB_{\psi\psi\psi}-\frac{1}{2}\UB\VB_{\psi\psi\psi},G^2\By+\MO\Big(\eq Q^2\|\nabla G\|^2\Big).
\end{align}
Thus, it holds that
\begin{align}\label{I3}
\begin{aligned}
I_3=&\Bz\frac{1}{2}\UB\VB_\psi-\frac{1}{2}\VB\UB_\psi,G_\xi^2\By+\Bz\frac{3}{2}\UB\VB_\psi-\frac{3}{2}\VB\UB_\psi,G_\psi^2\By\\
&+\Bz\frac{1}{2}\VB\UB_{\psi\psi\psi}-\frac{1}{2}\UB\VB_{\psi\psi\psi},G^2\By+\MO\Big(\eq Q^2\|\nabla G\|^2\Big).
\end{aligned}
\end{align}
Next, one can estimate $I_4$ as 
\begin{align}\nonumber
I_4:=&-\Bz\frac{1}{J}\Pi_\xi\p_\psi\big(J\Delta_{\xi,\psi}\bar{\Phi}_s\big),-G\By\\\nonumber
=&-\Bz\Pi_\xi\big(\UB_{\psi\psi}+\UB_{\xi\xi}\big),G\By+\Bz\frac{J_\psi}{J}\big(-\UB_\psi+\VB_\xi\big)\Pi_\xi,G\By\\\nonumber
=&-\Bz\UB\UB_{\psi\psi},GG_\xi\By-\Bz\UB_\xi\UB_{\psi\psi},G^2\By\\\nonumber
&+\MO\Big(\big\|(\psi+1)(1-\psi)\UB_{\xi\xi}\big\|_\infty\big\|\Pi_\xi\big\|\Big\|\frac{G}{(\psi+1)(1-\psi)}\Big\|\Big)\\\nonumber
&-\Bz\frac{J_\psi}{J}\UB\UB_\psi,GG_\xi\By-\Bz\frac{J_\psi}{J}\UB_\xi\UB_\psi,G^2\By+\MO\Big(\Big\|\frac{J_\psi}{J}\VB_\xi\Big\|_\infty\big\|\Pi_\xi\big\|\big\|G\big\|\Big)\\\nonumber
=&\Bz\frac{1}{2}\big(\UB\UB_{\psi\psi}\big)_\xi, G^2\By-\Bz\UB_\xi\UB_{\psi\psi},G^2\By+\Bz\frac{1}{2}\Big[\frac{J_\psi}{J}\UB\UB_\psi\Big]_\xi,G^2\By\\\nonumber
&+\MO\Big(\Big\|(\psi+1)^2(1-\psi)^2\frac{J_\psi}{J}\UB_\xi\UB_\psi\Big\|_\infty\Big\|\frac{G}{(\psi+1)(1-\psi)}\Big\|^2\Big)\\\nonumber
&+\MO\Big(\eq Q^2\big\|\nabla G\big\|^2\Big)\\\nonumber
=&\Bz\frac{1}{2}\UB\UB_{\xi\psi\psi}-\frac{1}{2}\UB_\xi\UB_{\psi\psi}, G^2\By\\\nonumber
&+\MO\Big(\Big\|(\psi+1)^2(1-\psi)^2\Big[\frac{J_\psi}{J}\UB\UB_\psi\Big]_\xi\Big\|_\infty\Big\|\frac{G}{(\psi+1)(1-\psi)}\Big\|^2\Big)\\\nonumber
&+\MO\Big(\eq Q^2\big\|\nabla G\big\|^2\Big)\\\label{I4}
=&\Bz\frac{1}{2}\UB\UB_{\xi\psi\psi}-\frac{1}{2}\UB_\xi\UB_{\psi\psi}, G^2\By+\MO\Big(\eq Q^2\big\|\nabla G\big\|^2\Big).
\end{align}
Finally, one needs to estimate the term involving the bi-Laplacian term. 
\begin{align}\nonumber
I_5:=&-\e\Bz \frac{1}{J}\Delta_{\xi,\psi}\big(J\Delta_{\xi,\psi}\Pi\big),G\By\\\nonumber
=&-\e\Bz\Pi_{\xi\xi\xi\xi},G\By-\e\Bz2\Pi_{\xi\xi\psi\psi},G\By-\e\Bz\Pi_{\psi\psi\psi\psi},G\By\\\nonumber
&-\e\Bz \frac{2J_\xi}{J} \Delta_{\xi,\psi}\Pi_\xi,G\By-\e \Bz\frac{2J_\psi}{J} \Delta_{\xi,\psi}\Pi_\psi, G\By-\e\Bz\frac{\Delta_{\xi,\psi}J}{J} \Delta_{\xi,\psi}\Pi, G\By\\\nonumber
=&:I_{5,1}+I_{5,2}+I_{5,3}+I_{5,4}+I_{5,5}+I_{5,6}.
\end{align}
It will be shown that $I_{5,4}+I_{5,5}+I_{5,6}$ by $\MO\big(\eq Q^2\|\nabla G\|^2\big)$. Indeed,
\begin{align}\nonumber
I_{5,4}=&-\e\Bz \frac{2J_\xi}{J}\big(\Pi_{\xi\xi\xi}+\Pi_{\xi\psi\psi}\big),G\By\\\nonumber
=&\e\Bz \frac{2J_\xi}{J}\big(\UB G_{\xi\xi}+2\UB_\xi G_\xi+\UB_{\xi\xi} G\big),G_\xi\By\\\nonumber
&-\e\Bz \Big[\frac{2J_\xi}{J}\Big]_\xi\Pi_{\xi},G_\xi\By-\e\Bz \Big[\frac{2J_\xi}{J}\Big]_{\xi\xi}\Pi_{\xi},G\By\\\nonumber
&+\e\Bz \frac{2J_\xi}{J}\big(\UB G_{\xi\psi}+\UB_\xi G_\psi+\UB_\psi G_\xi+\UB_{\psi\xi} G\big),G_\psi\By\\\nonumber
&-\e\Bz \Big[\frac{2J_\xi}{J}\Big]_\psi\Pi_{\psi},G_\xi\By-\e\Bz \Big[\frac{2J_\xi}{J}\Big]_{\psi\xi}\Pi_{\psi},G\By\\\nonumber
=&\e\Bz \frac{2J_\xi}{J}\UB G_{\xi\xi},G_\xi\By+\e\Bz \frac{2J_\xi}{J}\UB G_{\xi\psi},G_\psi\By\\\nonumber
&+\MO\Big(\eq Q\big\|\nabla \Pi\big\|\big\|\nabla G\big\|+\eq Q^2\big\|\nabla G\big\|^2\Big)\\\nonumber
=&-\e\Bz \Big[\frac{J_\xi}{J}\UB\Big]_\xi,G^2_\xi\By-\e\Bz\Big[ \frac{J_\xi}{J}\UB\Big]_\xi ,G^2_\psi\By+\MO\Big(\eq Q^2\big\|\nabla G\big\|^2\Big)\\\nonumber
=&\MO\Big(\eq Q^2\big\|\nabla G\big\|^2\Big),
\end{align}
similarly,
\begin{align}\nonumber
I_{5,5}=&-\e\Bz \frac{2J_\psi}{J}\big(\Pi_{\xi\xi\psi}+\Pi_{\psi\psi\psi}\big),G\By\\\nonumber
=&\e\Bz \frac{2J_\psi}{J}\big(\UB G_{\xi\psi}+\UB_\xi G_\psi+\UB_\psi G_\xi+\UB_{\xi\psi} G\big),G_\xi\By\\\nonumber
&-\e\Bz \Big[\frac{2J_\psi}{J}\Big]_\xi\Pi_{\xi},G_\psi\By-\e\Bz \Big[\frac{2J_\psi}{J}\Big]_{\xi\psi}\Pi_{\xi},G\By\\\nonumber
&+\e\Bz \frac{2J_\xi}{J}\big(\UB G_{\psi\psi}+2\UB_\psi G_\psi+\UB_{\psi\psi} G\big),G_\psi\By\\\nonumber
&-\e\Bz \Big[\frac{2J_\psi}{J}\Big]_\psi\Pi_{\psi},G_\psi\By-\e\Bz \Big[\frac{2J_\psi}{J}\Big]_{\psi\psi}\Pi_{\psi},G\By\\\nonumber
=&\e\Bz \frac{2J_\psi}{J}\UB G_{\xi\psi},G_\xi\By+\e\Bz \frac{2J_\psi}{J}\UB G_{\psi\psi},G_\psi\By\\\nonumber
&+\MO\Big(\eq Q\big\|\nabla \Pi\big\|\big\|\nabla G\big\|+\eq Q^2\big\|\nabla G\big\|^2\Big)\\\nonumber
=&-\e\Bz \Big[\frac{J_\psi}{J}\UB\Big]_\psi,G^2_\xi\By-\e\Bz\Big[ \frac{J_\psi}{J}\UB\Big]_\psi ,G^2_\psi\By+\MO\Big(\eq Q^2\big\|\nabla G\big\|^2\Big)\\\nonumber
=&\MO\Big(\eq Q^2\big\|\nabla G\big\|^2\Big),
\end{align}
and 
\begin{align}\nonumber
I_{5,6}=&-\e\Bz\frac{\Delta_{\xi,\psi}J}{J} \Pi_{\xi\xi}, G\By-\e\Bz\frac{\Delta_{\xi,\psi}J}{J} \Pi_{\psi\psi}, G\By\\\nonumber
=&\e\Bz\frac{\Delta_{\xi,\psi}J}{J} \Pi_{\xi}, G_\xi\By+\e\Bz\Big[\frac{\Delta_{\xi,\psi}J}{J}\Big]_\xi \Pi_{\xi}, G\By\\\nonumber
&+\e\Bz\frac{\Delta_{\xi,\psi}J}{J} \Pi_{\psi}, G_\psi\By+\e\Bz\Big[\frac{\Delta_{\xi,\psi}J}{J}\Big]_\psi \Pi_{\psi}, G\By\\\nonumber
=&\MO\Big(\frac{\e}{Q} Q^2\big\|\nabla G\big\|^2\Big).
\end{align}
To estimate $I_{5,k}$, $k=1,2,3$, one can get
\begin{align}\nonumber
I_{5,1}=&-\e\Bz\Pi_{\xi\xi\xi\xi},G\By\\\nonumber
=&-\e\Bz\UB,G^2_{\xi\xi}\By+\e\Bz2\UB_{\xi\xi},G_\xi^2\By+\e\Bz\UB_{\xi\xi\xi},GG_{\xi}\By\\\nonumber
=&-\e\Bz\UB,G^2_{\xi\xi}\By+\MO\Big(\frac{\e}{Q} Q^2\big\|\nabla G\big\|^2\Big),
\end{align}
with $-\e\big\z\UB,G^2_{\xi\xi}\big\y$ non-nogative,
\begin{align}\nonumber
I_{5,2}=&-2\e\Bz\Pi_{\xi\xi\psi\psi},G\By\\\nonumber
=&-2\e\Bz\UB,G^2_{\xi\psi}\By+\e\Bz\UB_{\psi\psi},G^2_{\xi}\By+\e\Bz\UB_{\xi\xi},G^2_{\psi}\By\\\nonumber
&+2\e\Bz\UB_{\xi\psi},G_\xi G_{\psi}\By+2\e\Bz\UB_{\xi\xi\psi},GG_{\psi}\By\\\nonumber
=&-2\e\Bz\UB,G^2_{\xi\psi}\By+\e\Bz\UB_{\psi\psi},G^2_{\xi}\By+\MO\Big(\eq Q^2\big\|\nabla G\big\|^2\Big).
\end{align}
and 
\begin{align}\nonumber
I_{5,3}=&-\e\Bz\Pi_{\psi\psi\psi\psi},G\By\\\nonumber
=&-\e\Bz\UB,G^2_{\psi\psi}\By+\e\Bz \UB_{\psi},G^2_\psi \By\Big|_{\psi=1}-\e\Bz \UB_{\psi},G^2_\psi \By\Big|_{\psi=-1}\\\nonumber
&+\e\Bz2\UB_{\psi\psi},G^2_{\psi}\By-\e\Bz\frac{1}{2}\UB_{\psi\psi\psi\psi},G^2\By\\\nonumber
\geqslant &-\e\Bz\UB,G^2_{\psi\psi}\By+\e\Bz2\UB_{\psi\psi},G^2_{\psi}\By-\e\Bz\frac{1}{2}\UB_{\psi\psi\psi\psi},G^2\By,
\end{align}
where one has used the fact that $\pm\UB_{\psi}|_{\psi=\pm1}>0$. Thus $I_5$ admits the following estimate:
\begin{align}\label{I5}
\begin{aligned}
I_5\geqslant &-\e\Bz\UB,G^2_{\xi\xi}+2G^2_{\xi\psi}+G^2_{\psi\psi}\By+\e\Bz\UB_{\psi\psi},G^2_{\xi}\By+\e\Bz2\UB_{\psi\psi},G^2_{\psi}\By\\
&-\e\Bz\frac{1}{2}\UB_{\psi\psi\psi\psi},G^2\By-C\eq Q^2\big\|\nabla G\big\|^2.
\end{aligned}
\end{align}
It follows from (\ref{I1}) - (\ref{I5}) that
\begin{align}\label{Xin1}
\begin{aligned}
&I_1+I_2+I_3+I_4+I_5\\
\geqslant &-\e\Bz\UB,G^2_{\xi\xi}+2G^2_{\xi\psi}+G^2_{\psi\psi}\By\\
&+\Bz\frac{1}{2}\UB\VB_\psi-\frac{1}{2}\VB\UB_\psi-\frac{J_\xi}{J}\UB^2+\e\UB_{\psi\psi}, G_\xi^2\By\\
&+\Bz\frac{3}{2}\UB\VB_\psi-\frac{3}{2}\VB\UB_\psi-\frac{J_\xi}{J}\UB^2+2\e\UB_{\psi\psi}, G_\psi^2\By\\
&+\frac{1}{2}\Bz \UB\UB_{\xi\psi\psi}+\VB\UB_{\psi\psi\psi}-\UB_\psi\VB_{\psi\psi}-\UB_\xi\UB_{\psi\psi}+\frac{J_\xi}{J}\big(\UB_\psi^2+\UB\UB_{\psi\psi}\big)-\e\UB_{\psi\psi\psi\psi},G^2\By\\
&-C\eq Q^2\big\|\nabla G\big\|^2.
\end{aligned}
\end{align}
We now analyze the coefficient of $G^2$ in above inequality. In the region $-1\leqslant\psi\leqslant0$, for $\eta=\qe(\psi+1)$, one can get
\begin{align}\nonumber
&(\psi+1)^2\Big(\UB\UB_{\xi\psi\psi}+\VB\UB_{\psi\psi\psi}-\UB_\psi\VB_{\psi\psi}-\UB_\xi\UB_{\psi\psi}+\frac{J_\xi}{J}\big(\UB_\psi^2+\UB\UB_{\psi\psi}\big)-\e\UB_{\psi\psi\psi\psi}\Big)\\\nonumber
=&\eta^2\Big(\UB\UB_{\xi\eta\eta}+\qe\VB\UB_{\eta\eta\eta}-\qe\UB_\eta\VB_{\eta\eta}-\UB_\xi\UB_{\eta\eta}+\frac{J_\xi}{J}\big(\UB_\eta^2+\UB\UB_{\eta\eta}\big)-Q\UB_{\eta\eta\eta\eta}\Big)\\\nonumber
=&Q^2\eta^2\Big(u_pu_{\xi\eta\eta}+v_pu_{p\eta\eta\eta}-u_{p\eta}v_{p\eta\eta}-u_{p\xi}u_{p\eta\eta}+\frac{J_\xi}{J}\big(u_{p\eta}^2+u_pu_{p\eta\eta}\big)-u_{p\eta\eta\eta\eta}\Big)\\\nonumber
&+\MO(\eq Q^2)\\\nonumber
=&Q^2\eta^2\Big(u_pu_{\xi\eta\eta}+v_pu_{p\eta\eta\eta}-u_{p\eta}v_{p\eta\eta}-u_{p\xi}u_{p\eta\eta}+\frac{J_\xi}{J}\Big|_{\psi=-1}\big(u_{p\eta}^2+u_pu_{p\eta\eta}\big)-u_{p\eta\eta\eta\eta}\Big)\\\nonumber
&+Q^2\eta^2\Big(\frac{J_\xi}{J}-\frac{J_\xi}{J}\Big|_{\psi=-1}\Big)\big(u_{p\eta}^2+u_pu_{p\eta\eta}\big)+\MO(\eq Q^2)\\\nonumber
=&Q^2\eta^2\Big(u_pu_{\xi\eta\eta}+v_pu_{p\eta\eta\eta}+u_{p\eta}u_{p\xi\eta}+v_{p\eta}u_{p\eta\eta}+\frac{J_\xi}{J}\Big|_{\psi=-1}\big(u_{p\eta}^2+u_pu_{p\eta\eta}\big)-u_{p\eta\eta\eta\eta}\Big)\\\nonumber
&+\MO\Big(Q^2\eta^2(\psi+1)\big|u_{p\eta}^2+u_pu_{p\eta\eta}\big|\Big)+\MO\Big(\eq Q^2\Big)\\\nonumber
=&Q^2\eta^2\Big(u_pu_{\xi}+v_pu_{p\eta}+\frac{1}{2}\frac{J_\xi}{J}\Big|_{\psi=-1}u^2_p-u_{p\eta\eta}\Big)_{\eta\eta}+\MO\Big(\eq Q^2\eta^3\big|u_{p\eta}^2+u_pu_{p\eta\eta}\big|\Big)\\\nonumber
&+\MO\Big(\eq Q^2\Big)\\\nonumber
=&\MO\Big(\eq Q^2\Big),
\end{align}
since $[u_p,v_p]$ solves to the Prandtl problem (\ref{PE}). In the whole region $-1\leqslant\psi\leqslant 1$, one has
\begin{align}\nonumber
&(\psi+1)^2(1-\psi)^2\Big(\UB\UB_{\xi\psi\psi}+\VB\UB_{\psi\psi\psi}-\UB_\psi\VB_{\psi\psi}-\UB_\xi\UB_{\psi\psi}+\frac{J_\xi}{J}\big(\UB_\psi^2+\UB\UB_{\psi\psi}\big)-\e\UB_{\psi\psi\psi\psi}\Big)\\\nonumber
=&\MO\Big(\eq Q^2\Big),
\end{align}
It thus holds that
\begin{align}\label{Xin2}
\begin{aligned}
&\frac{1}{2}\Bz \UB\UB_{\xi\psi\psi}+\VB\UB_{\psi\psi\psi}-\UB_\psi\VB_{\psi\psi}-\UB_\xi\UB_{\psi\psi}+\frac{J_\xi}{J}\big(\UB_\psi^2+\UB\UB_{\psi\psi}\big)-\e\UB_{\psi\psi\psi\psi},G^2\By\\
=&\MO\Big(\eq Q^2\Big\|\frac{G}{(\psi+1)(1-\psi)}\Big\|^2\Big)=\MO\Big(\eq Q^2\big\|G_\psi\big\|^2\Big).
\end{aligned}
\end{align}
Next one can show that the coefficients of $G^2_\xi$ and $G^2_\psi$ in (\ref{Xin1}) are strictly positive. Indeed, in the region $-1\leqslant\psi\leqslant0$, 
\begin{align}\nonumber
&\frac{1}{2}\UB\VB_\psi-\frac{1}{2}\VB\UB_\psi-\frac{J_\xi}{J}\UB^2+\e\UB_{\psi\psi}\\\nonumber
=&\frac{1}{2}\qe\UB\VB_\eta-\frac{1}{2}\qe\VB\UB_\eta-\frac{J_\xi}{J}\UB^2+Q\UB_{\eta\eta}\\\nonumber
=&Q^2\Big(\frac{1}{2}u_pv_{p\eta}-\frac{1}{2}v_pu_{p\eta}-\frac{J_\xi}{J}u_p^2+u_{p\eta\eta}\Big)+\MO\Big(\eq Q^2\Big)\\\nonumber
=&Q^2\Big(-\frac{1}{2}u_pu_{p\xi}-\frac{1}{2}v_pu_{p\eta}-\frac{J_\xi}{J}u_p^2+u_{p\eta\eta}\Big)+\MO\Big(\eq Q^2\Big).
\end{align}
Due to the Prandtl equation, it holds that  
$$-u_pu_{p\xi}-v_pu_{p\eta}=\frac{1}{2}\frac{J_\xi(\xi,-1)}{J(\xi,-1)}u^2_p-u_{p\eta\eta}-\frac{1}{8}\frac{J_\xi(\xi,-1)}{J(\xi,-1)}.$$
Thus, it follows that for $-1\leqslant\psi\leqslant0$,
\begin{align}\nonumber
&\frac{1}{2}\UB\VB_\psi-\frac{1}{2}\VB\UB_\psi-\frac{J_\xi}{J}\UB^2+\e\UB_{\psi\psi}\\\nonumber
=&Q^2\Big(\frac{1}{4}\frac{J_\xi}{J}\Big|_{\psi=-1}u^2_p-\frac{1}{16}\frac{J_\xi}{J}\Big|_{\psi=-1}-\frac{J_\xi}{J}u_p^2+\frac{1}{2}u_{p\eta\eta}\Big)+\MO\Big(\eq Q^2\Big).
\end{align}
For $\eta=\qe(\psi+1)\leqslant \eta_0$, since $-\frac{J_\xi}{J}\geqslant 2\mu>0$ and $u_{p\eta\eta}\geqslant0$, one can obtain
\begin{align}\nonumber
&\frac{1}{4}\frac{J_\xi}{J}\Big|_{\psi=-1}u^2_p-\frac{1}{16}\frac{J_\xi}{J}\Big|_{\psi=-1}-\frac{J_\xi}{J}u_p^2+\frac{1}{2}u_{p\eta\eta}\\\nonumber
\geqslant&\frac{1}{4}\frac{J_\xi}{J}\Big|_{\psi=-1}u^2_p-\frac{1}{16}\frac{J_\xi}{J}\Big|_{\psi=-1}-\frac{1}{4}\frac{J_\xi}{J}u_p^2\\\nonumber
\geqslant&-C(\psi+1)+\frac{\mu}{8}\\\nonumber
\geqslant&-C\eq+\frac{\mu}{8},
\end{align}
while for $\eta\geqslant \eta_0$, since $-\frac{1}{2}\leqslant u_p(\xi,\eta)\leqslant-m_0\eta_0<0$, it holds that
\begin{align}\nonumber
&\frac{1}{4}\frac{J_\xi}{J}\Big|_{\psi=-1}u^2_p-\frac{1}{16}\frac{J_\xi}{J}\Big|_{\psi=-1}-\frac{J_\xi}{J}u_p^2+\frac{1}{2}u_{p\eta\eta}\\\nonumber
\geqslant&\frac{1}{4}\frac{J_\xi}{J}\Big|_{\psi=-1}u^2_p-\frac{1}{16}\frac{J_\xi}{J}\Big|_{\psi=-1}-\frac{J_\xi}{J}u_p^2\\\nonumber
\geqslant&2\mu u_p^2(\xi,\eta_0)\\\nonumber
\geqslant&2\mu m_0^2\eta_0^2.
\end{align}
In the whole region $-1\leqslant \psi\leqslant1$, one can obtain
\begin{align}\nonumber
&\frac{1}{2}\UB\VB_\psi-\frac{1}{2}\VB\UB_\psi-\frac{J_\xi}{J}\UB^2+\e\UB_{\psi\psi}\geqslant Q^2c^*-C\eq Q^2,
\end{align}
where $c^*=\min\big\{\frac{\mu}{8},2\mu m^2_0\eta^2_0\big\}$. Thus, the following estimate holds
\begin{align}\label{positivity1}
&\Bz\frac{1}{2}\UB\VB_\psi-\frac{1}{2}\VB\UB_\psi-\frac{J_\xi}{J}\UB^2+\e\UB_{\psi\psi}, G_\xi^2\By\geqslant Q^2c^*\big\|G_\xi\big\|^2-C\eq Q^2\|\nabla G\big\|^2.
\end{align}
Similarly, in the region $-1\leqslant\psi\leqslant0$, one has
\begin{align}\nonumber
&\frac{3}{2}\UB\VB_\psi-\frac{3}{2}\VB\UB_\psi-\frac{J_\xi}{J}\UB^2+2\e\UB_{\psi\psi}\\\nonumber
=&Q^2\Big(-\frac{3}{2}u_pu_{p\xi}-\frac{3}{2}v_pu_{p\eta}-\frac{J_\xi}{J}u_p^2+2u_{p\eta\eta}\Big)+\MO\Big(\eq Q^2\Big)\\\nonumber
=&Q^2\Big(\frac{3}{4}\frac{J_\xi}{J}\Big|_{\psi=-1}u^2_p-\frac{3}{16}\frac{J_\xi}{J}\Big|_{\psi=-1}-\frac{J_\xi}{J}u_p^2+\frac{1}{2}u_{p\eta\eta}\Big)+\MO\Big(\eq Q^2\Big).
\end{align}
For $\qe(\psi+1)\leqslant \eta_0$, it holds that
\begin{align}\nonumber
&\frac{3}{4}\frac{J_\xi}{J}\Big|_{\psi=-1}u^2_p-\frac{3}{16}\frac{J_\xi}{J}\Big|_{\psi=-1}-\frac{J_\xi}{J}u_p^2+\frac{1}{2}u_{p\eta\eta}\\\nonumber
\geqslant&\frac{3}{4}\frac{J_\xi}{J}\Big|_{\psi=-1}u^2_p-\frac{3}{16}\frac{J_\xi}{J}\Big|_{\psi=-1}-\frac{3}{4}\frac{J_\xi}{J}u_p^2\\\nonumber
\geqslant&-C\eq+\frac{3\mu}{8}\\\nonumber
\geqslant& -C\eq+c^*,
\end{align}
while for $\qe(\psi+1)\geqslant \eta_0$, one has
\begin{align}\nonumber
&\frac{3}{4}\frac{J_\xi}{J}\Big|_{\psi=-1}u^2_p-\frac{3}{16}\frac{J_\xi}{J}\Big|_{\psi=-1}-\frac{J_\xi}{J}u_p^2+\frac{1}{2}u_{p\eta\eta}\\\nonumber
\geqslant&\frac{3}{4}\frac{J_\xi}{J}\Big|_{\psi=-1}u^2_p-\frac{3}{16}\frac{J_\xi}{J}\Big|_{\psi=-1}-\frac{J_\xi}{J}u_p^2\\\nonumber
\geqslant&2\mu u_p^2(\xi,\eta_0)\\\nonumber
\geqslant&c^*.
\end{align}
Consequently,
\begin{align}\label{positivity2}
&\Bz\frac{3}{2}\UB\VB_\psi-\frac{3}{2}\VB\UB_\psi-\frac{J_\xi}{J}\UB^2+2\e\UB_{\psi\psi}, G_\psi^2\By\geqslant Q^2c^*\big\|G_\psi\big\|^2-C\eq Q^2\|\nabla G\big\|^2.
\end{align}
Finally, we conclude from (\ref{Xin1})-(\ref{positivity2}) that
\begin{align}\nonumber
&I_1+I_2+I_3+I_4+I_5\\\nonumber
\geqslant &-\e\Bz\UB,G^2_{\xi\xi}+2G^2_{\xi\psi}+G^2_{\psi\psi}\By+Q^2c^*\big\|G_\xi\big\|^2+Q^2c^*\big\|G_\psi\big\|^2\\\nonumber
&-C\eq Q^2\big\|\nabla G\big\|^2.
\end{align}
Since $-\UB$ is non-negative and $\eq$ is small enough, the following estimate holds for $G$:
\begin{align}\nonumber
&-\e\Bz\UB,G^2_{\xi\xi}+2G^2_{\xi\psi}+G^2_{\psi\psi}\By+Q^2\big\|G_\xi\big\|^2+Q^2\big\|G_\psi\big\|^2\leqslant \frac{C}{Q^2}\big\|F\big\|^2.
\end{align}
Since $\Pi=\UB G$, the estimate (\ref{KeyPie}) follows from the Hardy inequality. 
\end{proof}
\begin{remark}
In term of the linearized equation of (\ref{NPie}), the multiplier in the proof above is $-\frac{1}{J^2}\frac{\Pi}{\UB}$. The weight $\frac{1}{J^2}$ helps to gain the uniform positivity (see (\ref{positivity1}) and (\ref{positivity2})), and cancel both the terms caused by the non-shear Euler flow ( $\Delta_{\xi,\psi}\big[\frac{J_\xi}{J}\big]=0$ in the estimate of $I_1$) and the terms due to the normal derivatives of the Prandtl solution near the boundaries (see (\ref{Xin2})).
\end{remark}
{\bf Proof of Theorem \ref{main}:}

To this end, we use the method of the contraction mapping again. Define 
 \begin{align}\label{PiUpe}
\begin{aligned}
\big\|\Pi\big\|_{\Y}=&\big\|\nabla_{\xi,\psi}\Pi\big\|+\eq\big\|\nabla_{\xi,\psi}^2\Pi\big\|+\Big(\frac{\e}{Q}\Big)^\frac{5}{4}\big\|\nabla^3_{\xi,\psi}\Pi\big\|\\
&+\Big(\frac{\e}{Q}\Big)^2\big\|\nabla_{\xi,\psi}^4\Pi\big\|.
\end{aligned}
\end{align}
Define $\TT:H^4\big(\mathbb{R}\times(-1,1)\big)\rightarrow H^4\big(\mathbb{R}\times(-1,1)\big)$ by $\TT(\Pi)=\Psi$, where $\Psi$ solves the following problem:
\begin{equation}\label{Pieup}
\left\{
\begin{aligned}
&J\bar{\Phi}_{s\psi}\p_\xi\big(J\Delta_{\xi,\psi}\Psi\big)+J\Psi_{\psi}\p_\xi\big(J\Delta_{\xi,\psi}\bar{\Phi}_s\big)-J\bar{\Phi}_{s\xi}\p_\psi\big(J\Delta_{\xi,\psi}\Psi\big)-J\Psi_\xi\p_\psi\big(J\Delta_{\xi,\psi}\bar{\Phi}_s\big)\\
&+\e J\Delta_{\xi,\psi}\big(J\Delta_{\xi,\psi}\Psi\big)=-J^2\bar{R}_s-J^2N(\Pi), \\
&\Psi|_{\psi=1}=\Psi|_{\psi=-1}=0,\\
&\Psi_\psi|_{\psi=1}=\Psi_\psi|_{\psi=-1}=0,
\end{aligned}
\right.
\end{equation}
where 
$$N(\Pi)=\frac{1}{J}\Pi_\psi\p_\xi\big(J\Delta_{\xi,\psi}\Pi\big)-\frac{1}{J}\Pi_\xi\p_\psi\big(J\Delta_{\xi,\psi}\Pi\big).$$ 
Let 
$$B=\bigg\{\Pi\in H^4\big(\mathbb{R}\times(-1,1)\big):\big\|\Pi\big\|_{\Y}\leqslant \Big(\eq\Big)^{n-1}Q\bigg\}.$$
We will show that $\TT$ is a contractive mapping in $B$ if $\big\|\bar{R}_s\big\|\leqslant C\Big(\eq\Big)^{n}Q^2.$

Set $F=-\bar{R}_s-N(\Pi)$. It follows from Proposition \ref{Keyp} that
$$\big\|\nabla_{\xi,\psi}\Psi\big\|+\eq\big\|\nabla^2_{\xi,\psi}\Psi\big\|\leqslant \frac{C}{Q}\big\|F\big\|.$$
Rewrite (\ref{Pieup}) as
\begin{equation}\nonumber
\left\{
\begin{aligned}
&\Delta_{\xi,\psi}^2 \Psi=\tilde{F}, \\
&\Psi|_{\psi=1}=\Psi|_{\psi=-1}=0,\\
&\Psi_\psi|_{\psi=1}=\Psi_\psi|_{\psi=-1}=0,
\end{aligned}
\right.
\end{equation}
where
\begin{align}\nonumber
\tilde{F}=&\frac{F}{\e}-\frac{1}{\e J}\bar{\Phi}_{s\psi}\p_\xi\big(J\Delta_{\xi,\psi}\Psi\big)-\frac{1}{\e J}\Psi_{\psi}\p_\xi\big(J\Delta_{\xi,\psi}\bar{\Phi}_s\big)+\frac{1}{\e J}\bar{\Phi}_{s\xi}\p_\psi\big(J\Delta_{\xi,\psi}\Psi\big)\\\nonumber
&+\frac{1}{\e J}\Psi_\xi\p_\psi\big(J\Delta_{\xi,\psi}\bar{\Phi}_s\big)- \frac{2}{J}\nabla_{\xi,\psi}J\cdot\nabla_{\xi,\psi}\Delta_{\xi,\psi}\Psi-\frac{1}{J}\Delta_{\xi,\psi}J\Delta_{\xi,\psi}\Psi.
\end{align}
By the standard theory for fourth-order elliptic equations, one has
$$\big\|\Psi\big\|_{H^4}\leqslant C\big\|\tilde{F}\big\|.$$
Thus, it holds that
\begin{align}\nonumber
\big\|\nabla^4_{\xi,\psi}\Psi\big\|\leqslant&C\Big \|\frac{F}{\e}\Big\|+C\Big\|\frac{1}{\e J}\bar{\Phi}_{s\psi}\p_\xi\big(J\Delta_{\xi,\psi}\Psi\big)\Big\|+C\Big\|\frac{1}{\e J}\Psi_{\psi}\p_\xi\big(J\Delta_{\xi,\psi}\bar{\Phi}_s\big)\Big\|\\\nonumber
&+C\Big\|\frac{1}{\e J}\bar{\Phi}_{s\xi}\p_\psi\big(J\Delta_{\xi,\psi}\Psi\big)\Big\|+C\Big\|\frac{1}{\e J}\Psi_\xi\p_\psi\big(J\Delta_{\xi,\psi}\bar{\Phi}_s\big)\Big\|\\\nonumber
&+C\Big\| \frac{1}{J}\nabla_{\xi,\psi}J\cdot\nabla_{\xi,\psi}\Delta_{\xi,\psi}\Psi\Big\|+C\Big\|\frac{1}{J}\Delta_{\xi,\psi}J\Delta_{\xi,\psi}\Psi\Big\|\\\nonumber
\leqslant&\frac{C}{\e}\big \|F\big\|+C\Big(\frac{\e}{Q}\Big)^{-1}\big\|\nabla^3_{\xi,\psi}\Psi\big\|+C\Big(\frac{\e}{Q}\Big)^{-1}\big\|\nabla^2_{\xi,\psi}\Psi\big\|+C\Big(\frac{\e}{Q}\Big)^{-2}\big\|\nabla_{\xi,\psi}\Psi\big\|\\\nonumber
\leqslant&\frac{C}{Q}\Big(\frac{\e}{Q}\Big)^{-1}\big \|F\big\|+C\Big(\frac{\e}{Q}\Big)^{-1}\big\|\nabla^3_{\xi,\psi}\Psi\big\|+\frac{C}{Q}\Big(\frac{\e}{Q}\Big)^{-\frac{3}{2}}\big\|F\big\|+\frac{C}{Q}\Big(\frac{\e}{Q}\Big)^{-2}\big\|F\big\|\\\nonumber
\leqslant &\frac{C}{Q}\Big(\frac{\e}{Q}\Big)^{-2}\big\|F\big\|+C\Big(\frac{\e}{Q}\Big)^{-1}\big\|\nabla^3_{\xi,\psi}\Psi\big\|.
\end{align}
The Gagliardo-Nirenberg inequality implies
 \begin{align}\nonumber
\begin{aligned}
\big\|\nabla^3_{\xi,\psi}\Psi\big\|\leqslant C\big\|\nabla^2_{\xi,\psi}\Psi\big\|+C\big\|\nabla^4_{\xi,\psi}\Psi\big\|^\frac{1}{2}\big\|\nabla^2_{\xi,\psi}\Psi\big\|^\frac{1}{2}.
\end{aligned}
\end{align}
Hence one gets 
\begin{align}\nonumber
\big\|\nabla^4_{\xi,\psi}\Psi\big\|\leqslant&\frac{C}{Q}\Big(\frac{\e}{Q}\Big)^{-2}\big\|F\big\|+C\Big(\frac{\e}{Q}\Big)^{-1}\big\|\nabla^3_{\xi,\psi}\Psi\big\|+C\Big(\frac{\e}{Q}\Big)^{-1}\big\|\nabla^4_{\xi,\psi}\Psi\big\|^\frac{1}{2}\big\|\nabla^2_{\xi,\psi}\Psi\big\|^\frac{1}{2}\\\nonumber
\leqslant&\frac{C}{Q}\Big(\frac{\e}{Q}\Big)^{-2}\big\|F\big\|+C\Big(\frac{\e}{Q}\Big)^{-1}\big\|\nabla^3_{\xi,\psi}\Psi\big\|+C\Big(\frac{\e}{Q}\Big)^{-2}\big\|\nabla^2_{\xi,\psi}\Psi\big\|\\\nonumber
\leqslant&\frac{C}{Q}\Big(\frac{\e}{Q}\Big)^{-2}\big\|F\big\|.
\end{align}
Applying the Gagliardo-Nirenberg inequality again yields
 \begin{align}\nonumber
\begin{aligned}
\big\|\nabla^3_{\xi,\psi}\Psi\big\|\leqslant &C\big\|\nabla^2_{\xi,\psi}\Psi\big\|+C\big\|\nabla^4_{\xi,\psi}\Psi\big\|^\frac{1}{2}\big\|\nabla^2_{\xi,\psi}\Psi\big\|^\frac{1}{2}\leqslant\frac{C}{Q}\Big(\frac{\e}{Q}\Big)^{-\frac{5}{4}}\big\|F\big\|.
\end{aligned}
\end{align}
Therefore
$$\big\|\Psi\big\|_{\Y}\leqslant\frac{C}{Q}\big\|F\big\|.$$
It is easy to check that
 \begin{align}\nonumber
\begin{aligned}
\big\|N(\Pi)\big\|\leqslant& C\big\|\nabla_{\xi,\psi}\Pi\big\|_{\infty}\big\|\nabla^3_{\xi,\psi}\Pi\big\|\\
 \leqslant &C\Big(\big\|\nabla_{\xi,\psi}\Pi\big\|+\big\|\nabla_{\xi,\psi}\Pi\big\|^\frac{1}{4}\big\|\nabla_{\xi,\psi}\Pi\big\|^\frac{1}{2}\big\|\nabla^3_{\xi,\psi}\Pi\big\|^\frac{1}{4}\Big)\big\|\nabla^3_{\xi,\psi}\Pi\big\| \\
\leqslant& C\Big(\frac{\e}{Q}\Big)^{-\frac{5}{2}}\big\|\Pi\big\|_{\Y}^2.
\end{aligned}
\end{align}
For $n=12$, if $\Pi\in B$, then it holds that
 \begin{align}\nonumber
\begin{aligned}
\big\|\Psi\big\|_{\Y}\leqslant &\frac{C}{Q}\big\|\bar{R}_s\big\|+\frac{C}{Q}\big\||N(\Pi)\big\|\\
\leqslant &\frac{C}{Q}\big\|\bar{R}_s\big\|+\frac{C}{Q}\Big(\frac{\e}{Q}\Big)^{-\frac{5}{2}}\big\|\Xi\big\|_{\X}^2\\
\leqslant&C\Big(\eq\Big)^{n}Q+C\Big(\eq\Big)^{2n-12}Q\\
\leqslant&\Big(\eq\Big)^{n-1}Q,
\end{aligned}
\end{align}
provided that $\eq$ is small enough. Thus, $\TT(B)\subset B$. If $\Pi_1$, $\Pi_2\in B$, then
\begin{align}\nonumber
\begin{aligned}
\big\|\TT(\Pi_1-\Pi_2)\big\|_{\Y}\leqslant &C\big\|N(\Pi_1)-N(\Pi_2)\big\|\\
\leqslant&C\big\|\nabla_{\xi,\psi}\Pi_1\big\|_{\infty}\big\|\nabla^2_{\xi,\psi}\big(\Pi_1-\Pi_2\big)\big\|\\
 &+C\big\|\nabla_{\xi,\psi}\big(\Pi_1-\Pi_2\big)\big\|_{\infty}\big\|\nabla^2_{\xi,\psi}\Pi_2\big\|\\
\leqslant& \frac{C}{Q}\Big(\frac{\e}{Q}\Big)^{-\frac{5}{2}}\Big(\big\|\Pi_1\big\|_{\X}+\big\|\Pi_2\big\|_{\X}\Big)\big\|\Pi_{1}-\Pi_{2}\big\|_{\X}\\
\leqslant& C\Big(\eq\Big)^{n-11}\big\|\Pi_{1}-\Pi_{2}\big\|_{\Y},
\end{aligned}
\end{align}
hence $\TT$ is contractive on $B$ for $\eq$ small enough. Thus, there exists a $\Pi\in B$ such that $\TT(\Pi)=\Pi$. The problem (\ref{NPie}) admits a solution $\Pi\in B$ satisfying 
\begin{align}\nonumber
\begin{aligned}
&\big\|\nabla_{\xi,\psi}\Pi\big\|_{\infty}\leqslant C\Big(\frac{\e}{Q}\Big)^{-\frac{5}{4}}\big\|\Pi\big\|_{\Y}\leqslant C\Big(\eq\Big)^{n-1-\frac{5}{2}}Q\leqslant C\eq Q.\\
\end{aligned}
\end{align}
Thus, $\Phi=\bar{\Phi}_s+\Pi$ is a solution to the Navier-Stokes problem (\ref{NSP}), which satisfies
\begin{align}\nonumber
\begin{aligned}
&\big\|-\Phi_\psi-\UB\big\|_\infty\leqslant C\eq Q,\\
&\big\|\Phi_\xi-\VB\big\|_\infty\leqslant C\eq Q.
\end{aligned}
\end{align}
Moreover, since $\Pi\in H^4\big(\mathbb{R}\times(-1,1)\big)$, one has the asymptotic behaviors
\begin{align}\nonumber
\begin{aligned}
\underset{\xi\rightarrow\pm\infty}{\lim}\big[-\Pi_\psi,\Pi_\xi\big]=0,\text{ for $\psi\in[-1,1]$ uniformlly.}
\end{aligned}
\end{align}
Thus, we have shown that
\begin{align}\nonumber
\begin{aligned}
\underset{\xi\rightarrow\pm\infty}{\lim}\big[-\Psi_\psi,\Phi_\xi\big]=\underset{\xi\rightarrow\pm\infty}{\lim}\big[-\bar{\Phi}_{s\psi},\bar{\Phi}_{s\xi}\big]=[A^{\pm},0],\text{ for $\psi\in[-1,1]$ uniformlly.}
\end{aligned}
\end{align}
So the proof of Theorem \ref{main} is completed.
\qed

\section{Appendix}
{\bf Appendix A: A conformal isomorphism from $\Omega$ to $(-\infty,\infty)\times(-1,1)$}

In section \ref{Euler-section}, we have shown that $\psi$ satifies the following estimates:
\begin{align}\label{NEFE-in-Appendix}
\begin{aligned}
&\Big|r^k\p_r^k\p_\te^l\big[\psi(r,\te)-\frac{\te}{\beta}\big]\Big|\leqslant\frac{C_{k,l}(\lambda)}{r^\frac{1}{4}},\text{   for }r\geqslant1,\\
&\Big|r^k\p_r^k\p_\te^l\big[\psi(r,\te)-\frac{\te}{\al}\big]\Big|\leqslant C_{k,l}(\lambda)r^\frac{1}{4},\text{   for }r\leqslant1,
\end{aligned}
\end{align}
and 
\begin{align}\label{lob-in-Appendix}
\underset{\Omega}{\inf}\hspace{1mm}\psi_\te\geqslant b>0.
\end{align}
Define 
$$\xi=\frac{1}{\beta}\ln{r}-\int_{r}^\infty\frac{1}{\rho}\big(\psi_\te(\rho,\te)-\frac{1}{\beta}\big)\dd\rho.$$ 
Since $\psi$ is a harmonic function, one has that 
$$\p_r\xi=\frac{\p_\te\psi}{r}, \quad \frac{\p_\te\xi}{r}=-\p_r\psi.
$$ 
It follows from (\ref{NEFE-in-Appendix}) that
\begin{align}\label{xib-in-Appendix}
\xi\sim
\left\{
\begin{aligned}
&\frac{1}{\beta}\ln{r}, \quad r\rightarrow \infty,\\
&\frac{1}{\al}\ln{r}, \quad r\rightarrow0.
\end{aligned}
\right.
\end{align}
\begin{lemma}If $\Gamma^\pm$ are assume to be smooth curves satisfying (\ref{CUR}) and (\ref{CURC}). Then the mapping $(\xi,\psi):\Omega\rightarrow (-\infty,\infty)\times(-1,1)$ is a conformal isomorphism.
\end{lemma}
\begin{proof}
It suffices to show that $(\xi,\psi)$ is a bijection. Let $\bar{\mathcal{C}}_1=\big\{(r,\te)\big|r=1,|\te|\leqslant g(1)\big\}$. Since $\psi|_{r=1,\te=\pm g(1)}=\pm 1$, $\psi_\te\geqslant b>0$, and $\psi:\bar{\mathcal{C}}_1\rightarrow[-1,1]$ is a bijection. For any $c\in[-1,1]$, there exists a unique $\te_c\in\big[-g(1),g(1)\big]$, such that $\psi|_{r=1,\te=\te_c}=c$. Let $(x,y)=(r\cos\te,r\sin\te)$ be the Cartesian coordinate and $(x_c(t),y_c(t))$ be the characteristic line defined by
\begin{equation}\label{ODEcl}
\left\{
\begin{aligned}
&\frac{\dd x_c}{\dd t}=\xi_x(x_c(t),y_c(t)),\\
&\frac{\dd y_c}{\dd t}=\xi_y(x_c(t),y_c(t)),\\
&\big(x_c(0),y_c(0)\big)=\big(\cos(\te_c),\sin(\te_c)\big).
\end{aligned}
\right.
\end{equation}
Since $\xi_x=\psi_y$, $\xi_y=-\psi_x$, $\xi_x^2+\xi_y^2>0$ and $\psi|_{\Gamma^\pm}=\pm1$, it is easy to check that $\Gamma^\pm$ is the integral curve of $c=\pm1$. Set $r_c(t):=\sqrt{x_c(t)^2+y_c(t)^2}$.  Since $(\xi_x,\xi_y)\cdot (\frac{x}{r},\frac{y}{r})=\xi_r=\frac{\psi_\te}{r}$, one has
$$\frac{C}{r_c}\geqslant\frac{\dd r_c}{\dd t}=\frac{\xi_x(x_c,y_c)x_c+\xi_y(x_c,y_c)y_c}{r_c}\geqslant \frac{b}{r_c}.$$
Then $2bt+1\leqslant r_c^2(t)\leqslant 2Ct+1$ for $t\geqslant0$. For any $c\in(-1,1)$, as $t$ increases,  $r_c(t)$ increases and $\big(x_c(t),y_c(t)\big)$ does not touch $\Gamma^\pm$, so the solution to (\ref{ODEcl}) exists for $t\in[0,\infty)$. Assume that the maximal existence interval of (\ref{ODEcl}) to be $(t_c,\infty)$. It can be shown that $t_c>-\infty$. Indeed, since
$$\frac{\dd r_c}{\dd t}=\frac{\xi_x(x_c,y_c)x_c+\xi_y(x_c,y_c)y_c}{r_c}\geqslant \frac{b}{r_c},$$
then $r_c^2(t)\leqslant 2bt+1$ for $t\leqslant 0$. It follows that $\big(x_c(t),y_c(t)\big)$ arrives at the origin at a finite $t<0$, thus, $t_c>-\infty$. Let
\begin{align}\nonumber
\Gamma_c=\big\{\big(x_c(t),y_c(t)\big)\big| t>t_c\big\}.
\end{align}
Since $\frac{\dd}{\dd t}\psi(x_c(t),y_c(t))=\psi_x\xi_x+\psi_y\xi_y=0$, so $\psi|_{\Gamma_c}=c$. And $\frac{\dd}{\dd t}\xi(x_c(t),y_c(t))=\xi_x^2+\xi_y^2>0$, thus $\xi(x_c(t),y_c(t))$ increases with respect to $t$. 

Now we show that $(\xi,\psi)$ is a surjection. In fact, $2bt+1\leqslant r_c^2(t)$ for $t\geqslant0$ implies that $r_c(t)$ goes to $\infty$ as $t\rightarrow\infty$. Since $\xi\sim\frac{1}{\beta}\ln r$ for $r$ large enough, it holds that $\underset{t\rightarrow\infty}{\lim}\xi(x_c(t),y_c(t))=\infty$. As $t\rightarrow t_c$, $r_c(t)$ convergences to $0$, so (\ref{xib-in-Appendix}) implies $\underset{t\rightarrow t_c}{\lim}\xi(x_c(t),y_c(t))=-\infty$. Thus, $(\xi,\psi)$ is a surjection.

If $(\xi(x,y),\psi(x,y))=(\xi(\bar{x},\bar{y}),\psi(\bar{x},\bar{y}))$, it follows from the above argument for $r_c(t)$ that the characteristic lines start from $(x,y)$ and $(\bar{x},\bar{y})$ will intersect with $\mathcal{C}_1$. Let $\bar{c}=\psi(x,y)=\psi(\bar{x},\bar{y})$. Then the two characteristic lines intersect with $\mathcal{C}_1$ at a same point $(\cos\bar{c},\sin\bar{c})$. It means that $(x,y)$ and $(\bar{x},\bar{y})$ belong to $\Gamma_{\bar{c}}$. Since $\xi(x,y)=\xi(\bar{x},\bar{y})$ and $\frac{\dd}{\dd t}\xi(x_{\bar{c}}(t),y_{\bar{c}}(t))>0$, hence $(x,y)=(\bar{x},\bar{y})=(x_{\bar{c}}(\tau),y_{\bar{c}}(\tau))$ for some $\tau$. Thus, $(\xi,\psi)$ is an injection.
\end{proof}

{\bf Appendix B:  The force term of the linearized Prandtl equation}

In this appendix, we give the precise form of the force term $F^m_b$ in the linearized Prandtl equations (\ref{brandtlm}). For simplicity, we introduce the following notations: 
\begin{align}\nonumber
&u_e^m=-\frac{1}{Q}\Phi_{e\psi}^m,\quad v_e^m=\frac{1}{Q}\Phi_{e\xi}^m,\\\nonumber
&u_e^{(m)}=\overset{m}{\underset{j=0}{\sum}}\frac{\eta^j}{j!}\p^j_\psi u_e^{m-j}|_{\psi=-1},\quad v_e^{(m)}=\overset{m+1}{\underset{j=0}{\sum}}\frac{\eta^j}{j!}\p^j_\psi v_e^{m+1-j}|_{\psi=-1},\\\nonumber
&u_p^m=u_b^m+u_e^{(m)},\quad v_p^m=v_b^m+\overset{m+1}{\underset{j=0}{\sum}}\frac{\eta^j}{j!}\p^j_\psi v_e^{m+1-j}|_{\psi=-1},\\\nonumber
&J_1^m=\eta^m\p^m_\psi\Big[\frac{J_\xi}{J}\Big]\bigg|_{\psi=-1},\quad J_2^m=\eta^m\p^m_\psi\Big[\frac{J_{\xi\xi}}{J}\Big]\bigg|_{\psi=-1},\\\nonumber
&J_3^m=\eta^m\p^m_\psi\Big[\frac{J_\psi}{J}\Big]\bigg|_{\psi=-1},\quad J_2^m=\eta^m\p^m_\psi\Big[\frac{J_{\psi\psi}}{J}\Big]\bigg|_{\psi=-1}.
\end{align}
Then the force term is 
$$F^m_b=-\int_\eta^\infty f^m_b(\xi,\eta')\dd \eta',$$
where 
\begin{align}\nonumber
f^m_b=&-\overset{m-1}{\underset{j=1}{\sum}}\Big(u_p^ju_{b\xi\eta}^{m-j}+u_b^j u_{e\xi\eta}^{(m-j)}\Big)-u_b^{0}u_{e\xi\eta}^{(m)}-u_e^{(m)}u_{b\xi\eta}^0\\\nonumber
&+\overset{m-2}{\underset{j=0}{\sum}}\Big(u_p^jv_{b\xi\xi}^{m-2-j}+u_b^jv_{e\xi\xi}^{(m-2-j)}\Big)-\underset{\underset{(j,l)\notin\{(0,m),(m,0)\}}{j+l+k=m}}{\sum}\Big(u_p^ju_{b\eta}^lJ_1^k+u_b^ju_{e\eta}^{(l)}J_1^k\Big)\\\nonumber
&-u_b^0u_{e\eta}^{(m)}J_1^0-u_e^{(m)}u_{b\eta}^0J_1^0+\underset{\underset{j,l,k\geqslant0}{j+l+k=m-2}}{\sum}\Big(u_p^jv_{b\xi}^lJ_1^k+u_b^jv_{e\xi}^{(l)}J_1^k\Big)\\\nonumber
&-\overset{m-1}{\underset{j=1}{\sum}}\Big(v_p^ju_{b\eta\eta}^{m-j}+v_b^j u_{e\eta\eta}^{(m-j)}\Big)-v_b^{(0)}u_{e\eta\eta}^{(m)}-u_{b\eta\eta}^0\overset{m+1}{\underset{j=1}{\sum}}\frac{\eta^j}{j!}\p^j_\psi v_e^{m+1-j}|_{\psi=-1}\\\nonumber
&-\overset{m-2}{\underset{j=0}{\sum}}\Big(v_p^ju_{b\xi\xi}^{m-2-j}+v_b^j u_{e\xi\xi}^{(m-2-j)}\Big)-\underset{\underset{j,l,k\geqslant0}{j+l+k=m-1}}{\sum}\Big(v_p^ju_{b\eta}^lJ_3^k+v_b^ju_{e\eta}^{(l)}J_3^k\Big)\\\nonumber
&+\overset{m-3}{\underset{j=0}{\sum}}\Big(v_p^jv_{b\xi}^{m-3-j}+v_b^j v_{e\xi}^{(m-3-j)}\Big)+2u_{b\xi\xi\eta}^{m-2}-v_{b\xi\xi\xi}^{m-4}\\\nonumber
&+2\overset{m-2}{\underset{j=0}{\sum}}J_1^ju_{b\xi\eta}^{m-2-j}-2\overset{m-4}{\underset{j=0}{\sum}}J_1^jv_{b\xi\xi}^{m-4-j}+2\overset{m-1}{\underset{j=0}{\sum}}J_3^ju_{b\eta\eta}^{m-1-j}+2\overset{m-3}{\underset{j=0}{\sum}}J_3^ju_{b\xi\xi}^{m-3-j}\\\nonumber
&+\overset{m-2}{\underset{j=0}{\sum}}\Big(J_2^j+J_4^j\Big)u_{b\eta}^{m-2-j}-\overset{m-4}{\underset{j=0}{\sum}}\Big(J_2^j+J_4^j\Big)v_{b\xi}^{m-4-j}.
\end{align}

{\bf Appendix C:  The uniqueness of non-positive solution to (\ref{INBE})}

In Section \ref{infty behavior}, we have constructed a solution to the following problem:
\begin{equation}\label{Ape-in-app}
\left\{
\begin{aligned}
&-2\beta A^+\frac{\dd A^+}{\dd \psi}-\e\Big(\frac{\dd^3 A^+}{\dd\psi^3}+4\beta^2\frac{\dd A^+}{\dd \psi}\Big)=0,\\
&A^+(-1)=A^+(1)=0, \\
&\int_{-1}^1 A^+(\psi)\dd\psi=-Q.
\end{aligned}
\right.
\end{equation}
Set $A^\infty(\te)=\beta A^+(\beta\te)$. Then $A^\infty$ solves the problem (\ref{INBE}). Based on the solution constructed in Section \ref{infty behavior}, we will show that the non-positive solution to (\ref{Ape-in-app}) is unique.
\begin{lemma}If $\eq$ is small enough and $\tilde{A}^+$ is a smooth non-positive solution to (\ref{Ape-in-app}), then $\tilde{A}^+=A^+$. 
\end{lemma}
\begin{proof}
Since $A^+$ solves (\ref{Ape-in-app}), then
$$
\beta \big(A^+\big)^2+\e A^+_{\psi\psi}+4\beta\e A^+\equiv c
$$
is a constant. Similarly,
$$
\beta \big(\tilde{A}^+\big)^2+\e \tilde{A}^+_{\psi\psi}+4\beta\e \tilde{A}^+\equiv \tilde{c}.
$$
Since $A^+(-1)=0$, one gets $c=\e A^+_{\psi\psi}(-1)$. It follows from Proposition \ref{inftyest} that 
$$
c=\e A^+_{\psi\psi}(-1)\geqslant Q^2 A^+_{p\eta\eta}|_{\eta=0}-C\e Q\qe=\frac{Q^2\beta}{4}-CQ^2\eq>0,
$$
where (\ref{PEI}) has been used. Set $W=(A^+-\tilde{A}^+)\rm{sgn}(c-\tilde{c})$. Then 
$$
\beta\big(A^++\tilde{A}^+\big)W+\e W_{\psi\psi}+4\beta\e W\geqslant0,
$$
and $\int_{-1}^1 W\dd\psi=0$. Let $R=\frac{W}{A^+-\delta}$, $\delta>0$. Then $R$ satisfies 
$$
\Big(\beta\big(A^++\tilde{A}^+\big)\big(A^+-\delta\big)+\e A^+_{\psi\psi}+4\beta\e \big(A^+-\delta\big)\Big)R+\e \big(A^+-\delta\big) R_{\psi\psi}+2\e A^+_{\psi}R_\psi\geqslant0.
$$
Since $A^+$ and $\tilde{A}^+$ are non-positive, the coefficient of $R$ in the inequality above bounded below as
\begin{align}\nonumber
&\beta\big(A^++\tilde{A}^+\big)\big(A^+-\delta\big)+\e A^+_{\psi\psi}+4\beta\e \big(A^+-\delta\big)\\\nonumber
=&\beta\tilde{A}^+\big(A^+-\delta\big)-\beta A^+\delta+\beta \big(A^+\big)^2+\e A^+_{\psi\psi}+4\beta\e A^+-4\beta\e\delta\\\nonumber
\geqslant& c-4\beta\e\delta>0,
\end{align}
provided that $\delta$ is small enough. Applying the maximum principle yields $R\geqslant0$, thus $W\leqslant0$. $\int_{-1}^1 W\dd\psi=0$ implies $W=0$, namely $\tilde{A}^+=A^+$.
\end{proof}

\noindent \textbf{Acknowledgements:} 
This research is partially supported by Zheng Ge Ru Foundation, Hong Kong RGC Earmarked Research Grants CUHK-14301421, CUHK-14300917, CUHK-14302819, and CUHK-14300819, and the key project of NSFC No. 12131010. C.Gao is deeply grateful to Prof. Liqun Zhang, Prof. Linqiu Li and Dr. Chuankai Zhao for very valuable discussion.

\bibliographystyle{springer}
\bibliography{mrabbrev,literatur}
\newcommand{\noopsort}[1]{} \newcommand{\printfirst}[2]{#1}
\newcommand{\singleletter}[1]{#1} \newcommand{\switchargs}[2]{#2#1}

\end{document}